\documentclass[11pt]{aims}
\usepackage{amsmath, amssymb, mathrsfs}
  \usepackage{paralist}
  \usepackage{graphics} 
  \usepackage{epsfig} 
\usepackage{graphicx}
  \usepackage{epstopdf}
 \usepackage[colorlinks=true]{hyperref}
 \hypersetup{urlcolor=blue, citecolor=red}
 \usepackage[toc,page]{appendix}
\usepackage{chngcntr}
\usepackage{tikz-cd}
\usepackage{dsfont}
\usepackage{boondox-cal}
\usepackage[nolabel]{showlabels}
\usepackage{titlesec}
\titleformat{\section}{\Large\bfseries}{\thesection.}{4pt}{}
\titleformat{\subsection}{\large\bfseries}{\thesection.\arabic{subsection}.}{4pt}{}
\titleformat{\subsubsection}{\bfseries}{\thesection.\arabic{subsection}.\arabic{subsubsection}.}{4pt}{}
\titleformat*{\paragraph}{\bfseries}
\titleformat*{\subparagraph}{\bfseries}
 \usepackage[margin=1.1in]{geometry}

  \def\currentyear{}
   \def\currentmonth{}
    \def\ppages{}
     \def\DOI{}

\def\RR{\mathbb{R}}

\newcommand{\be}{\begin{equation}}
\newcommand{\ee}{\end{equation}}

\newcommand{\pa}{\partial} 
 
\newcommand{\lab}{\label} 
\newcommand{\ba}{\begin{array}}
\newcommand{\ea}{\end{array}}
\newcommand{\bee}{\begin{eqnarray*}}
\newcommand{\eee}{\end{eqnarray*}}
\newcommand{\bea}{\begin{eqnarray}}
\newcommand{\eea}{\end{eqnarray}}
\newcommand{\non}{\nonumber}

\def\fref#1{{\rm (\ref{#1})}}
\newcommand{\donothing}[1]{{}}

\newcommand{\pd}{\partial}

\newtheorem{propositionmain}{Proposition}
\newtheorem{theoremmain}{Theorem}
\newtheorem{theorem}{Theorem}[section]

\newtheorem{definition}[theorem]{Definition}

\newtheorem{lemma}[theorem]{Lemma}

\newtheorem{remark}[theorem]{Remark}

\numberwithin{equation}{section}

\def\pa{\partial}

\def\RR{\mathbb{R}}

\def\fref#1{{\rm (\ref{#1})}}

\catcode`@=11
\def\supess{\mathop{\operator@font Sup\,ess}}
\catcode`@=12

\def\RR{\mathbb{R}}

\def\bar#1{{\overline #1}}
\def\fref#1{{\rm (\ref{#1})}}

\def\R2+{\RR ^2_+}

\def\pa{\partial}

\def\lim{\mathop{\rm lim}}

\def\sup{\mathop{\rm sup}}

\def\log{{\rm log}}

\def\T{\Theta}

\def\tt{\tilde{T}}

\def\pa{\partial}

\def\tb{\tilde{b}}

\def\pa{\partial}

\def\L{\mathcal L}

\def\ba{\begin{array}}
\def\ea{\end{array}}
\def\lab{\label}
\def\non{\nonumber}
\def\mX{\mathcal X}
\def\mY{\mathcal Y}
\def\tb{\tilde b}
\def\ta{\tilde a}
\def\T{\Theta}

\def\sft{\mathsf{t}}
\def\sfX{\mathsf{X}}
\def\sfY{\mathsf{Y}}
\def\mH{\mathcal{H}}
\def\sfv{\mathsf{v}}

\def\ma{\mathcal a}
\def\mb{\mathcal b}

\def\tmb{\tilde{\mathcal b}}

\begin{document}
\title{Singularities and unsteady separation for the inviscid two-dimensional Prandtl system}

\author[C. Collot]{Charles Collot}
\address{Courant Institute of Mathematical Sciences, New York University, 251 Mercer Street, New York, 10003 NY, United States of America.}
\email{cc5786@nyu.edu}
\author[T.-E. Ghoul]{Tej-Eddine Ghoul}
\address{NYUAD Research Institute, New York University in Abu Dhabi, Saadiyat Island, P.O. Box 129188, Abu Dhabi, United Arab Emirates.
Department of Mathematics, New York University in Abu Dhabi, Saadiyat Island, P.O. Box 129188, Abu Dhabi, United Arab Emirates.}
\email{teg6@nyu.edu}
\author[N. Masmoudi]{Nader Masmoudi}
\address{ NYUAD Research Institute, New York University Abu Dhabi, PO Box 129188, Abu Dhabi, United Arab Emirates.
 Courant Institute of Mathematical Sciences, New York University, 251 Mercer Street, New York, NY 10012, USA,}
\email{masmoudi@cims.nyu.edu}

\maketitle

\begin{abstract} 
We consider the inviscid unsteady Prandtl system in two dimensions, motivated by the fact that it should model to leading order separation and singularity formation for the original viscous system. We give a sharp expression for the maximal time of existence of regular solutions, showing that singularities only happen at the boundary or on the set of zero vorticity, and that they correspond to boundary layer separation. We then exhibit new Lagrangian formulae for backward self-similar profiles, and study them also with a different approach that was initiated by Elliott-Smith-Cowley and Cassel-Smith-Walker. One particular profile is at the heart of the so-called Van-Dommelen and Shen singularity, and we prove its generic appearance (that is, for an open and dense set of blow-up solutions) for any prescribed Eulerian outer flow. We comment on the connexion between these results and the full viscous Prandtl system. This paper combines ideas for transport equations, such as Lagrangian coordinates and incompressibility, and for singularity formation, such as self-similarity and renormalisation, in a novel manner, and designs a new way to study singularities for quasilinear transport equations.
\end{abstract}


\section{Introduction}

We consider the inviscid Prandtl equations on the upper half-plane $\mH:=\mathbb R \times [0,\infty)$:
\begin{align}\label{2DPrandtlp}
\left\{\begin{array}{ll}
&u_t+uu_x+vu_y=-p_x^E \\
&u_x+v_y=0,\\
&v\arrowvert_{y=0}=0, \quad \lim_{y\to \infty} u(t,x,y)=u^E(t,x),
\end{array}
 \right. \qquad (t,x,y)\in[0,T)\times \mH,
\end{align}
where $p^E$ and $u^E$ are the trace of the Eulerian pressure and tangential flow at the boundary $\mathbb R \times \{ 0\}$ induced by the Eulerian flow at infinity. The functions $p^E$ and $u^E$ are prescribed, and then act as forcing terms for $u$.  They are linked through Bernoulli's equation:
\be \lab{eq:bernouilli}
u^E_t+u^Eu^E_x=-p^E_x,
\ee
whose solutions have to be global in two dimensions.

\subsection{Historic background}

\subsubsection{Boundary layer separation and singularities} \label{subsubsec:prandtl}

Prandtl's system comes from the vanishing viscosity limit of the Navier-Stokes equations with Dirichlet boundary condition. It describes the formation of a boundary layer, where the solution does a sharp transition from the vanishing at the boundary of an obstacle, induced by the Dirichlet condition, to a solution of the Euler system away from its boundary.

Prandtl describes in \cite{prandtl04} the phenomenon of boundary layer separation, when "a fluid layer, set rotating as a result of friction at the wall, moves out into the free stream" (this English translation being taken from \cite{ackroyd}). For the steady Prandtl system, separation happens at the boundary with the condition of vanishing wall shear stress $\pa_y u(x_0,0)=0$. Goldstein \cite{goldstein} finds that at such location the solution is singular, and as a result so is the displacement thickness (a quantity describing the influence of the layer on the outer flow), revealing that the layer separates past $x_0$. We refer to the textbook \cite{schlichting} for details. Dalibard and Masmoudi \cite{DM} gave recently a mathematically rigorous description of the Goldstein singularity.

For the unsteady Prandtl system, Moore, Rott and Sears \cite{moore,rott,sears} realised that $\frac{\pa}{\pa y}_{|y=0}u=0$ was not a correct condition for separation, and came up with what is now known as the MRS conditions. These conditions state that if a piece of layer that is separating off the boundary, then there is a point $(x_s(t),y_s(t))$ at the center of this separating piece of layer, at which the vorticity vanishes $\pa_y u(t,x_s(t),y_s(t))=0$, and whose tangential velocity $u(t,x_s(t),y_s(t))$ is equal to the tangential velocity of the whole separating piece of layer. We recall that $\pa_y u$ is the vorticity, when considering the approximation of the Navier-Stokes equations with the Prandtl system at high Reynolds number. It was then believed \cite{sears2} that a criterion for separation at time $T_0$ was that at such a point the solution $u$ becomes singular as $t\uparrow T_0$. Van Dommelen and Shen \cite{VanShen80,van2} showed this equivalence between separation and singularity formation. We refer to the introduction of \cite{casselconlisk} for an historic perspective for criteria for steady and unsteady separation.

It is noted in  \cite{VanShen80} that the layer becomes inviscid to leading order during separation. This justifies our study of the inviscid Prandtl system \fref{2DPrandtlp}. The aim of this article is to give a rather complete study of \fref{2DPrandtlp}. In particular, we justify in a mathematical and rigorous way the aforementioned results: we prove the equivalence between separation and singularities for \fref{2DPrandtlp}, and describe the generic mechanism for this phenomenon.

This introduction will now be restricted to unsteady problems.

\subsubsection{Existence and regularity of solutions}

The first rigorous mathematical justification of the unsteady Prandtl system is due to Oleinik \cite{Ol66}. She proves the local wellposedness by imposing a monotonicity in $y$ condition on the tangential velocity $u$ in order to use the Crocco transform. Xin and Zhang \cite{XinZhang04} obtain global existence of weak solutions by imposing monotonicity and an extra condition on the pressure.
The monotonicity condition allows Masmoudi and Wong in \cite{MasWong15}, and Alexandre, Wang, Xu and Yang in \cite{AWXY} to prove wellposedness in Sobolev regularity.
Without the monotocity condition, the equation can be ill-posed in Sobolev regularity \cite{GVDormy10}. Indeed, the authors in \cite{grenier2000,Grenier2016} constructed instabilities that can prevent the Prandtl system to be a good approximation of the Navier-Stokes system in the vanishing viscosity limit. We refer to \cite{CafSam98I,GMM,maekawa2014inviscid} and the references therein for further informations. Otherwise, in the general case the system is locally well-posed in the analytical setting  \cite{CafSam98II,lombardo2003well,KV13,dietert2018well}.

Up to our knowledge, the only local existence with general smooth initial data for the inviscid Prandlt system \fref{2DPrandtlp} is due to Hong and Hunter, in \cite{HongHunter03} (see also \cite{brenier1999homogeneous,grenier1999derivation}). They find a lower bound for the maximal time of existence which corresponds to that of the Burger's equation $T=(-\inf_{\mH}\pd_xu_0)^{-1}$ in the case of a trivial outer Eulerian flow \fref{2DPrandtlphom}. We find in this current paper that the sharp maximal existence time for the homogeneous inviscid Prandtl system \fref{2DPrandtlphom} is $T=-(\inf_{\{\pd_y u_0=0 \}\cup \{y=0 \}}\pd_x u_0)^{-1}$ which is larger. We prove that this time is sharp which clarifies why the aforementioned monotonicity condition is important to guarantee global wellposedeness for the inviscid Prandtl system. We obtain more generally a sharp maximal time of existence in the case of nontrivial Eulerian flows.

\subsubsection{Description of singularities}

The first reliable numerical result on the unsteady Prandtl system explaining how the separation is linked to the formation of singularity was obtained by Van Dommelen and Shen \cite{VanShen80}. They characterise the singularity as a result of particles being squashed in the streamwise direction, with a compensating expansion in the normal direction of the boundary. We refer to \cite{gargano2009singularity,della2006singularity} for recent numerical simulations and references therein for previous ones.

The first rigorous result on singularity formation for the unsteady Prandtl system is due to E and Enquist \cite{EE97} (see \cite{kukavica2017van} for nontrivial outer flows). They consider the trace of the tangential derivative of odd solutions in $x$ along the transversal axis $\xi(t,y)=-u_x(t,0,y)$, which obeys the following equation for $y\in [0,\infty)$:
\begin{align}\label{1DPrandtl}
\left\{\begin{array}{ll}
\xi_t-\xi_{yy}-\xi^2+\left(\int_0^y \xi \right)\xi_y=p^E_{xx}(t,0),\\
\xi(t,0)=0, \ \ \xi(0,y)=\xi_0(y), \ \ \lim_{y\rightarrow \infty}\xi(t,y)=-u^E_x(t,0).
\end{array}
 \right.
\end{align}
They prove singularity formation for \fref{1DPrandtl}. In \cite{CGIM} we give a precise description of the singular dynamics of the equation above, where we find a stable profile (and instable ones) and prove that the blowup point is ejected to infinity in the transversal direction because of the incompressibility condition. This result can be interpreted as a partial stability result for one of the profiles studied in the present paper, see Remark \ref{rk:degen}.\\

\noindent It is interesting to understand singularities for simplified models. In a first part of \cite{CGM}, we treat the inviscid Burgers equation $u_t+uu_x=0$ and prove that the Taylor expansion of the initial data around the blowup point will decide which profile and scaling law the flow will select to form the singularity. The solution will be of the form $u(t,x)=(T-t)^{1/(2i)} \Psi_i(x/(T-t)^{1+1/(2i)})$ where $(\Psi_i)_{i\geq 1}$ are the profile, and the Taylor expansion selects the integer $i$.
We prove that the generic profile
 \be \lab{eq:def Psi1}
\Psi_1(\mX):= \left(-\frac \mX2 +\left(\frac{1}{27}+\frac{\mX^2}{4}\right)^{\frac 12} \right)^{\frac 13}+ \left(-\frac \mX2 -\left(\frac{1}{27}+\frac{\mX^2}{4}\right)^{\frac 12} \right)^{\frac 13},
\ee
appears generically during blow-up. Surprisingly, the above profile is also going to play a role in the generic separation phenomenon for the Prandlt's system.

 In order to understand the effect of the transversal viscosity on the horizontal transport we consider in the second part of \cite{CGM} a two dimensional Burgers system:
\be \label{2DBurger}
u_t- u_{yy}+uu_x=0\quad\quad (t,x,y)\in[0,T)\times\mathbb R^2 .
\ee
We found infinitely many different profiles, one being stable under suitable perturbations. We find that despite the infinite speed of propagation induced by the transverse viscosity, the Taylor expansion of the initial datum around the blowup point will still decide the profile and the scaling law.
We prove that the vertical viscosity affects the shock formation of Burgers equation, in the sense that the solutions are now anisotropic and of the form $u(x,y,t)\sim\lambda^{1/(1+2i)}(t,y)\Psi_i(x/\lambda (t,y))$ where $\Psi_i$ is a profile of the Burger's equation and $\lambda\rightarrow 0$ depends on the solutions of a parabolic system similar to \eqref{1DPrandtl} without the nonlocal term.

Inspired by \cite{van1990lagrangian,VanShen80,ESC83,CSW96,cowley1990use} where it is suggested based on numerics and formal calculations that the viscosity is asymptotically negligible during singularity formation for the unsteady Prandtl system, we treat in this paper the singularity formation for the inviscid problem \fref{2DPrandtlp}. In a forthcoming paper we treat the viscous case.\\

\subsection{Results}

\subsubsection{A sharp existence result and equivalence between separation and singularity} \label{subsec:lwpth1}

We give here a sharp local well-posedness result for \fref{2DPrandtlp}, relying on the Lagrangian approach initiated in \cite{VanShen80,HongHunter03}. As a by-product, we prove that the layer separates if and only if the solution becomes singular, establishing rigorously this standard criterion for separation in the physics literature (see for example \cite{sears2}).

We denote by $(X,Y)$ the Lagrangian variables and $(x,y)$ the Eulerian ones for equation \fref{2DPrandtlp}. The position at time $t$ of a particle with initial position $(X,Y)$ and that is transported by the flow, is denoted by $(x(t,X,Y),y(t,X,Y))$. Using dots for differentiation with time $t$, they are related by the characteristics ODE:
\be \label{id:charac}
\dot x =u(t,x,y), \ \ \dot y=v(t,x,y), \ \ (x(0),y(0))=(X,Y).
\ee
Our definition of separation, meaning that the layer penetrates the outer flow (following Prandtl as cited in Subsubsection \ref{subsubsec:prandtl}), is:
\begin{definition} \label{def:separation}
Given $T_0>0$, $u^E,p^E\in C^1([0,T_0)\times \mathbb R)$ solving \fref{eq:bernouilli} and $u\in C^1([0,T_0)\times \mH)$ solving \fref{2DPrandtlp}, we say that there is boundary layer separation at time $T_0$ if there exists a solution $(x(t),y(t))$ of \fref{id:charac} such that $\lim_{t\uparrow T_0}y(t)=\infty$.
\end{definition}

From \fref{2DPrandtlp}, along the characteristics $u$ solves the ODE $\dot u=-p^E_x(t,x)$. Therefore, the tangential position $x$ of the particle can be retrieved without any knowledge about the normal one $y$, by solving the following ODE for each triple $(X,Y,u_0(X,Y))$:
\be \lab{eq:charx}
\left\{  \begin{array}{l l}\dot x =u, \\ \dot u =-p^E_x (t,x), \end{array} \right. \qquad (x(0),u(0))=(X,u_0(X,Y)).
\ee
The above equation is that of one dimensional particle moving in a force field $-p^E_x$. The corresponding change of variables $(X,Y)\mapsto (X,u)$ is called the Crocco transform and has been used extensively in the study of the Prandtl system. One key fact about \fref{2DPrandtlp} is that the vorticity $u_y$ is preserved along the characteristics, as differentiating \fref{2DPrandtlp} with respect to $y$ yields $ u_{yt}+uu_{yx}+vu_{yy}=0$
because $p^E$ does not depend on $y$, and using incompressibility. Hence the set $\{u_y=0 \}$ is transported by the characteristics. The boundary $\{y=0\}$ is also preserved as $v_{|y=0}=0$. Hence for any $(x,y)$ either in the set $\{u_y=0\}$ or at the boundary $\{y=0\}$, one has $v_xu_y=0$, so that differentiating \fref{2DPrandtlp} with respect to $x$ yields $ u_{xt}+u u_{xx}+vu_{xy}=-(u_x)^2-p_{xx}^E$. It follows from this equation and\fref{eq:charx} that the transport along the tangential variable and the tangential compression $u_x$, when restricted to these two sets, are given by the previous ODE completed by an inhomogeneous Riccati equation:
\be \lab{eq:ricattizerovorticity}
\left\{  \begin{array}{l l}\dot x =u, \\ \dot u =-p^E_x (x), \\ \dot u_x=-(u_x)^2-p^E_{xx}(x), \end{array} \right. \qquad (x,u,u_x)(0)=(X,u_0(X,Y),u_{0x}(X,Y)).
\ee
Given a global in time pressure field $p^E\in C^k([0,\infty)\times \mathbb R)$ with $k\geq 2$, the solution to the above system might not exist for all time due to the nonlinearity in the last equation, and we denote by $T(X,Y)$ the corresponding maximal time of existence. We will distinguish later on between singularities happening at the boundary or away from it, and define to this aim:
\be \label{def:T1}
T:= \min (T_a,T_b ) , \ \ T_a:= \min \{T(X,Y), \ \ \pa_Y u_{0}(X,Y)=0, \ Y>0 \} , \ \ T_b:= \min \{T(X,Y), \  \ Y=0 \} .
\ee
The time $T$ defined above is a natural upper bound for the maximal existence of a solution to \fref{2DPrandtlp} with $u_x\in L^{\infty}$. In fact, this time is sharp. We introduce the spaces

\begin{align}
\non &L^{\infty}_{\text{loc},0}([0,\infty)\times \mathbb R)=\left\{ f, \ \| f \|_{L^{\infty}([0,t]\times \mathbb R)}<\infty \mbox{ and } \lim_{m\rightarrow \infty }\| f \|_{L^{\infty}([0,t]\times \{|x|\geq m\})}=0 \mbox{ for all } t>0\right\}, \\
\lab{def:Fk} &\mathcal F^k=\left\{ (u^E,p^E_x) \in C^k ([0,\infty)\times \mathbb R) \mbox{ solving } \fref{eq:bernouilli}, \mbox{ with } (u^E_x,p^E_{xx}) \in L^{\infty}_{\text{loc},0}([0,\infty)\times \mathbb R) \right\},
\end{align}
for the outer flow, and the following spaces for the initial datum (depending on $u^E(0)$):
\begin{align}
\non &L^{\infty}_{0}(\mH) = \left\{ f, \ \| f \|_{L^{\infty}(\mH )}<\infty \mbox{ and } \lim_{m\rightarrow \infty }\| f \|_{L^{\infty}( \{|x|\geq m\}\times [0,\infty))}=0\right\},\\
\lab{def:Ek} &\mathcal E^k=\Bigl\{ u_0\in C^k(\mH) \mbox{ with } \pa_y u_0 \in L^{\infty}(\mH), \ \pa_x u_0\in L^{\infty}_{0}(\mH) \\
\non &\qquad \qquad  \qquad \qquad  \qquad \qquad  \mbox{ and } \lim_{y\rightarrow \infty} \| u_0(\cdot ,y)-u^E(0,\cdot)\|_{C^1(\{ |x|\leq m\})}, \ \mbox{ for all } m>0\Bigr\}.
\end{align}

\begin{theoremmain} \lab{th:main}

For any $(u^E,p^E_x)\in \mathcal F^2$ and $u_0 \in \mathcal E^2$, there exists a unique solution $u\in C^1([0,T)\times \mH)$ of \fref{2DPrandtlp} where $T$ is defined by \fref{def:T1}. It satisfies $\sup_{t\in [0,\tilde T]}\| \nabla u(t)\|_{L^{\infty}(\mH)}<\infty$ for any $\tilde T<T$. If $T$ is finite then:
$$
\lim_{t\uparrow T} \| u_x(t) \|_{L^{\infty}(\mH)}=\infty
$$
and if moreover $T=T_a$ then there is boundary layer separation at time $T$ in the sense of Definition \ref{def:separation}. If in addition $u_0\in C^{k}(\mH)$ and $(u^E,p^E_x)\in C^k([0,\infty)\times \mathbb R)$ for some $k\geq 3$, then $u\in C^{k-1}([0,T)\times \mH)$. The mapping which to $u_0$ assigns the solution $u$ is strongly continuous from $C^k(\mH)\cap \mathcal E^2$ into $C^{k-1}([0,T']\times \mH)$ for any $T'<T$.

\end{theoremmain}

\begin{remark} \label{rk:applications}

Boundedness of the gradient in \fref{def:Fk} and \fref{def:Ek} is the natural requirement for classical solutions. The decay as $|x|\rightarrow \infty$ facilitates the study of singularities, forcing them to appear at a finite location in $x$, and not at infinity in $x$. We use the word singularity as it is usual, but it can be misleading: the solution might remain smooth at time $T$. Indeed, the points where $u_x$ becomes large can be sent to infinity in the normal direction as the study below shows. Also, from Theorem \ref{th:main} and the ODE \fref{eq:ricattizerovorticity} one easily derives criterions for global-well posedness or finite time blow-up:
\begin{itemize}
\item In the case of normal monotonicity $u_y>0$ or $u_y<0$, the solution is global if and only if the solution to the Burgers equation at the boundary 
$$
\pa_t u_{|y=0}+u_{|y=0}\pa_x u_{|y=0}=-p^E_x
$$
is global.
\item The solution is global in the case of tangential growth $u_{0x}\geq 0$ on both the set of zero vorticity $\{u_{0y}=0 \}$ and the boundary $\{y=0\}$, and of concave pressure $p^E_{xx}\leq 0$.
\item The maximal time of existence is $T=(-\min_{\{u_{0y}=0 \}\cup \{y=0 \}}u_{0x})^{-1}$ in the pressureless case $p^E=0$ (with the convention $T=\infty$ if the $\min$ is nonnegative).
\end{itemize}

\end{remark}

\begin{remark}
For a singularity away from the wall, $T_a<T_b$, let $(X_0,Y_0)$ be the point attaining the minimum in the definition of $T_a$. Then the time dependent position $(x(t),y(t))$ in Definition \ref{def:separation} can be chosen as the characteristics starting initially at $(X_0,Y_0)$. One thus recovers\footnote{The MRS conditions are satisfied because of Definition \ref{def:separation} and because $\pa_y u_0(X_0,Y_0)=0$ from \fref{def:T1} so that $\pa_y u(t,x(t),y(t))=0$ by conservation of vorticity. There is a Burgers type compression because $\pa_x u(t,x(t),y(t))\rightarrow -\infty$ as $t\rightarrow T_a$ since the third equation in \fref{eq:ricattizerovorticity} becomes singular from the definition of $T_a$.} a standard criterion for separation in the physics literature: the MRS conditions of Subsubsection \ref{subsubsec:prandtl} at this point $(x(t),y(t))$ supplemented by the appearance of a Burgers type compression at this point (this supplementary condition was found by Van Dommelen and Shen \cite{VanShen80}).
\end{remark}

\subsubsection{The generic self-similar singular solution} \label{subsec:thetafundamental}

Considering the nonlinearity in the ODE \fref{eq:ricattizerovorticity}, Theorem \ref{th:main} and a standard convexity argument shows that given any prescribed Eulerian flow $u^E$ and $p^E$, there exist initial data $u_0$ such that the corresponding solution becomes singular in finite time. Theorem \ref{th:main} also indicates where singularities of \fref{2DPrandtlp} form: either at the boundary $\{y=0 \}$, or away from the boundary on the set of zero vorticity $\{y>0, \ u_y=0 \}$. We now focus on the description of this phenomenon. Motivated by the full original viscous Prandtl system, where the Dirichlet boundary condition forces $u_{|y=0}= 0$, we will only study what happens when the singularity forms away from the boundary, i.e. $T_a<T_b$ and $T_a<\infty$. The study done in this document could be adapted to the case of a blow-up at the boundary.

We are interested first in a leading order description of the singularity. Since as $\nabla u$ becomes large, the effects of the pressure $p^E$ and of the boundary conditions become of lower order, we start by dropping them and investigate the homogeneous inviscid Prandtl system
\begin{align}\label{2DPrandtlphom}
\left\{\begin{array}{ll}
&u_t+uu_x+vu_y=0,\\
&u_x+v_y=0, \ \ v\arrowvert_{y=0}=0,
\end{array}
 \right. \qquad (t,x,y)\in[0,T)\times\mH.
\end{align}
This equation has the following invariances. If $u$ is a solution then so is
\be \label{def:invariancesprandtl}
 \frac{\mu}{\lambda} \iota u\left(\frac{t}{\lambda},\iota \frac{x-ct}{\mu},\frac{y}{\nu} \right)+c
\ee
for $(\iota,\lambda,\mu,\nu,c)\in \{-1,1\}\times (0,\infty)^3\times \mathbb R$. Backward self-similar solutions are special solutions living in the orbit of the initial datum under the action of a one dimensional scaling subgroup, of the form $u(t,x,y)=(T-t)^{\alpha-1} \T(x/(T-t)^{\alpha},y/(T-t)^\beta)$. We develop in this paper, in Section \ref{sec:selfsim}, a method to find the admissible exponents $\alpha$ and $\beta$, and an explicit formula for $\T$. To have a solution to \fref{2DPrandtlphom} of this form is equivalent to have a solution of the stationary equation:
\begin{equation} \label{eq:selfsimstationary}
\left\{ \begin{array}{l l}
(1-\alpha)\Theta+(\alpha \mX+\T)\pa_\mX \T +(\beta \mY +\Upsilon )\pa_\mY \T=0,\\
\pa_\mX \T+\pa_\mY \Upsilon=0,
\end{array}
\right. \qquad (\mX,\mY)\in \Omega,
\end{equation}
subject to the condition $\lim_{\mY \rightarrow 0} \Upsilon=0$, on an open set $\Omega \subset \mH$ (which might be different than $\mH$, see below). The above equation \fref{eq:selfsimstationary} is nonlinear and nonlocal. Cassel, Smith and Walker \cite{CSW96} (see also \cite{ESC83}) give a change of variables to transform \fref{eq:selfsimstationary} into a nonlinear local equation, see (iii) in Lemma \ref{lem:resolutioneqprofiles}, proving the existence of solutions to \fref{eq:selfsimstationary} but without providing explicit formulas. We find here a new way to solve \fref{eq:selfsimstationary} and show, see (ii) in Lemma \ref{lem:resolutioneqprofiles}, that \fref{eq:selfsimstationary} is equivalent to the following linear and local equation:
\be \label{eq:autosimstattransformee}
(\alpha-1)\ma\pa_\ma \mX+(1+\beta)\mb\pa_\mb \mX=\alpha \mX-\ma, \qquad (\ma,\mb)\in \Omega',
\ee
where $\mX$ is seen as a function of the variables $(\ma,\mb)$ given by a volume preserving change of variables $(\mX,\mY)\mapsto (\ma,\mb)$ with $\ma=-\T$. A Lagrangian interpretation of this change of variables is given in Remark \ref{re:lagrangian interpretation}.

We obtain two important explicit solutions of Equation \fref{eq:autosimstattransformee}, with explicit changes of variables $(\ma,\mb)\mapsto (\mX,\mY)$. Although these changes of variables appear from a different perspective in \cite{van2,van1990lagrangian}, it seems that the authors did not link them to solutions of \fref{eq:selfsimstationary}. Classifying analytic solutions to \fref{eq:selfsimstationary2}, relying on \fref{eq:autosimstattransformee}, is an interesting open problem. The first solution we obtain, the generic self-similar profile, corresponds to $\alpha=3/2$, $\beta=-1/4$ and is the one related to the so-called Van-Dommelen and Shen singularity \cite{VanShen80}. The terminology "generic" is justified in the next Subsection. We introduce ($\Gamma$ being the Gamma function)
\be \lab{def:pstar}
p^*:=\frac{4}{9\pi^3}\Gamma\Big(\frac 14\Big)^4.
\ee
We describe below the generic profile $\Theta$, obtained from an explicit solution $\mX$ to \fref{eq:autosimstattransformee} and an explicit change of variables $\Phi:(\ma,\mb)\mapsto (\mX,\mY)$ relating \fref{eq:autosimstattransformee} to \fref{eq:selfsimstationary}.

\begin{propositionmain}[Generic self-similar profile] \lab{pr:selfsimfunda}
The mapping $\Phi:(\ma,\mb)\mapsto (\mX,\mY)$ given by:
\be \lab{def:Phi}
\Phi (\ma,\mb)=\left(\ma+\mb^2+p^{*2}\ma^3 \ , \ \int_{-\infty}^\mb \frac{d\tmb}{1+3\Psi_1^2 \left(p^*\left(\ma+p^{*2}\ma^3+\mb^2-\tmb^2\right)\right)} \right),
\ee
where $\Psi_1$ and $p^*$ are defined by \fref{eq:def Psi1} and \fref{def:pstar}, satisfies the following properties:
\begin{itemize}
\item[(i)] It is an analytic volume preserving diffeomorphism between $\mathbb R^2$ and the subset of the upper half plane $\{ (\mX,\mY)\in \mH, \ 0< \mathcal Y< 2 \mY^*(\mX) \}$ where:
\be \lab{eq:def Y*}
\mY^*(\mX)=\int_{0}^{\infty}\frac{d\tmb}{1+3\Psi_1^2(p^*(\mX-\tmb^2))}.
\ee
The above curve $ \mY^*$ is analytic and satisfies:
\be \label{exp Y 0}
\mY^*(\mX) = \frac{3\pi}{8}+\mX+\frac{c_2}{2}\mX^2+O(|\mX|^3),\qquad \text{as} \ \mX\rightarrow  0,
\ee
\be \label{exp Y infty}
\mY^*(\mX)=C_{\pm} |\mX|^{-\frac 16}+O(|\mX|^{-\frac 56}) \qquad \text{as} \ \mX\rightarrow \pm \infty,
\ee
where $c_2=-\frac{5\Gamma (1/4)}{72\pi^5}$, and $C_{\pm}=2^{-1}p^{*-2/3}\int_{\mp 1}^\infty (z^3 \pm 1)^{-1/2}dz$ (see \fref{id:cpm} for a formula).
\item[(ii)] The opposite of the first component of its inverse $\Phi^{-1}=(\Phi^{-1}_1,\Phi^{-1}_2)$:
\be \lab{def:thetageom}
\Theta:=-\Phi_1^{-1}:(\mX,\mY)\mapsto -\ma,
\ee
is a self-similar profile, that is, the following is a solution of \fref{2DPrandtlphom} on its support:
\be \lab{def:selfsimtheta}
u(t,x,y)=(T-t)^{\frac{1}{2}}\Theta \left(\frac{x}{(T-t)^{\frac 32}},\frac{y}{(T-t)^{-\frac 14}} \right).
\ee
\item[(iii)] It enjoys the symmetry $\Theta(\mX,\mY^*(\mX)+\mY)=\Theta(\mX,\mY^*(\mX)-\mY)$ for $|\mY|<\mY^*(\mX)$.
\item[(iv)] The set of zero vorticity $\{\Theta_\mY=0\}$ is the curve $\{\mY=\mY^*(\mX)\}$, where $\Theta (\mX,\mY^*(\mX))=\Psi_1(p^*\mX)/p^*$. On this curve, $\pa_\mX\Theta$ is minimal at $(0,3\pi/8)$ with, as $(\mX,\mY)\rightarrow 0$:
\be \lab{taylorthetalag}
\T\left(\mX,\frac{3\pi}{8}+\mY\right)=-\mX+\left(\mX-\mY\right)^2+\left(p^{*2}+c_2\right)\mX^3-c_2\mX^2\mY+O(\mX^4+\mY^4).
\ee
\item[(v)] $\Theta$ has the following behaviour near the boundary of its domain:
\begin{align}\label{expanboundary}
&\Theta(\mX,\mY)=p^{*-2}\mY^{-2}+O(|\mX|\mY^4+\mY^2) \qquad \qquad \mbox {for } 0<\mY \leq \mY^*(\mX)\\
\label{expanboundary12} &\Theta(\mX,2\mY^*(\mX)-\mY)=p^{*-2}\mY^{-2}+O(|\mX|\mY^4+\mY^2)  \qquad \qquad \mbox {for } 0<\mY \leq \mY^*(\mX),
\end{align}
and the following behaviour at infinity\footnote{The condition $\mY<(2-\epsilon)\mY^*$ ensures both quantities $\T(\mX,\mY)$ and $\varphi_{\pm \infty} (\mY |\mX|^{1/6})$ are defined. Note that on $[(2-\epsilon)\mY^*,2\mY^*)$ the asymptotics \fref{expanboundary12} prevails.}: for any $\epsilon>0$, for $0<\mY<(2-\epsilon)\mY^*(\mX)$:
\begin{align}\label{expanboundary2}
&\Theta(\mX,\mY)= |\mX|^{\frac 13} \varphi_{\pm \infty} (\mY |\mX|^{1/6}) +O\left(|\mX|^{-\frac 23}\mY^{-2}+|\mX|^{-1} (2\mY^*(\mX)-\mY)^4\right),
\end{align}
as $\mX\rightarrow \pm \infty$, where $\varphi_\pm \in C^{\infty}((0,2C_\pm),\mathbb R)$ is decreasing on $(0,C_\pm)$, increasing on $(C_\pm,2C_{\pm})$, with $\varphi_\pm (z)\sim p^{*-2}z^{-2}$ and $\varphi_\pm (2C_\pm-z)\sim p^{*-2}z^{-2}$ as $z\rightarrow 0$, and $\varphi_\pm(C_\pm)=\mp p^{*-2/3}$.
\end{itemize}
\end{propositionmain}

\begin{remark}
From the invariances \fref{def:invariancesprandtl} of the equation \fref{2DPrandtlphom}, $\Theta$ generates the full family of generic profiles $(\Theta_{\mu,\nu,\iota})_{(\mu,\nu,\iota)\in (0,\infty)^2\times \{-1,1 \}}$ where $\Theta_{\mu,\nu,\iota}(\mX,\mY)= \mu \iota \Theta \left(\iota \mX/\mu, \mY/\nu \right),$ in that $u(t,x,y)=(T-t)^{1/2}\Theta_{\mu,\nu,\iota} \left(x/(T-t)^{3/2},y/(T-t)^{-1/4} \right)$ also solves \fref{2DPrandtlphom}. The profile $\Theta_{\mu,\nu,\iota}$ is obtained from the mapping:
$$
\Phi_{\mu,\nu,\iota} (\ma,\mb)=\left(\ma+\iota\frac{\mu}{\nu^2}\mb^2+\frac{p^{*2}}{\mu^2}\ma^3 \ , \ \int_{-\infty}^{\mb} \frac{d\tmb}{1+3\Psi_1^2 \left(p^*\left(\frac{\ma}{\mu}+\frac{p^{*2}}{\mu^3}\ma^3+ \frac{\iota}{\nu^2}\mb^2-\frac{\iota}{\nu^2} \tmb^2\right)\right)} \right)
$$
via $\Theta_{\mu,\nu,\iota}=-\Phi_{\mu,\nu,\iota,1}^{-1}$ where we wrote $\Phi_{\mu,\nu,\iota}^{-1}=(\Phi_{\mu,\nu,\iota,1}^{-1},\Phi_{\mu,\nu,\iota,2}^{-1})$. Therefore, changing the value \fref{def:pstar} of $p^*$ in \fref{def:Phi} would have defined another self-similar profile as well. The choice of $p^*$ \fref{def:pstar} was made to normalise the Taylor expansion \fref{taylorthetalag}.
\end{remark}

\begin{remark} \label{re:lagrangian interpretation}
The variables $(\ma,\mb)$ in \fref{def:Phi} should be thought, asymptotically as $t\uparrow T$, as Lagrangian self-similar variables. Indeed, for the self-similar solution $u(t,x,y)=(1-t)^{1/2}\T(x/(1-t)^{3/2},y/(1-t)^{-1/4})$ of \fref{2DPrandtlphom} with initial datum $u_0(x,y)=\T(x,y)$, let $(X,Y)\mapsto (x,y)$ be the Lagrangian to Eulerian map defined in Subsubsection \ref{subsec:lwpth1}, let $(x,y)\mapsto (\mX,\mY)=(x/(1-t)^{3/2},y/(1-t)^{-1/4})$ be the Eulerian self-similar renormalisation, and let $(\mX,\mY)\mapsto (\ma,\mb)=\Phi^{-1}(\mX,\mY)$ be as in \fref{def:Phi}. Compositing these maps, we get a map $(X,Y)\mapsto (\ma,\mb)$, allowing to interpret what the $(\ma,\mb)$ variables represent for the Lagrangian variables $(X,Y)$. A computation shows that they are related through: $(X,Y)=\Phi((1-t)^{1/2}\ma,(1-t)^{3/4}\mb)$. Thus, for $(\ma,\mb)=O(1)$, a Taylor expansion, using Proposition \ref{pr:selfsimfunda}, produces:
$$
(X,Y)=\left(0,\frac{3\pi}{8}\right)+(1-t)^{\frac 12}\ma \cdot (1,1)+(1-t)^{\frac 34}\mb \cdot (0,1)+O((1-t)).
$$
Asymptotically as $t\uparrow 1$, the change of variables $(X,Y)\mapsto (\ma,\mb)$ is then a linear rescaling centred at $(0,3\pi/8)$, namely, it is asymptotically a \emph{self-similar renormalisation centred at that point}.
\end{remark}

\subsubsection{Generic singularity leading to boundary layer separation}

The self-similar profile $\Theta$ given by Proposition \ref{pr:selfsimfunda} is at the heart of the singularity formation for the inviscid Prandtl equations, triggering the separation of the boundary layer. In this paper we show the appearance of this self-similar profile from a generic set (dense, and open, among data producing singularities) of initial data $u_0$. This shows that the effects of the outer Eulerian flow $(u^E,p^E)$ might alter the position and time of appearance of the singularity, but not its structure (the profile $\T$ and the self-similar exponents $3/2$ and $-1/4$), and hence are of lower order. The generic singularity is a consequence of a tangential generic Burgers-type compression occurring on a line of zero vorticity, and induces a normal expansion by volume preservation. We justify the following picture, which was first derived in \cite{van2,VanShen80}:

\begin{center}
\includegraphics[width=16cm]{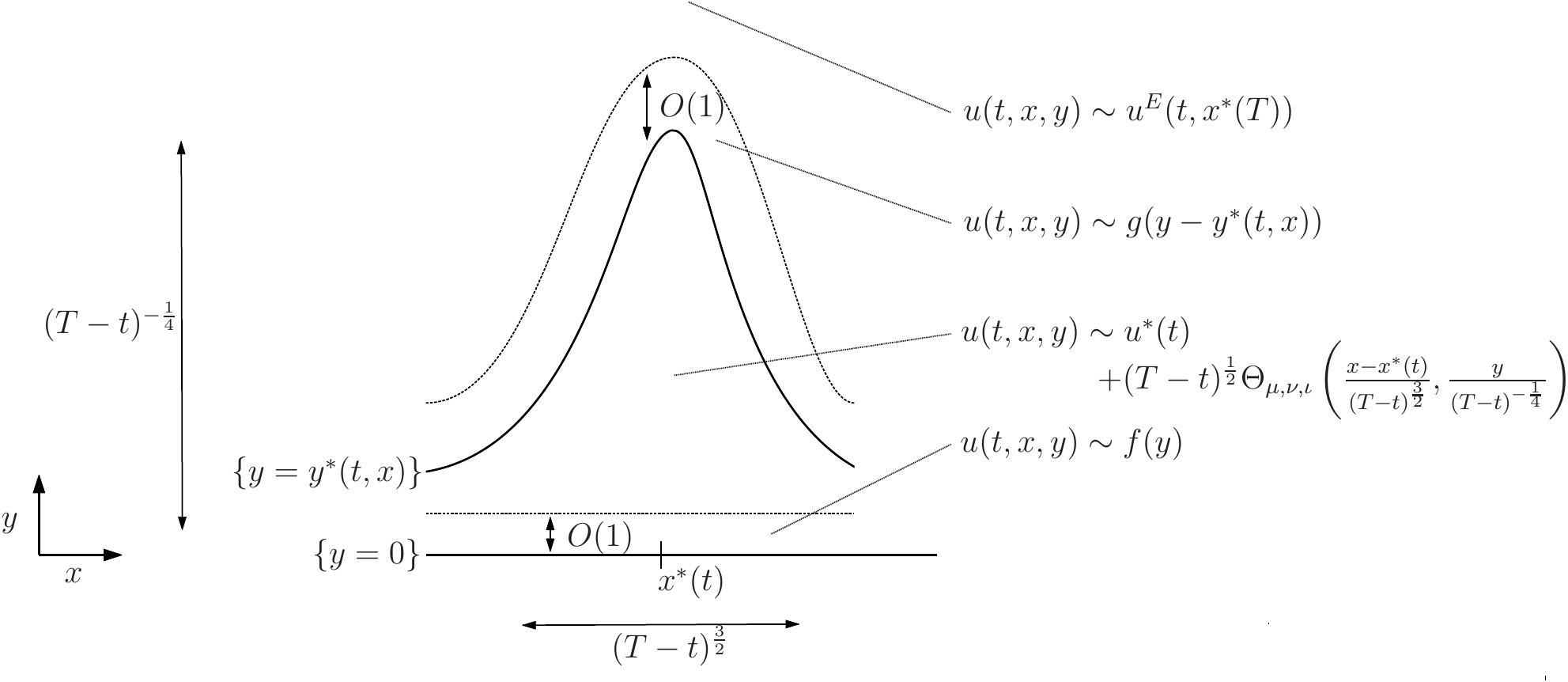}
\end{center}

The idea behind the generic appearance of $\T$ is the following. A shock is forming at a moving location $(x^*,y^*)$ in Eulerian variables. This point corresponds in Lagrangian variables to a point $(X_0,Y_0)$, at which the characteristics map $(X,Y)\mapsto (x,y)$ becomes critical (i.e., the Jacobian vanishes) at blow-up time. Generically, the characteristics map will satisfy certain properties at this location, see Section \ref{sec:characteristics}, ensuring that the following holds true. Around that point, there exist self-similar Lagrangian variables $(a,b)$ (see \fref{def:abrenorm} for a precise definition) such that for appropriate self-similar Eulerian variables $(\mX,\mY)$ the renormalised characteristics map becomes asymptotically time independent as $t\uparrow T$, and converges to $\Phi$:

\begin{center}
\begin{tikzcd}[column sep=180pt,row sep=30pt]
  \mbox{Lagrangian } (X,Y) \arrow[r, "\displaystyle characteristics"] 
    & \mbox{Eulerian } (x,y) \arrow[d] \\
  \mbox{Self-similar Lagrangian } (a,b) \arrow[r, "\displaystyle renormalised \ characteristics", "\displaystyle \longrightarrow \Phi \ as \ t\uparrow T "'] \arrow[u]
&\mbox{Self-similar Eulerian } (\mX,\mY) \end{tikzcd}
\end{center}

Finding a solution to the transport equation \fref{2DPrandtlphom} is equivalent to finding the characteristics. Thus, this knowledge of an asymptotic equivalent for the characteristics map, allows to show the existence of an asymptotic equivalent for $u$: it is close to a rescaling of the self-similar profile $\T$. Said differently, solutions of the inviscid Prandtl equations become asymptotically self-similar during singularity formation, \emph{because the characteristics become asymptotically self-similar}.

The profile $\T$ cannot be an approximation of $u$ away from the singularity: the support of $\Theta$ is finite, and there are two singular zones within its support, at infinity in space in the tangential variable \fref{expanboundary2}, and close to the boundary \fref{expanboundary}-\fref{expanboundary12}. One thus needs as well to describe the solution in these two zones, and above the support of $\Theta$. We show that away from the singularity, either below or on the sides, $u$ reconnects to a regular profile ($f$ and $g$ in the Theorem below and the picture above). Above the singularity, $u$ reconnects to a less singular profile as well, but that undergoes a strong vertical motion created by the singularity underneath it.

We denote by $d_{C^k_{\text{loc}}}$ the standard $C^k_{\text{loc}}$ distance\footnote{For exemple: $ d_{C^k_{\text{loc}}}(u_0,u_0')=\sum_{l=0}^k \sum_{n=1}^\infty 2^{-n} \min (\| \nabla^l (u_0-u_0')\|_{L^{\infty}(\mathcal K_n)},1)$ with $\mathcal K_n=\{(x,y)\in \mH, \ |x|\leq n \mbox{ and } |y|\leq n\}$. Another choice for the covering of compact sets $(\mathcal K_n)_n$ would yield an equivalent distance.} for functions on $\mH$ that are $k$ times continuously differentiable, and equip $\mathcal E^k$ with the topology associated to the distance
$$
d_{\mathcal E^k}(u_0,u_0')=d_{C^k_{\text{loc}}}(u_0,u_0')+\| \nabla (u_0-u_0')\|_{L^{\infty}(\mH)}.
$$
We define $\mY^*_{\mu,\nu,\iota}(\mX)=\nu \mY^*(\iota \mX/\mu)$.

\begin{theoremmain}[Generic Separation] \lab{th:main2}

Let $(p^E_x,u^E)\in \mathcal F^4$. In the subset of $\mathcal E^4$ of initial data $u_0$ such that $T<\infty$ and $T<T_b$, there exists a dense open set  for which the corresponding solution satisfies the following. There exist parameters $(\mu,\nu,\iota)\in (0,\infty)^{2}\times \{-1,1\}$ and two constants $\kappa,C>0$ such that:
\begin{itemize}
\item \emph{Location of the singularity}. There exists $x^*\in C^4([0,T],\mathbb R)$, regular up to time $T$ such that $\nabla u(t)$ remains bounded in $\{(x,y,t), \ 0\leq t\leq T, \ |x-x^*(t)|\geq \epsilon \}$ for any $\epsilon>0$. 
\item \emph{Displacement line}. There exists $y^*\in C^3([0,T)\times \mathbb R)$ for which the properties below hold true, with: 
\be \lab{th:y*}
y^*(t,x)=\frac{2}{(T-t)^{\frac 14}} \mY^*_{\mu,\nu,\iota}\left(\frac{x-x^*}{(T-t)^{\frac 32}}\right)\left( 1+O\left((1-t)^{\kappa}+|x-x^*(t)|^{\kappa}\right)\right).
\ee
\item \emph{Self-similarity}. Let $\eta(t,x,y)=\left((T-t)^{\kappa}+|x-x^*(t)|^{\kappa}+y^{-\kappa}+|y^*(t,x)-y|^{-\kappa} \right)$ and let $u^*(t)=\pa_t x^*(t)$. For any $\epsilon>0$, there exists $\tilde \epsilon>0$ such that for $|x-x^*|\leq \tilde \epsilon$ and $y \leq (1-\epsilon)y^*(t,x)$:
$$
u(t,x,y)=u^*(t)+(T-t)^{\frac 12} \left(\Theta_{\mu,\nu,\iota}+\tilde u \right) \left(\mX,\mY \right),
$$
where $\mX=(T-t)^{-3/2}(x-x^*(t))$, $\mY=(T-t)^{1/4}y$, and where $\tilde u$ satisfies:
\bea
\lab{bd:tildeu} & |\tilde u(t,\mX,\mY)| \leq C \left( |\mX|^{\frac 13}+|\mY|^{-2}+\left((T-t)^{\frac 14}y^*(t,x)-\mY\right)^{-2} \right)\eta(t,x,y) \\
\lab{bd:paXtildeu} & |\pa_{\mX} \tilde u (t,\mX,\mY)| \leq  C \left( |\mY|^4+(1+|\mX|)^{-\frac 76}\left((T-t)^{\frac 14}y^*(t,x)-\mY\right)^{-3} \right)\eta(t,x,y) \\
\lab{bd:paYtildeu} & |\pa_{\mY} \tilde u(t,\mX,\mY)| \leq  C \left( |\mY|^{-3}+\left((T-t)^{\frac 14}y^*(t,x)-\mY\right)^{-3} \right)\eta(t,x,y).
\eea
\item \emph{Close to the displacement line}. Given any $K,\epsilon>0$, there exists $\tilde \epsilon>0$ such that for $(1-\epsilon)y^*(t,x)\leq y \leq y^*(t,x)-K$ and $|x-x^*|\leq \tilde \epsilon$:
\be \lab{id:topselfsim}
u(t,x,y)=u^*(t)+\frac{\mu \nu^2}{p^{*2}(y-y^*(t,x))^2}+\tilde v(t,x,y),
\ee
where, for $\tilde \eta(t,x,y)=\left((T-t)^{\kappa}+|x-x^*|^{\kappa}+(y^*-y)^{-\kappa}+(|y^*-y|/y^*)^{\kappa} \right)$:
\be \lab{bd:tildevtop}
|\tilde v|\leq \frac{C\tilde \eta(t,x,y)}{|y-y^*(t,x)|^2}, \ \ |\pa_x\tilde v|\leq \frac{C\tilde \eta(t,x,y)}{\left((T-t)^{\frac 32}+|x| \right)^{\frac76}|y-y^*(t,x)|^3}, \ \ |\pa_y \tilde v|\leq \frac{C\tilde \eta (t,x,y)}{|y-y^*(t,x)|^3}.
\ee
\item \emph{Reconnections below and above}. There exist two functions $f\in C^3([0,\infty), \mathbb R)$ and $g\in C^3(\mathbb R, \mathbb R)$ depending on $u_0$, $u^E$ and $p^E$ with:
\be \label{id:asymptotiquefetg}
f(y)-u^*\underset{y \rightarrow \infty}{\sim} \frac{\mu \nu^2}{p^{*2}}y^{-2}, \qquad g(y)-u^*\underset{y \rightarrow- \infty}{\sim} \frac{\mu \nu^2}{p^{*2}}y^{-2}, \qquad g(y)\underset{y \rightarrow \infty}{\rightarrow} u^E(T,x^*(T)),
\ee
such that for any $K>0$, as $(t,x)\rightarrow (T,x^*(T))$:
\be \lab{id:recobottom}
u(t,x,y) \rightarrow f(y) \qquad \text{uniformly for} \ y\leq K, \ \ 
\ee
\be \lab{id:recotop}
u(t,x,y) \rightarrow g(y-y^*(t,x))  \qquad \text{uniformly for} \ y^*(t,x)-K\leq y.
\ee

\end{itemize}

\end{theoremmain}

Let us make the following comments on the results of Theorem \ref{th:main2}.

\vspace{0.2cm}

\noindent{\it 1.} The set considered in the above Theorem is nonempty. Indeed, given any outer Eulerian flow $(u^E,p^E)$, there exist solutions blowing up outside the boundary, as negative enough initial data for the third equation in \fref{eq:ricattizerovorticity} will tend to $-\infty$ in finite time. However, the structure for unsteady separation described in this Theorem is not the only one occurring, and degenerate instable singularities also exist, see Proposition \ref{pr:selfsimdegen}.\\

\noindent{\it 2.} The estimates for the error in \fref{bd:tildeu}, \fref{bd:paXtildeu}, \fref{bd:paYtildeu} and \fref{bd:tildevtop} should be interpreted as follows. The first term in the right hand side is the typical size of $\Theta$, $\pa_\mX \T$ and $\pa_\mY \T$ respectively. The $\eta(t,x,y)$ term then quantifies a gain as $T-t\ll 1$, $|x-x^*|\ll 1 $ and $1\ll y \ll y^*$. For example, for $(\mX,\mY)$ in a compact set $\mathcal K$ in the support of $\Theta_{\mu,nu,\iota}$, these estimates imply $\| \tilde u\|_{C^1(\mathcal K)}=O((T-t)^{C(\kappa)})\rightarrow 0$. We have to distinguish between below the displacement line and near it as in this latter region the solution is close to a displaced version of $\Theta$. The identities \fref{id:topselfsim} and \fref{bd:tildevtop} indeed show the solution is close to the asymptotic expansion \fref{expanboundary} of $\T$ near the top part of its support, but with $y^*$ replacing $2\mY^*$. From the asymptotic behaviour of $\T$ in Proposition \ref{pr:selfsimfunda}, $\tilde u$ is of lower order compared to $\T$ precisely in a size one zone in $x$ around $x^*$, and up to a size one distance to the boundary and the displacement curve. These estimates are then sharp since they precisely fail when the solution reconnect to another nonsingular behaviour.\\

\noindent{\it 3.} Note that the asymptotic behaviour of $\T$ and that of the reconnection functions $f$ and $g$, in \fref{id:recobottom} and \fref{id:recotop} respectively, are compatible from Proposition \fref{expanboundary}-\fref{expanboundary12}.\\

\noindent{\it 4.} This convergence result also holds for higher order derivatives, which is a direct consequence of the proof of the Theorem. In particular the weighted estimates adapt naturally.

\subsubsection{A self-similar profile with symmetry}

Other degenerate singular behaviours are also possible. The degeneracy can come from two distinct aspects: at the singular point, the set of zero vorticity can locally not be a line (for example the intersection of two lines), and the tangential compression can be induced by a degenerate shock formation for Burgers. There exist a large range of self-similar profiles corresponding to these (infinitely many) degenerate cases. Their properties can also be studied with the same strategy used in the proof of Proposition \ref{pr:selfsimfunda}, and their stability similarly as in Theorem \ref{th:main2}. As a particular interesting example, we study in this paper one of the least degenerate cases, corresponding to a Burgers generic shock happening at the crossing of two lines of zero vorticity. A particular self-similar profile corresponding to this case enjoys remarkable properties: it is odd in $x$, admits an analytic expansion beyond its support, and explicit formulas can be obtained on the vertical axis. In a forthcoming paper we shall show its stability for the full Prandtl system. The prime ' notation below is not a differential, it is simply to distinguish from Proposition \ref{pr:selfsimfunda}.

\begin{propositionmain}[Degenerate symmetric profile] \lab{pr:selfsimdegen}
The mapping $\Phi' :(\ma,\mb)\mapsto (\mX,\mY)$ defined by
\be \lab{def:tildetheta}
\Phi' (\ma,\mb)=\left(\ma+\ma^3+\frac{\mb^2a}{4} \ , \ 2\int_{-\infty}^{\frac \mb2} \frac{d\tmb}{\left(1+\tmb^2\right)\left(1+3\Psi_1^2 \left(\frac{\ma+\ma^3+\frac{\mb^2a}{4}}{\left(1+\tmb^2\right)^{3/2}}\right)\right)} \right),
\ee
where $\Psi_1$ is defined by \fref{eq:def Psi1}, satisfies the following properties:
\begin{itemize}
\item[(i)] It is an analytic volume preserving diffeomorphism between $\mathbb R^2$ and the subset of the upper half plane $\{0<\mY<2 \mY^{'*}(\mX) \}$, where
\be \lab{def:tildeY*}
\mY^{'*}(\mX)=\int_{-\infty}^{\infty} \frac{d\tmb}{\left(1+\tmb^2\right)\left(1+3\Psi_1^2 \left(\frac{\mX}{\left(1+\tmb^2\right)^{3/2}}\right)\right)}
\ee
is an analytic curve with asymptotic expansion (with $B$ the Euler integral of the first kind)
\be \label{exp tildeY 0}
\mY^{'*}(\mX)\underset{\mX\rightarrow 0}{=} \pi- \frac{15}{16}\pi \mX^2+O(\mX^4), \ \ \mY^{'*}(\mX)\underset{\mX\rightarrow \pm \infty}{=} \frac{B(\frac 16,\frac 12)}{3}|\mX|^{-\frac 13}+O(|\mX|^{-1}).
\ee
\item[(ii)] The opposite of the first component of its inverse, extended by $0$ on the upper half-space:
\be \lab{selfsimtildetheta2}
\Theta'(\mX,\mY)=\left\{ \begin{array}{l l} - \Phi_1^{'-1}(\mX,\mY)=\ma \qquad \mbox{if } \ 0<\mY<2\mY^{'*}(\mX), \\ 0 \qquad \qquad \qquad \qquad \mbox{if } \ \mY=0 \ \mbox{ or } \ 2 \mY^{'*}(\mX)\leq \mY, \end{array} \right.
\ee
is a self-similar profile. Namely, the following is a solution of \fref{2DPrandtlphom}:
\be \lab{selfsimtildetheta}
u(t,x,y)=(T-t)^{\frac{1}{2}} \Theta' \left(\frac{x}{(T-t)^{\frac 32}},\frac{y}{(T-t)^{-\frac 12}} \right)
\ee
\item[(iii)] There holds $ \T'\in C^1(\mH)$. Moreover, $ \T'$, restricted to the set $\{0\leq \mY\leq 2 \mY^{'*}(\mX)\}$ is analytic, including at the boundaries $\{\mY=0\}\cup \{\mY=2 \mY^{'*}(\mX)\}$. The following extension\footnote{where for $z\geq 0$ and $z'>0$, we write $z=kz' +z\mbox{ mod }z'$, $k\in \mathbb N$, $0\leq z\mbox{ mod }z'<z'$. At each fixed $\mX$, this is a periodic extension along $\mY$ with period $2 \mY^{'*}(\mX)$}:
\be \label{eq:analyticextensiondegenerate}
\overline{\Theta}' (\mX,\mY):= \T '(\mX,\mY \mbox{ mod }2 \mY^{'*}(\mX))
\ee
is analytic on the whole upper half space, and is also a self-similar profile, in the sense that $\overline u(t,x,y)=(T-t)^{1/2}\overline \Theta' (x/(T-t)^{\frac 32},y/(T-t)^{-\frac 12} )$ is a solution of \fref{2DPrandtlphom}.
\item[(iv)] $\Theta'$ is odd in $\mX$, and satisfies $\Theta'(\mX,\mY^{'*}(\mX)+\mY)=\Theta'(\mX, \mY^{'*}(\mX)-\mY)$ for all $|\mY|< \mY^{'*}(\mX)$. $\T'$ is positive on the set $\{\mX<0, \ 0<\mY<2 \mY^{'*}(\mX) \}$ and negative on $\{\mX>0, \ 0<\mY<2 \mY^{'*}(\mX) \}$.
\item[(v)] The set $S=\{\pa_\mY  \Theta'=0\}$ is equal to $S_1\cup S_2$ where $S_1=\{\mX=0\}\cup \{\mY= \mY^{'*}(\mX)\}$ on which $ \T'(\mX,\mY^{'*}(\mX))=\Psi_1(\mX)$, and $S_2=\{\mY=0\}\cup \{\mY\geq 2 \mY^{'*}(\mX)\}$. The minimum of $\pa_\mX \T'$ on the set $S$ is attained at $(0,\pi)$ where one has the expansion:
\be \lab{eq:taylortildetheta}
\Theta' (\mX,\pi+\mY)=-\mX+\mX^3+\frac{\mY^2}{4}\mX+O(|\mX|^5+|\mY|^4|\mX|) \qquad \text{as} \ (\mX,\mY)\rightarrow (0,0).
\ee
\item[(vi)] The trace of the first order tangential derivative on the vertical axis is given by:
\be \label{eq:formulederiveeaxe}
\pa_\mX \Theta' (0,\mY)=-\sin^2 \left(\frac{\mY}{2}\right)\mathds 1_{0\leq \mY \leq 2\pi}.
\ee
\item[(vii)] The trace of the third order derivative on the vertical axis is given by:
$$
\pa_\mX^3 \Theta' (0,Y)=\frac{1}{576} \left[ \frac{96\sin^8\left(\frac{\mY}{2}\right)}{\cos^2\left(\frac{\mY}{2}\right)+\frac 16}-\sin (\mY)\left(270\mY-80\sin (\mY)+3\sin (2\mY)-\frac{686\sin (\mY)}{3\cos^2\left(\frac{\mY}{2}\right)+\frac 12}\right)\right]\mathds 1_{0\leq \mY \leq 2\pi}.
$$
The above function is positive on $(0,2\pi)$. It admits the expansions as $\mY\rightarrow 0$:
$$
\pa_\mX^3 \Theta' (0,\mY)=c\mY^8+O(\mY^{10}), \ \ \pa_\mX^3\Theta' (0,2\pi-\mY)=c'\mY+O(\mY^2), \ \ c,c'>0.
$$

\end{itemize}

\end{propositionmain}

\begin{remark} \lab{rk:degen}

\begin{itemize}
\item Note the difference in the scaling exponents when comparing \fref{selfsimtildetheta} with the generic profile \fref{def:selfsimtheta}. The above degenerate profile yields a slower expansion along the normal direction.
\item In \cite{CGIM}, we show that there is a stable blow-up pattern for equation \ref{1DPrandtl}, for which solutions converge to $(T-t)^{-1} \sin^2 \left( y/(\nu 2(T-t)^{-1/2})\right)\mathds 1_{0\leq y \leq \nu 2\pi(T-t)^{-1/2}}$. This partially shows that the profile of Proposition \ref{pr:selfsimdegen} is the stable attractor for solutions that are odd in $x$ when the singularity is located on the transversal axis.
\end{itemize}

\end{remark}

\subsection{Ideas of the proofs and organisation of the paper}

The proof of the local well-posedness result in Theorem \ref{th:main} for regular initial data relies on a careful study of the characteristics \fref{id:charac} and of the ODEs for tangential displacement \fref{eq:charx} underlying them, which then permits to retrieve the normal displacement by volume preservation (Lemma \ref{lem:paramincompressibility}). The sharp expression we find for the maximal time of existence is a consequence of two volume preserving dynamics: that of the characteristics map from Lagrangian to Eulerian variables, and that of the ODE for tangential displacement \fref{eq:charx} in the $(x,u)$ phase space.

In the study of the self-similar profiles, the transformation of \fref{eq:selfsimstationary} into \fref{eq:autosimstattransformee} uses a modified Crocco transform relying on the fact that the vector field in \fref{eq:selfsimstationary} has constant divergence. The solutions \fref{def:Phi} and \fref{def:tildetheta} are found by solving \fref{eq:autosimstattransformee}, and all the properties of the profiles can be obtained from computations on these explicit formulas.

To obtain the generic appearance of the generic self-similar profile during separation/singularity formation, we first define (Definition \ref{def:generic}) a condition for the characteristics at their first critical point at time $T$. This condition is stable under perturbations, by standard ODE stability arguments. Using a control argument, we moreover show an initial datum $u_0$ leading to blow-up can always be perturbed to ensure this condition is met at time $T$. Roughly speaking, this means that the map $u_0\mapsto u(T)$ is invertible, which relies on the fact that \fref{eq:charx} is volume preserving in the $(x,u)$ phase space. We then reconstruct the solution around the point at which the shock is forming using the characteristics map. We prove this map, since satisfying the aforementioned condition, after suitable self-similar renormalisations both in Lagrangian and Eulerian variables, converge to $\Phi$ defined by \fref{def:Phi}. We invert the characteristics, by a uniform application of the local inversion Theorem for $(x,y)\in [x^*-\epsilon,x^*+\epsilon]\times [0,\infty)$, and show it is close to $\Phi^{-1}$ in the self-similar zone. This part is lengthy and technical since, as the characteristics map becomes degenerate and has distinct asymptotic behaviours in various zones, each requiring a specific treatment, and since we track precisely all error terms to obtain an optimal picture (in an Eulerian zone of optimal $O(1)\times [0,\infty)$ size). The solution $u$ is then retrieved from the characteristics, and explicit computations end the proof of Theorem \ref{th:main2}. \\

\noindent The paper is organised as follows. Section \ref{sec:T} is devoted to the proof of Theorem \ref{th:main}, showing local existence of solutions and computing the maximal time of existence. Self-similar profiles are studied in Section \ref{sec:selfsim}. Subsection \ref{subsec:transforselfsimeq} transforms the self-similar profile equation \fref{eq:selfsimstationary} into the linear and local equation \fref{eq:autosimstattransformee}. The generic profile is studied in Subsection \ref{subsec:funda} where Proposition \ref{pr:selfsimfunda} is proved, the degenerate one in Subsection \ref{subsec:degen} where Proposition \ref{pr:selfsimdegen} is proved. The rest of the article is devoted to the proof of Theorem \ref{th:main2}. The Definition \ref{def:generic} of generic singularity for the characteristics is given in Subsection \ref{sec:characteristics} where it is proved to hold generically. Section \ref{sec:gen} then establishes the conclusions of Theorem \ref{th:main2} for solutions satisfying this condition. Appendix \ref{ap:y} shows how to retrieve one component of a $2$-d volume preserving map from the other, and Appendix \ref{ap:y2} contains computations for the characteristics used in the proof of Theorem \ref{th:main2}.

\subsection*{Acknowledgements}  The authors thank the anonymous referees for their useful comments. The work of T.-E. G. and N. M. is supported by Tamkeen under the NYU Abu Dhabi Research Institute grant
of the center SITE. C. C. is supported by the ERC-2014-CoG 646650 SingWave. N. M. is supported by NSF grant DMS-1716466. Part of this work was done while C. C., T.-E. G. and N. M. were visiting IH\'ES and they thank the institution. C. C. is grateful to New York University in Abu Dhabi for a stay during which part of this work was carried out.

\subsection*{Notations} We write $x\lesssim y$ if there exists a constant $C>0$ independent of the context such that $x\leq Cy$. We write $x\approx y$ if $x\lesssim y $ and $y\lesssim x$. We use Lagrangian variables $(X,Y)$ and Eulerian variables $(x,y)$. As they are equal at the initial time, we might use one notation or the other in several places, but in this context only. We use the notations $\pa_x$, $\pa/\pa_x$ or the subscript $\cdot_x$ to indicate partial differentiation. In some contexts, we write $\frac{\pa f}{\pa x}_{|y}$ to indicate partial differentiation with respect to $x$ with the variable $y$ being kept fixed.


\section{Local well-posedness and time of existence} \lab{sec:T}

This section is devoted to the proof of Theorem \ref{th:main}. We establish here local-existence of solutions to the inviscid Prandtl system \fref{2DPrandtlp}, prove that $T$ given by \fref{def:T1} is the maximal time of existence, and that separation occurs if $T$ is finite.

\begin{proof}[Proof of Theorem \ref{th:main}]

\noindent The proof relies on the special structure of the characteristics and uses the Crocco transformation. The existence follows from their nondegeneracy until time $T$, while the regularity follows from standard regularity theory for level sets of functions. We denote by $\nabla=(\pa_X,\pa_Y)$ the nabla operator in Lagrangian variables.

\noindent \textbf{Step 1} \emph{Existence}. We aim at establishing the existence and a formula for the characteristics, $\phi[t]:(X,Y)\mapsto (x,y)$ with $\phi=(\phi_1,\phi_2)$. We first solve for the tangential displacement:
\be \lab{id:xproof}
x(t,X,Y)=\phi_1[t](X,Y),
\ee
where $x$ above is the solution of \fref{eq:charx}. Notice that \fref{eq:charx} can always be solved \emph{globally} in time, so that $x(t,X,Y)$ is well-defined for all $(t,X,Y)\in \mathbb R_+\times \mH$. We next study the level sets $x=Cte$ in Lagrangian variables. Let us show first that they are non-degenerate, in that $\nabla x\neq 0$.

In the first case, we assume that $(X_0,Y_0)$ is such that $u_{0Y}(X_0,Y_0)\neq 0$. We then claim that $\nabla x(t,X_0,Y_0)\neq 0$ for any $t>0$, as obtained from the following diagram that we explain below.
\begin{center}
 \begin{tikzcd}[column sep=150pt,row sep=30pt]
  (X,Y) \arrow[d, "\displaystyle diffeomorphism"'] \arrow[dr, "\displaystyle diffeomorphism"] & \\
   (X,u_0(X,Y)) \arrow[r, "\displaystyle \underset{\displaystyle preserving}{volume}"']&(x(t,X,Y),u(t,X,Y))
  \end{tikzcd}
  \end{center}
Indeed, the Crocco transformation $(X,Y)\mapsto (X,u_0(X,Y))$ is in this case a well defined local diffeomorphism near $(X_0,Y_0)$. The vector field $(x,u)\mapsto (u,-p^E_x(t,x))$ in the ODE \fref{eq:charx} is divergence free in the $(x,u)$ phase space. Hence, the mapping $(x,u)\mapsto (x(t),u(t))$ is volume preserving in the $(x,u)$ phase space, hence a diffeomorphism. By composition, $(X,Y)\mapsto (x(t,X,Y),u(t,X,Y))$ is a local diffeomorphism near $(X_0,Y_0)$ implying that $\nabla x(t,X_0,Y_0)\neq 0$.
  
In the second case, we assume that $u_{0Y}(X_0,Y_0)=0$ or $Y_0=0$. The couple $(x,u)$ solves \fref{eq:charx}, so that in particular:
$$
\pa_t (\pa_{X} x)=\pa_X u, \ \ \text{implying} \ \ \pa_{tt} (\pa_{X} x)=\pa_X (-p^E_x(t,x))=-(\pa_X x) p^E_{xx}(t,x).
$$
This shows that at each fixed $(X,Y)$ as long as $\pa_X x$ does not vanish:
$$
\frac{d}{dt} \left(\frac{\pa_t \pa_X x}{\pa_X x} \right)=-\left( \frac{\pa_t \pa_X x}{\pa_X x}\right)^2-p^E_{xx}(t,x), \ \ \frac{\pa_t \pa_X x}{\pa_X x}(0)=\pa_X u_0.
$$
At the point $(X_0,Y_0)$, the quantity $\pa_t \pa_X x/\pa_X x$ is precisely the third component of the ODE system \fref{eq:ricattizerovorticity}. Because of the definition of $T$ \fref{def:T1}, the solution to the above differential equation is well defined for $t<T$. Hence $\pa_t \log (\pa_X x)$ is well-defined for $t<T$, implying $\pa_X x(t,X_0,Y_0)>0$ after integration. Hence, $\nabla x(t,X_0,Y_0)\neq 0$ for $t<T$ in this second case as well.\\

\noindent We just showed that $\nabla x \neq 0$ everywhere on $\mathbb R_+\times \mathbb R$, as long as $t<T$. Hence, in Lagrangian variables, the level sets $x=Cte$ are non-degenerate. At the boundary, as showed in the second case above: $\pa_X x_{|Y=0}\neq 0$. Therefore, the upper half plane is foliated by curves corresponding to the level sets $\Gamma [t,x]=\{(X,Y)\mbox{ such that }x(t,X,Y)=x\}$. Since $u_0,u^E,p^E_x$ are $C^2$, solving the ODE \fref{eq:charx} produces a solution map that is also of class $C^2$, and $x(t,X,Y)$ is a $C^2$ function. Hence the curves $\Gamma[t,x]$ are $C^2$. We then define an arclength parametrisation $s$ for each of these curves, $\Gamma[t,x]=\{ \gamma[t,x](s), \ s\geq 0\}$ where $s=0$ corresponds to the point at the boundary $Y=0$. Denoting by $X^E\mapsto x^E$ the Lagrangian to Eulerian map of the Bernouilli equation \fref{eq:bernouilli} (with $\dot x^E=u^E(t,x^E)$ and $x^E(0,X^E)=X^E$) then there holds $\gamma[t,x](s)\rightarrow X^E(t,x)$ as $s\rightarrow \infty$.\\

\noindent We now apply (i) in Lemma \ref{lem:paramincompressibility}. Since $x=\phi_1[t]$ is non-degenerate in that $\nabla \phi_1[t]\neq 0$, and that $\gamma$ is an arclength parametrisation its level curves, then the second component of the characteristics $y=\phi_2[t]$ is such that $\phi_2[t]=0$ at the boundary and $\phi[t]=(\phi_1[t],\phi_2[t])$ is volume preserving if and only if:
\be \lab{id:yproof}
y(t,X,Y)=\phi_2[t](X,Y)= \int_0^{s[t,x](X,Y)} \frac{d\tilde s}{|\nabla \phi_1[t](\gamma[t,x](\tilde s))|}.
\ee
Note that before $T$, the denominator in the above integral is uniformly away from $0$. The function $y$ above is of class $C^1$ because $\gamma $, $s$ and $\nabla \phi_1$ are. The mapping $(t,X,Y)\mapsto (t,x,y)$ is thus a $C^1$ diffeomorphism from $[0,T)\times \mH$ onto itself. We finally define the solution as $u(t,x,y)=u(t,X,Y)$ where the right hand side is the solution to \fref{eq:charx} (abusing notations). Clearly,
$$
\frac{\pa x}{\pa t}_{|X,Y}=u_0(X,Y)=u(t,x,y).
$$
Since the mapping $(X,Y)\mapsto (x,y)$ is $C^1$ and preserves the measure, $\pa_x\frac{\pa x}{\pa t}_{|X,Y}+\pa_y\frac{\pa y}{\pa t}_{|X,Y}=0$, yielding:
$$
\frac{\pa y}{\pa t}_{|X,Y}=-\int_0^y \pa_x u (t,x,\tilde y)d\tilde y.
$$
And since $\pa_t u(t,x(t),y(t))=-p^E_x(x(t)) $ and $u$ is $C^1$, one deduces that $u$ solves the inviscid Prandtl equations. Note that the matching condition at infinity in \fref{2DPrandtlp} are indeed satisfied for the following reason. Initially as $y\rightarrow \infty$, $u_0\rightarrow u^E_0$. $u^E$ solves the Bernouilli equation \fref{eq:bernouilli} that has a global solution, and whose characteristics correspond to the tangential displacement \fref{eq:charx} of the characteristics for $u$. This gives the desired compatibility.\\

\noindent \textbf{Step 2} \emph{Regularity}. Assume $u_0\in C^k$. The formula \fref{id:xproof} for $x(t,X,Y)$ defines a $C^k$ function since $x$ is obtained as the solution of the ODE \fref{eq:charx} with a $C^k$ vector field. In the formula \fref{id:yproof}, $\nabla \phi_1 [t]$ is $C^{k-1}$, and $s$ and $\gamma$ come from the parametrisation of the level sets of a $C^k$ function, hence are also $C^{k}$. Therefore, $u$ is of class $C^{k-1}$. The continuity of the flow follows from similar arguments.\\

\noindent \textbf{Step 3} \emph{Uniqueness}. If $u$ is a $C^2$ solution then uniqueness is straightforward as the characteristics are well defined and have to produce the diffeomorphism constructed above from the uniqueness property (i) in Lemma \ref{lem:paramincompressibility}. In the case where $u \in C^1$ only, let us detail how the normal component of the characteristics and the volume preservation can be obtained. Define the characteristics $(x(t),y(t))$ through:
$$
\frac{\pa x}{\pa t}=u(t,x,y), \ \ x(0)=X, \ \ \frac{\pa y}{\pa t}=-\int_0^{y(t)}u_x(t,x,y), \ \ y(0)=Y. 
$$
One can indeed solve the second equation because the function $\int_0^{y}u_x(t,x,y)$ is $C^1$ in the third variable. One obtains characteristics $(x,y)$ such that $x$ is $C^1$ in $(X,Y)$ and $y$ is only $C^1$ in $t$ and continuous in the other variables. $u$ then solves $\dot u=-p^E_x(x)$ along the characteristics, implying that $x$ is given by the formula \fref{id:xproof}. Moreover, since $x$ is a $C^{1}$ function, and $y$ is a $C^{1}$ function in $t$, with $\pa_t y$ being $C^1$ in $y$, such that $\pa_y(\pa_{t}y(t))=-\pa_x (\pa_t(x(t)))$, an approximation argument using a regularisation procedure gives that the characteristics must preserve volume. The mapping $(X,Y)\mapsto (x,y)$ is then a bijection preserving volume with $x\in C^1$ and $y$ continuous, which can be showed to be necessarily of the form described in Step 1.\\

\noindent \textbf{Step 4} \emph{Blow-up and separation}. Assume that $T<\infty$. Then by definition of $T$ one solution to the ODEs \fref{eq:ricattizerovorticity} must blow up at time $T$, which is only possible if $u_x\rightarrow -\infty$ as $t\rightarrow T$. If moreover $T=T_a$, then from Step 1, there exists a point $(X_0,Y_0)\in \mH$ with $Y_0>0$, such that $\nabla x(t,X_0,Y_0)\rightarrow 0$ as $t\uparrow T$. Set $(x_s(t),y_s(t))=(x(t,X_0,Y_0),y(t,X_0,Y_0))$. Then $x_s$ is a regular up to $T$ solution to \fref{eq:charx} and has a finite limit $x_0$ as $t\uparrow T$, and $\lim_{t\uparrow T}y_s=\infty$ from \fref{id:yproof} and the vanishing of $\nabla x$. Hence there is boundary layer separation in the sense of Definition \ref{def:separation}.

\end{proof}


\section{Construction of self-similar profiles} \lab{sec:selfsim}

\subsection{The equation for self-similar profiles} \label{subsec:transforselfsimeq}

We study self-similar profiles $\T$ such that $u(t,x,y)=(T-t)^{\alpha-1} \T(x/(T-t)^{\alpha},y/(T-t)^\beta)$, for some $\alpha,\beta\in \mathbb R$, solves \fref{2DPrandtlphom}. Dropping the boundary condition for the normal velocity, this amounts to solve the stationary equation:
\begin{equation} \label{eq:selfsimstationary2}
\left\{ \begin{array}{l l}
(1-\alpha)\Theta+(\alpha \mX+\T)\pa_\mX \T +(\beta \mY +\Upsilon )\pa_\mY \T=0,\\
\pa_\mX \T+\pa_\mY \Upsilon=0.
\end{array}
\right.
\end{equation}
A first method, due to Cassel, Smith and Walker \cite{CSW96} (see also \cite{ESC83}) transforms the nonlinear nonlocal equation \fref{eq:selfsimstationary2} into a nonlinear local equation \fref{eq:autosimstattransformee2}. It relies on the Crocco change of variables $(\mX,\mY)\mapsto (\mX,\T)$. We find a new change of variables that transforms \fref{eq:selfsimstationary2} into the linear and local equation \fref{eq:autosimstattransformee1}. This change of variable can be thought of as a volume preserving Crocco transform, and the variables $(\ma,\mb)$ should be thought of as Lagrangian self-similar variables as explained in Remark \ref{re:lagrangian interpretation}. We include the proof of (iii) for the sake of completeness.

We use statement (ii), to find explicit solutions to \fref{eq:selfsimstationary2}. While (iii) is not useful to solve \fref{eq:selfsimstationary2} explicitly, the knowledge of the solution to \fref{eq:autosimstattransformee2}, i.e. $\pa_\mY \T $ as a function of $\mX$ and $\T$, is useful for certain computations. We mention that the classification of analytic solutions to \fref{eq:selfsimstationary2}, relying on \fref{eq:autosimstattransformee1}, is an interesting open problem.

\begin{lemma} \label{lem:resolutioneqprofiles}
Let $\alpha,\beta \in \mathbb R$, $\alpha\neq 1$, $\Omega \subset \mathcal H$ open, $\Theta\in C^3(\Omega)$, $\ma=-\Theta$ and assume $\T\neq 0$ on $\Omega$. Then the following statements are equivalent:
\begin{itemize}
\item[(i)] There exists $\Upsilon \in C^1(\Omega)$ such that $(\Theta,\Upsilon)$ solves \fref{eq:selfsimstationary2}. One has $\nabla \Theta \neq 0$ on $\Omega$ and for each $q$ in the range of $\T$ the level set $\Gamma_q=\{(\mX,\mY)\in \Omega, \ \Theta(\mX,\mY)=q\}$ is diffeomorphic to $\mathbb R$.
\item[(ii)] There exists $\mb\in C^2(\Omega)$ such that the mapping $(\mX,\mY)\mapsto (\ma,\mb)$ is a volume preserving diffeomorphism between $\Omega$ and an open set $\Omega' \subset \mathbb R^2$, and, writing $\mX(\ma,b)$ as a function on $\Omega'$, it satisfies:
\be \label{eq:autosimstattransformee1}
(\alpha-1)\ma\pa_\ma \mX+(1+\beta)\mb\pa_\mb \mX=\alpha \mX-\ma.
\ee
One has that for every $\ma\in \mathbb R$, the set $\{\mb\in \mathbb R, \ (\ma,\mb)\in \Omega'\}$ is either an interval or empty.
\end{itemize}
If moreover $\pa_\mY \T\neq 0$ on $\Omega$ these statements are equivalent to the following one:
\begin{itemize}
\item[(iii)] The change of variables $(\mX,\mY)\mapsto (\mX,\T)$ maps $\Omega$ onto an open set $\Omega''$. Writing $\tau(\mX,\T)$ as a function on $\Omega''$ it satisfies:
\be \label{eq:autosimstattransformee2}
(\alpha-1)\T \frac{\pa}{\pa \Theta}_{|\mX}\tau+(\alpha \mX+\T)\frac{\pa}{\pa \mX}_{|\Theta}\tau=(\alpha-1-\beta)\tau.
\ee
One has that for every $q$ in the range of $\T$, the set $\{ \mX \in \mathbb R, \ (\mX,q)\in \Omega''\}$ is an interval.
\end{itemize}
\end{lemma}

\begin{remark}

The hypothesis $\T\neq 0$ is not necessary. The implications between (i), (ii), (iii) would hold, provided additional regularity assumptions near the set $\{\T=0\}$, see the proof.

The hypothesis $\nabla \T \neq 0$ is necessary for (ii). For a solution to \fref{eq:selfsimstationary2}, there should be a different change of variables $(\mX,\mY)\mapsto (\ma_{\mathcal C},\mb_{\mathcal C})$ for each connected component $\mathcal C$ of the set $\{\nabla \T\neq 0\}$. For example, for $\bar \T'$ defined by \fref{eq:analyticextensiondegenerate}, for each component $\mathcal C_n=\{2(n-1)\mY^{'*}(\mX)<\mY<2n\mY^{'*}(\mX) \}$, a change of variables for (ii) is given by $(\mX,\mY)\mapsto (\ma_{\mathcal C_n},\mb_{\mathcal C_n})$ with $(\ma_{\mathcal C_n},\mb_{\mathcal C_n})=\Phi^{'-1}(\mX,\mY-2(n-1)\mY^{'*}(\mX))$, where $\Phi'$ is given by \fref{def:tildetheta}.

The assumption $\pa_\mY\T\neq 0$ is necessary for (iii). For a solution to \fref{eq:selfsimstationary2}, there should be a different change of variables $(\mX,\mY)\mapsto (\mX,\T)$ for each connected component $\mathfrak C$ of the set $\{\pa_\mY \T\neq 0\}$. For $\T$ given by \fref{pr:selfsimfunda}, there is a different change of variables on $\{ 0<\mY<\mY^*(\mX)\}$ and $\{ 0<\mY<\mY^*(\mX)\}$, yielding different solutions to \fref{eq:autosimstattransformee2}, see Lemma \ref{lem:ode}.

\end{remark}

\begin{proof}

\textbf{Step 1} \emph{(i) implies (ii)}. Assume (i) and let $\sfv =(\alpha \mX+\T )\partial_\mX+(\beta \mY +\Upsilon)\partial_\mY$ and write $\sfv .f=(\alpha \mX+\T )\partial_\mX f+(\beta \mY +\mathcal V)\partial_\mY f$ to denote the differentiation along $\sfv$ on $\Omega$. Then, the second equation in \fref{eq:selfsimstationary2} implies, with $\nabla_{\mX,\mY}.$ denoting the divergence in $(\mX,\mY)$ variables:
\be \label{id:divv}
\nabla_{\mX,\mY} .\sfv=\alpha+\beta.
\ee
Assume first that $\Omega$ is connected. As $\T$ is $C^3$ with $\nabla \Theta \neq 0$, and given the hypothesis on its level sets, we use one of the formulas provided by Lemma \ref{lem:paramincompressibility} and get the existence of $\mb\in C^2(\Omega )$ such that the mapping $\phi:(\mX,\mY)\mapsto (\ma,\mb)$ is a $C^2$ volume preserving diffeomorphism onto some open set $\Omega'\subset \mathbb R^2$. Let $\sfv'=g(\ma,\mb)\pa_\ma+h(\ma,\mb)\pa_\mb$ denote the push forward of $\sfv$ from $\Omega$ to $\Omega'$, that is, the vector field such that for any $f\in C^1(\Omega')$, $\sfv'.f=\sfv .(f\circ \phi )$.

Since $\phi$ is $C^2$ and $\sfv$ is $C^1$, one has that $\sfv'$ is $C^1$, that is, $g$ and $h$ are $C^1$. The first equation in \fref{eq:selfsimstationary2} gives $\sfv .\T=(\alpha-1)\T$, so that $\sfv'.\ma=(\alpha-1)\ma$, and we get:
\be \label{id:identitiesv}
g(\ma,\mb)=(\alpha-1)\ma
\ee
so $g$ is in fact smooth. As $\phi$ is $C^2$ and preserves volume, we get conservation of divergence so that: $\nabla_{\ma,\mb}.\sfv'=\nabla_{\mX,\mY}.\sfv=\alpha+\beta$ where we used \fref{id:divv}. This gives, using \fref{id:identitiesv}:
\be \label{id:divv'}
\pa_\mb h=1+\beta
\ee
Consequently, there exists $\varphi (\ma)$ a $C^1$ function such that $h(\ma,\mb)=(1+\beta)\mb+\varphi (\ma)$. We change variables and set $\tmb(\ma,\mb)=\mb+\xi (\ma)$ for $\xi$ a $C^1$ function to be determined. Let us denote by $\sfv''$ the pushforward of $\sfv'$ by $(\ma,\mb)\mapsto (\ma,\tmb)$, so that:
$$
\sfv''=(\alpha-1)\ma\frac{\pa}{\pa \ma}_{|\mb}+((1+\beta)\tmb+(\alpha-1)\ma \xi '(\ma)-(1+\beta)\xi(\ma)+\varphi (\ma))\frac{\pa}{\pa \tmb}_{|\ma}.
$$
Since $\ma\neq 0$ on $\Omega'$ because $\T\neq 0$ on $\Omega$, since $\varphi$ is $C^1$ and $\alpha\neq 1$, there exists $\xi$ a $C^2$ solution of $(\alpha-1)\ma \xi '(\ma)-(1+\beta)\xi(\ma)+\varphi (\ma)=0$, producing $\sfv''=(\alpha-1)\ma\pa_\ma+(1+\beta)\tmb\pa_{\tmb}$. Note that the change of variables $(\ma,\mb)\mapsto (\ma,\mb)$ is $C^1$ and volume preserving. Therefore, up to relabelling $\tmb$ as $\mb$, we can always choose $\mb$ such that:
\be \label{id:identitiesv2}
h(\ma,\mb)=(1+\beta)\mb.
\ee
Hence $\sfv'=(\alpha-1)\ma \pa_\ma+(1+\beta)\mb\pa_\mb$ by \fref{id:identitiesv} and \fref{id:identitiesv2}. Since $\sfv .\mX=\alpha \mX+\T$, we get that $\sfv'.\mX=\alpha \mX-\ma$, which is shows \fref{eq:autosimstattransformee1}. Due to the assumptions on the level sets of $\T$, for each $\ma$, the set $\{\mb\in \mathbb R, \ (\ma,\mb)\in \Omega'\}$ is either an interval or empty. Hence (ii) is established.

In case $\Omega$ is not connected, we partition it into connected components $\Omega=\cup_n \Omega_n$, and denote by $\mb_n\in C^1(\Omega_n)$ the function we just obtained and $\Omega_n'=(\ma,\mb_n)(\Omega_n)$. From the properties of the level sets of $\T$, as $\ma=-\T$, $\Omega_n'\cap \Omega_{n'}'=\emptyset$ whenever $n\neq n'$. Hence, the function $\mb(\mX,\mY)=\sum_n \delta_{(\mX,\mY)\in \Omega_n}\mb_n(\mX,\mY)$ and $\Omega'=\cup_n \Omega'_n$ give (ii) in that case.\\

\noindent \textbf{Step 2} \emph{(ii) implies (i)}. Assume (ii) and let $\sfv'=(\alpha-1)\ma\pa_\ma +(1+\beta)\mb\pa_\mb $. Let the pullback of the vector field $\sfv=i(\mX,\mY)\pa_\mX+j(\mX,\mY)\pa_\mY$ by the mapping $(\mX,\mY)\mapsto (\ma,\mb)$ be $\sfv'$. Then:
\be \label{id:identitiesv3}
i(\mX,\mY)=\alpha \mX+\T.
\ee
because of \fref{eq:autosimstattransformee1}. Exploiting as in Step 1 the preservation of divergence by the mapping $(\ma,\mb)\mapsto(\mX,\mY)$, we get that $j$ is $C^1$ with $\pa_\mX i+\pa_\mY j=\alpha+\beta$. Injecting \fref{id:identitiesv3} yields $\pa_\mY j=\beta-\pa_\mX \T$. Hence there exists $\Upsilon\in C^0$ continuously differentiable with respect to $\mY$ such that: 
\be \label{id:identitiesv4}
j(\mX,\mY)=\beta \mY+\Upsilon(\mX,\mY), \qquad \mbox{with} \qquad \pa_\mX\T+\pa_\mY \Upsilon=0.
\ee
The definition of $\sfv'$ gives $\sfv'.\ma=(\alpha-1)\ma$, which, combined with $\ma=-\T$, \fref{id:identitiesv3} and \fref{id:identitiesv4} shows that \fref{eq:selfsimstationary2} is satisfied. As $(\mX,\mY)\mapsto (\ma,\mb)$ preserves volume and $\ma=-\T$, we get $\nabla \ma\neq 0$ on $\Omega$. For $q$ in the range of $\T$, the identity $\{(\mX,\mY)\in \Omega, \ \T(\mX,\mY)=q \}=(\mX,\mY)(\{(q,\mb)\in \Omega' \})$, and the fact that $\{(q,\mb)\in \Omega' \}$ is a non-empty interval, imply that this set is diffeomorphic to $\mathbb R$. Hence (i) is established.\\

\noindent \textbf{Step 3} \emph{(ii) implies (iii)}. Assume (ii) and that $\pa_\mY \T\neq 0$ on $\Omega$. The mapping $(\mX,\mY)\mapsto (\ma,\mb)$ is a $C^1$ volume preserving diffeomorphism, hence the determinant of the Jacobian matrices is either $-1$ or $1$. Up to changing $\mb$ into $-\mb$, we assume that it is $1$, so that:
\be \label{id:jacobianvolume}
\begin{pmatrix} \frac{\pa \mY}{\pa \mb} &- \frac{\pa \mX}{\pa \mb} \\ -\frac{\pa \mY}{\pa \ma} & \frac{\pa \mX}{\pa \ma}  \end{pmatrix}=\begin{pmatrix} \frac{\pa \ma}{\pa \mX} & \frac{\pa \ma}{\pa \mY} \\ \frac{\pa \mb}{\pa \mX} & \frac{\pa \mb}{\pa \mY}  \end{pmatrix}.
\ee
In particular, using $\ma=-\T$, we get $\tau=\pa_\mb \mX$. Plugging this identity in \fref{eq:autosimstattransformee1}, then differentiating with respect to $\mb$ yields:
\be \label{id:identitiesv5}
(\alpha-1)\ma\pa_\ma \tau +(1+\beta)\mb\pa_\mb \tau=(\alpha-1-\beta)\tau.
\ee
The change of variables $(\ma,\mb)\mapsto (\ma,\mX)$ produces
$$
\frac{\pa}{\pa \ma}_{|\mb} \tau=\frac{\pa}{\pa \ma}_{|\mX} \tau+\frac{\pa}{\pa \ma}_{|\mb} \mX \frac{\pa}{\pa \mX}_{|\ma} \tau, \qquad \frac{\pa}{\pa \mb}_{|\ma}\tau=\frac{\pa}{\pa \mb}\mX\frac{\pa}{\pa \mX}_{|\ma}\tau.
$$
Injecting the above identities in \fref{id:identitiesv5} gives:
$$
(\alpha-1)\ma \frac{\pa}{\pa \ma}_{|\mX} \tau +\left((\alpha-1)\ma\frac{\pa}{\pa \ma}_{|\mb} \mX+(1+\beta)\mb\frac{\pa}{\pa \mb}_{|\ma} \mX \right)\frac{\pa}{\pa \mX}_{|\ma}\tau=(\alpha-1-\beta)\tau.
$$
Above, $(\alpha-1)\ma\frac{\pa}{\pa \ma}_{|\mb} \mX+(1+\beta)\mb\frac{\pa}{\pa \mb}_{|\ma} \mX=\alpha \mX-\ma$ from \fref{eq:autosimstattransformee1}, which proves \fref{eq:autosimstattransformee2} since $\ma=-\T$. Since the statement (i) is satisfied from Step 2, for each $q$ in the range of $\T$, the set $\{(\mX,\mY)\in \Omega, \ \T(\mX,\mY)=q\}$ is diffeomorphic to $\mathbb R$. This set being diffeomorphic to $\{ (\mX,q) \in \Omega''\}$, this latter set is diffeomorphic to $\mathbb R$, and (iii) is established.\\

\noindent \textbf{Step 4} \emph{(iii) implies (ii)}. Assume (iii) and $\pa_\mY \T\neq 0$ on $\Omega$. For $q$ in the range of $\T$, the set $\{(\mX,\mY)\in \Omega, \ \T(\mX,\mY)=q\}$ is diffeomorphic to the nonempty open interval $\{ \mX \in \mathbb R, \ (\mX,q)\in \Omega''\}$, hence is diffeomorphic to $\mathbb R$. As in the beginning of Step 1, there then exists a function $\mb\in C^2(\Omega)$ be such that $(\mX,\mY)\mapsto (\ma,\mb)$ is volume preserving, with determinant $1$. Change variables $(\mX,\T)\mapsto (\ma,\mb)$. Then as $\ma=-\T$, using \fref{id:jacobianvolume} and the definition of $\tau$ (recall $\tau \neq 0$):
\begin{align*}
&\frac{\pa}{\pa \T}_{|\mX}=\frac{\pa \ma}{\pa \T}_{|\mX}\frac{\pa}{\pa \ma}_{|\mb}+\frac{\pa \mb}{\pa \T}_{|\mX}\frac{\pa}{\pa \mb}_{|\ma}=-\frac{\pa}{\pa \ma}_{|\mb}+\frac{\pa \mb}{\pa \mY}_{|\mX}\frac{\pa \mY}{\pa \T}_{|\mX}\frac{\pa}{\pa \mb}_{|\ma}=-\pa_\ma+\frac{\pa_\ma \mX}{\tau} \pa_\mb,\\
&\frac{\pa}{\pa \mX}_{|\T}=\frac{\pa \ma}{\pa \mX}_{|\T}\frac{\pa}{\pa \ma}_{|\mb}+\frac{\pa \mb}{\pa \mX}_{|\T}\frac{\pa}{\pa \mb}_{|\ma}=\frac{\pa \mb}{\pa \mX}_{|\ma}\frac{\pa}{\pa \mb}_{|\ma}=\frac{1}{\frac{\pa \mX}{\pa \mb}_{|\ma}}\frac{\pa}{\pa \mb}_{|\ma}=\frac{1}{\tau}\pa_\mb,
\end{align*}
where in the last terms, and from now on, partial derivatives are taking in $(\ma,\mb)$ coordinates. Injecting these identities in \fref{eq:autosimstattransformee2}, using that $\tau=\pa_\mb \mX$ we get:
$$
(\alpha-1)\ma\pa_{\ma\mb}\mX+\left((1-\alpha)\ma\frac{\pa_\ma \mX}{\pa_\mb \mX}-\frac{\alpha \mX+\ma}{\pa_\mb \mX}\right)\pa_{\mb\mb}\mX=(\beta+1-\alpha)\pa_\mb \mX.
$$
Dividing by $\pa_\mb \mX$, the above identity yields:
$$
\pa_\mb \left((\alpha-1)\ma\frac{\pa_\ma \mX}{\pa_\mb \mX}+\frac{\alpha \mX+\ma}{\pa_\mb\mX} \right)=\beta+1.
$$
Thus, there exists a $C^1$ function $\varphi (\ma)$ such that:
$$
(\alpha-1)\ma \pa_\ma \mX+\alpha \mX+\ma=(\beta+1+\varphi (\ma))\pa_\mb \mX.
$$
There exists a $C^1$ change of variables of the form $\tmb=\mb+\xi (\ma)$ that transforms the above equation into \fref{eq:autosimstattransformee1}, by applying verbatim the same reasoning made in Step 1. For every $\ma\in \mathbb R$, the set $\{\mb\in \mathbb R, \ (\ma,\mb)\in \Omega'\}$ is diffeomorphic to $\{ \mX \in \mathbb R, \ (\mX,q)\in \Omega''\}$, hence is either empty or an interval. We have established (ii).

\end{proof}

\subsection{Generic self-similar profile} \lab{subsec:funda}

It is convenient to decompose the proof of Proposition \ref{pr:selfsimfunda} in two parts. The first part is a Lagrangian approach, using formula \fref{def:thetageom} and certain transformations on the Lagrangian side ; some properties are directly showed out of this formula. First we study the curve $\mX\mapsto \mY^*(\mX)$ and prove (i). This shows $\T$ is well-defined, and allows to prove that it is a self-similar profile, which is (ii). The symmetry (iii), the set $\{\pa_\mY \T=0\}$ and the Taylor expansion \fref{taylorthetalag} in (iv) can be studied directly from the formula \fref{def:thetageom}. Note that although this profile diverges to infinity as $\mY \rightarrow 0$ or $2\mY^*(\mX)$, it still makes sense to speak of \fref{def:selfsimtheta} as a solution of Prandtl's equations on its support: we show that the quantity $\int_{0}^y u_x(t,x,\tilde y)d\tilde y$ is well defined.

To establish (v) however, we rely on different techniques. The second part presents another approach for the study of self-similar profiles. It is an Eulerian one since it relies on the study of the equation \fref{eq:autosimstattransformee2}, following \cite{ESC83,CSW96}. The study of \fref{eq:autosimstattransformee2} can indeed complement that of \fref{eq:autosimstattransformee1}.

We shall use at many locations that $\Psi_1$ is the inverse of the function $\mX\mapsto-\mX-\mX^3$ (i.e. $-\Psi_1(\mX)-\Psi_1^3(\mX)=\mX$), that it is analytic and odd, that $\pa_\mX \Psi_1$ attains its minimum at $0$ where $\pa_\mX \Psi_1(0)=-1$ and that:
\begin{align} \lab{id:psi1implicit}
&u+p^{*2}u^3=v \qquad \Leftrightarrow \qquad u= -\frac{1}{p^*}\Psi_1\left(p^*v \right) \qquad \mbox{for all }(u,v)\in \mathbb R^2,\\
& \lab{id:psi1origin} \Psi_1(\mX)=-\mX+\mX^3+O(|\mX^5|)\qquad \mbox{as } \mX\rightarrow 0,\\
\lab{id:psi1infinity}
&\Psi_1(\mX)=\mp |\mX|^{\frac{1}{3}}\pm\frac{1}{3}|\mX|^{-\frac 13}+O(\mX^{-1}) \qquad \mbox{as }\mX\rightarrow \pm \infty.
\end{align}
We will write $\Phi=(\Phi_1,\Phi_2)$ for the components of $\Phi$.

\begin{proof}[Proof of (i) in Proposition \ref{pr:selfsimfunda}]

These identities are obtained through direct computations. First, the identity \fref{id:psi1implicit} implies after a direct computation that the determinant of the differential of $\Phi$ is $1$. It is thus a volume preserving diffeomorphism onto its image. 

Fixing any $\mX \in \mathbb R$ and parametrising the curve of equation $\Phi_1(\ma,\mb)=\ma+\mb^2+p^{*2}\ma^3=\mX$ with the variable $\mb$, we get that $\Phi_2=\int_{-\infty}^\mb (1+3\Psi_1^2 (p(\mX-\tmb^2)))^{-1}d\tmb$ on this curve from \fref{eq:def Y*}. Recalling \fref{eq:def Y*}, the set $\{ \Phi_2(\ma,\mb), \ \Phi_1=\mX\}$ is then equal to the interval $(0<\mY<2\mY^*(\mX))$. Therefore, the image of $\Phi$ is indeed the set $\{(\mX,\mY), \ 0<\mY<2\mY^*(\mX) \}$.

We next establish \eqref{exp Y 0}. After two successive changes of variable $\Psi_1(p^*(\mX-\tmb^2))/p^*=\Theta$ and $\Theta=z+\Psi_1(p^*\mX)/p^*$, and using \fref{id:psi1implicit}, we can rewrite \fref{eq:def Y*} as:
\be \lab{formulamystarproof}
\mY^*(\mX)=\int_{\frac{\Psi_1(p^* \mX)}{p^*}}^{\infty}\frac{d\T}{2\sqrt{\Theta+p^{*2}\Theta^3+\mX}}=\int_0^\infty\frac{dz}{2\sqrt{z}\sqrt{1+p^{*2}z^2+3zp^*\Psi_1(p^*\mX)+3\Psi_1^2(p^*\mX)}}.
\ee
We compute that $\mY^*(0)=\frac{1}{2}\int_0^\infty (z+p^{*2}z^3)^{-1/2}dz=4\Gamma(\frac{5}{4})^2 (\pi p^*)^{-\frac 12}=\frac{3\pi}{8}$ for the specific value $p=p^*$ given by \fref{def:pstar}. Other direct computations using \fref{formulamystarproof} then give:
\begin{align*}
&\pa_{\mX} \mY^*(0)= \sqrt{p^*}\frac 34 \int_0^{\infty} \frac{\sqrt z dz}{(1+z^2)^{\frac 32}}=\frac{3\sqrt{p^*}}{4\sqrt{\pi}}\Gamma\Big(\frac{3}{4}\Big)^2=1,\\
&\pa_{\mX}^2 \mY^*(0)=-\frac 32 p^{*\frac 32} \int_0^{\infty} \frac{dz}{\sqrt z (1+z^2)^{\frac 32}}+\frac{27}{8} p^{*\frac 32} \int_0^{\infty} \frac{z^{\frac 32}dz}{(1+z^2)^{\frac 52}}=-\frac{5\Gamma\Big(\frac 14\Big)}{72\pi^5},
\end{align*}
which concludes the proof of \eqref{exp Y 0}. We finish by proving \eqref{exp Y infty}. If $\mX>0$ we change variables in \fref{formulamystarproof} setting $z=-\tilde{z}\Psi_1(p^*\mX)/p^*$ (note that $\Psi_1(p^*\mX)<0$ in this case):
\begin{align}
\mY^*(\mX) = \frac{1}{2\sqrt{p^*}|\Psi_1(p^*\mX) |^{\frac 12}}\int_{0}^{\infty} \frac{d\tilde{z}}{\sqrt{\tilde{z}}\sqrt{3+\Psi_1^{-2}(p^*\mX)-3\tilde{z}+\tilde{z}^2}}.
\end{align}
We use the asymptotic behaviour \fref{id:psi1infinity} of $\Psi_1$ to deduce that
\begin{align}
\mY^*(\mX)=C_+ \mX^{-\frac{1}{6}}+O(\mX^{-\frac 56})=C_+\mX^{-\frac{1}{6}}+O(\mX^{-\frac 56}) \ \text{as} \ \mX\rightarrow \infty
\end{align}
where $C_+=2^{-1}p^{*-2/3}\int_{- 1}^\infty (z^3 + 1)^{-1/2}dz$. The same computation in the case $\mX\rightarrow -\infty$ gives:
\begin{align}
\mY^*(\mX)&=C_- |\mX|^{-\frac{1}{6}}+O(|\mX|^{-\frac 56}) \ \text{as} \ \mX\rightarrow -\infty,
\end{align}
where $C_-=2^{-1}p^{*-2/3}\int_{1}^\infty (z^3 -1)^{-1/2}dz$. The constants $C_\pm$ can be computed explicitely:
\be \lab{id:cpm}C_+=\frac{3\Big(\frac 32\Big)^{\frac 13}\pi^{\frac{5}{2}}\Gamma\Big(\frac 13\Big)}{4\Gamma\Big(\frac 14\Big)^{\frac 83}\Gamma\Big(\frac 56\Big)},~~C_-=\frac{3\Big(\frac 32\Big)^{\frac 13}\pi^{\frac{5}{2}}\Gamma\Big(\frac 76\Big)}{2\Gamma\Big(\frac 14\Big)^{\frac 83}\Gamma\Big(\frac 23\Big)}.\ee
This ends the proof of \fref{exp Y infty}.

 \end{proof}

 \begin{proof}[Proof of (ii) in Proposition \ref{pr:selfsimfunda}]
 
This is a consequence of (ii) in Lemma \ref{lem:resolutioneqprofiles}. Indeed, $\mX$ given by \fref{def:Phi} solves \fref{eq:autosimstattransformee1} with $\alpha=3/2$ and $\beta=-1/4$ by a direct computation. Moreover, we have already proved (i) in Proposition \ref{pr:selfsimfunda}, so that the mapping $(\mX,\mY)\mapsto (\ma,b)$ is volume preserving. Applying Lemma \ref{lem:resolutioneqprofiles}, we get that, on the set $\{(\mX,\mY), \T(\mX,\mY)\neq 0\}$, there exists $\Upsilon$ such that $(\T,\Upsilon)$ solves \fref{eq:selfsimstationary2}. Moreover, from Step 1 in the proof of Lemma \ref{lem:resolutioneqprofiles}, we have that $\frac 12 \ma\pa_\ma +\frac 34 \mb\pa_\mb$ is the pushforward of $(\frac 32 \mX+\T)\pa_\mX  +(-\frac 14 \mY +\Upsilon )\pa_\mY$ by the mapping $(\mX,\mY)\mapsto (a,\mb)$. Applying these vector fields to $\mY$ gives the equality
\be \label{id:Upsilonfunda}
-\frac 14 \mY +\Upsilon=\frac 12 \ma\pa_\ma\mY +\frac 34 \mb\pa_\mb \mY.
\ee
Therefore, $\Upsilon$ is analytic since all other terms above are, and \fref{eq:selfsimstationary2} is in fact satisfied everywhere on the set $\{0<\mY<2\mY^*(\mX)\}$. Let us show $\Upsilon=-\pa_{\mY}^{-1}\pa_\mX \T$. First, we show that the right hand side is well-defined. Recall $\pa_\mX \T=-\pa_\mb \mY$ from \fref{id:jacobianvolume}. Differentiating \fref{def:Phi}, using \fref{id:psi1implicit}, gives:
\be \label{id:paamYpabmY}
\pa_\ma \mY=-6p^*(1+3p^{*2}\ma^2)\int_{-\infty}^\mb g(p^*(\mX-\tmb^2))d\tmb, \quad \pa_\mb \mY=\frac{1}{1+3p^{*2}\ma^2}-12p^*\mb\int_{-\infty}^\mb g(p^*(\mX-\tmb^2))d\tmb
\ee
where $g(z)=(\frac{\Psi_1 \Psi_1'}{(1+3\Psi_1^2)^2})(z)$. Fix $\mX \in \mathbb R$ and let $\mY\rightarrow 0$. The constraint $\ma+p^{*2}\ma^3+\mb^2=\mX$ and \fref{def:Phi} imply that $\ma,\mb\rightarrow -\infty$ with $|\ma|^3\approx \mb^2$. Using the bounds $g(z)=O(z^{-5/3})$ as $z\rightarrow \infty$ from \fref{id:psi1infinity}, we get from \fref{id:paamYpabmY} that $|\ma\pa_\ma \mY|+|\mb\pa_\mb\mY|\rightarrow 0$ and $\pa_\mX \T=\pa_\mb \mY \rightarrow 0$ as $\mY \rightarrow 0$. Thus $\pa_{\mY}^{-1}\pa_\mX \T$ is well defined, and injecting this in \fref{id:Upsilonfunda} gives $\Upsilon(\mX,\mY)\rightarrow 0$ as $\mY \rightarrow 0$. Since $\pa_\mY \Upsilon=-\pa_\mX \T$, we obtain
$$
\Upsilon=\pa_{\mY}^{-1}\pa_\mX \T.
$$
The fact that $(\T,\Upsilon)$ solves \fref{eq:selfsimstationary2} with $\Upsilon$ given as above implies that $(T-t)^{1/2}\T(x/(T-t)^{3/2},y/(T-t)^{-1/4})$ solves \fref{2DPrandtlphom}, and (ii) in Proposition \ref{pr:selfsimfunda} is proved.

 \end{proof}

\begin{proof}[Proof of (iii) and (iv) in Proposition \ref{pr:selfsimfunda}]

\textbf{Step 1} \emph{Symmetry}. From \fref{def:Phi} and \fref{eq:def Y*} we have the following symmetries: that $\Phi_1(\ma,-\mb)=\Phi_1(\ma,\mb)$ and that $\Phi_2(\ma,\mb)-\mathcal Y^*(\Phi_1(\ma,\mb))=\mathcal Y^*(\Phi_1(\ma,\mb))-\Phi_2(\ma,-\mb)$. This implies that $\Phi^{-1}(\mX,\mY^*(\mX)+\mY)=(\Phi_1^{-1}(\mX,\mY^*(\mX)-\mY),-\Phi^{-1}_2(\mX,\mY^*(\mX)-\mY))$. This implies in particular (iii).\\

\noindent \textbf{Step 2} \emph{The set of zero vorticity}. Since the mapping $(\ma,\mb)\mapsto (\mX,\mY)$ is volume preserving, \fref{id:jacobianvolume} gives $\pa_\mb \mX=\pa_\mY\Theta$. Since $\pa_\mb \mX=2\mb$ from \fref{def:Phi}, the zero set of $\pa_\mY\Theta$ is $\{\mb=0\}$ or equivalently $\{\mY=\mY^*\}$. On this set, $\mX=\ma+p^{*2}\ma^3$, so that $\Theta=-\ma=\Psi_1(p^* \mX)/p^*$ applying \fref{id:psi1implicit}.\\

\noindent \textbf{Step 3} \emph{Taylor expansion}. This is a direct consequence of the Taylor expansion of the fonction $\Phi$ at $(0,0)$. Let us write $\Theta_{ij}=\pa_\mX^i\pa_\mY^j \Theta (0,\mY^*(0))$. We first look at the set $\{ \mX=0=\ma+\mb^2+p^{*2}\ma^3\}$, or equivalently $\ma=\Psi_1(p^*\mb^2)/p^*$ from \fref{id:psi1implicit}. One has from the Taylor expansion \fref{id:psi1origin} of $\Psi_1$:
$$
\mY= \int_{-\infty}^\mb \frac{d\tmb}{1+3\Psi_1^2 \left(p^*\tmb^2\right)} =\mY^*(0)+\mb+O(|\mb|^5).
$$
So that $\mb(0,\mY^*(0)+\mY)=\mY+O(|\mY|^5)$ as $\mY\rightarrow 0$. As $\ma=\Psi_1(p^*\mb^2)/p^*=-\mb^2+O(\mb^6)$ from \fref{id:psi1origin}, and $\ma=-\Theta$, one deduces the information on the vertical derivatives of $\Theta$:
\be \lab{eq:thetataylorproof1}
\Theta(0,\mY^*(0))=0, \ \ \Theta_{01}=0, \ \ \Theta_{02}=2, \ \ \Theta_{03}=0.
\ee
Let us secondly look at the set $\{\mb=0\}$, corresponding to $\{\mY=\mY^*(\mX)\}$. Recall that there holds $\pa_\mY \Theta (\mX,\mY^*(\mX))=0$ from Step 2. Differentiating with respect to $\mX$ once and twice this identity, setting $\mX=0$, using the Taylor expansion \fref{exp Y 0} of $\mY^*$ and \fref{eq:thetataylorproof1} gives:
\be \lab{eq:thetataylorproof22}
\Theta_{11}=-2, \ \ \Theta_{21}+2c_2+2\Theta_{12}=0.
\ee
Still on the set $\mb=0$, one has $\Theta=-p^*\Psi_1(\mX)/p^*$. We differentiate once, twice and three times this identity with respect to $\mX$, and set $\mX=0$. Using \fref{id:psi1origin}, \fref{eq:thetataylorproof1}, \fref{eq:thetataylorproof22} and \fref{exp Y 0} this gives:
$$
\Theta_{10}=-1, \ \ \Theta_{20}=2, \ \ \Theta_{30}+3\Theta_{21}+3\Theta_{12}=6p^{*2}.
$$
We need one last information. We take the identity $\pa_\mX \ma=\pa_\mb \mY$ from \fref{id:jacobianvolume}. Consider the set $\{\mX=\ma+\mb^2+p^{*2}\ma^3=0\}$. Then on this set from \fref{def:Phi}, since $\Psi_1(0)=0$ and $\Psi_1'=-1/(1+3\Psi_1^2)$:
\begin{align*}
\frac{\pa \mY}{\pa \mb} &=\frac{1}{1+3p^{*2}\ma^2}-\mb\int_{-\infty}^\mb \frac{12 p^* \Psi_1(p^*\tmb^2)}{(1+3\Psi_1^2(p^*\tmb^2))^3}d\tmb =1+O(\mb^4)-\mb\int_{-\infty}^0 \frac{12 p^* \Psi_1(p^*\tmb^2)}{(1+3\Psi_1^2(p^*\tmb^2))^3}d\tmb+O(\mb^3) \\
 &\qquad =  1-\mb\sqrt{p^*} \int_{0}^{\infty} \frac{1}{\left(\sqrt{ \frac x 3}+\sqrt{\frac x 3}^3 \right)^{\frac 12} (1+x)^2}d\tmb+O(\mb^3)
\end{align*}
where we changed variables $x=3p\Psi_1(p\tmb^2)$. Hence as $\mb(0,\mY^*(0)+\mY)=\mY+O(|\mY|^5)$ on this set, $\pa_\mX \Theta(0,\mY^*(0)+\mY)=-1+C\mY+O(|\mY|^3)$. This implies $\Theta_{12}=0$, and we obtain the desired Taylor expansion \fref{taylorthetalag} for $\Theta$ using the previous information.

\end{proof}

To finish the proof of Proposition \ref{pr:selfsimfunda}, there remains to prove item (v). We first establish that $\Theta$ solves a \emph{local} ODE in the Lemma below. This was, before our discovery of formula \fref{def:thetageom}, the way the existence and some properties of the profile $\Theta$ had been showed.

\begin{lemma}[\cite{ESC83,CSW96}] \label{lem:ode}
$\Theta$ defined by \fref{def:thetageom} solves:
\be \lab{eq:Thetabis1}
\left\{\begin{array}{ll}
&\frac{\pa \Theta}{\pa \mY}= -2\sqrt{\mX+\Theta+p^{*2}\Theta^3},\ \ \mbox{ for } 0<\mY\leq \mY^*(\mX),\\
&\frac{\pa \Theta}{\pa \mY}= 2\sqrt{\mX+\Theta+p^{*2}\Theta^3},\ \ \mbox{ for } \mY^*\leq \mY<2\mY^*(\mX).
\end{array}
\right.
\ee
\end{lemma}

\begin{proof}[Proof of Lemma \ref{lem:ode}]

From \fref{id:jacobianvolume} we have that $\pa_\mY \T=\pa_\mb \mX$. From \fref{def:Phi}, $\pa_\mb \mX=2\mb$ and $\mb^2=\mX-\ma-p^{*2}\ma^3$. Since $\ma=-\T$ this gives:
$$
(\pa_\mY \T)^2=4(\mX+\T+p^{*2}\T^3).
$$
Since $\pa_\mY \T=2\mb$, we get $\pa_\mY \T\geq0$ for $\mb\geq0$, and $\pa_\mY \T\leq 0$ for $\mb\leq 0$. As $\mb\geq0$ and $\mb\leq 0$ correspond respectively to $\mY^*(\mX)<2\mY^*$ and $0<\mY<\mY^*(\mX)$ from \fref{def:Phi} and \fref{eq:def Y*}, we obtain \fref{eq:Thetabis1} from the above equation and these sign considerations.

\end{proof}

\begin{proof}[Proof of (v) in Proposition \ref{pr:selfsimfunda}]

\noindent \textbf{Step 1} \emph{Proof of \fref{expanboundary} and \fref{expanboundary12}} Let $\mX\in \mathbb R$ be fixed. Using the ODE's \fref{eq:Thetabis1}, we express $\mY$ as a function of $\T$ and expand for $\T \gg \langle \mX\rangle^{1/3}$:
\be
\label{eq:mYTgen} \mY=\int_\T^\infty\frac{d\tilde \T}{2\sqrt{\mX+\tilde \T+p^{*2}\tilde \T^3}}=\int_\T^\infty d\tilde \T\left(\frac{1}{2p^*\tilde \T^{\frac 32}}+O\left(\frac{|\mX|}{\tilde \T^{\frac 92}}+\frac{1}{\tilde \T^{\frac 72}} \right) \right)=\frac{1}{p^{*2}\T^{\frac 12}}+O\left(\frac{|\mX|}{ \T^{\frac 72}}+\frac{1}{ \T^{\frac 52}}\right) 
\ee
which after inversion gives:
$$
\T=-\frac{1}{p^{*2}\mY^2}+O\left(|\mX|\mY^4+|\mY|^2 \right)
$$
which is exactly \eqref{expanboundary}. Note that we have only established it for $\T \gg \langle \mX\rangle^{1/3}$, which from the above expansion, corresponds to $\mY\ll \langle  \mX\rangle^{-1/6}$. The validity of the $O()$ in the whole region $\{\mY\leq \mY^*(\mX)\}$ is a consequence of step 2 below. Note finally that \eqref{expanboundary12} is obtained from \eqref{expanboundary} using the symmetry (iii) in Proposition \ref{pr:selfsimfunda}.\\

\noindent \textbf{Step 2} \emph{Proof of \fref{expanboundary2}}. We only prove the expansion in the case $\mX\rightarrow -\infty$, as the other case can be handled with similar ideas. Let $\psi_-:[p^{*-2/3},\infty)$ be the following function:
$$
 \psi_- (z)= \frac 12 \int_{z}^{\infty} \frac{dz}{\sqrt{-1+p^{*2}z^3}}.
$$
By definition \fref{id:cpm} of $C_-$, the range of $\psi_-$ is $(0,C_-]$. Denote by $\psi_-^{-1}:(0,C-]\rightarrow [p^{*-2/3},\infty)$ its inverse, and define $\varphi_-$ to be its extension on $(0,2C_-)$ by even symmetry about $C_-$:
$$
\varphi_-(\tilde \mY):=\left\{ \begin{array}{l l}  \psi^{-1}(\mY) \mbox{ if } 0<\tilde \mY\leq C_-,  \\  \psi^{-1}(2C_--\tilde \mY) \mbox{ if } C_-\leq \tilde \mY< 2C_-. \end{array} \right.
$$
The properties of $\varphi_-$ listed in (v) of Proposition \ref{pr:selfsimfunda} are verified by a direct check. It remains to prove the convergence. For this we use the first equality in \fref{eq:mYTgen}. Let first $0<\mY\leq \mY^*(\mX)$, then:
$$
\mY(\T)=\frac 12 \int_{\T}^{\infty} \frac{d\tilde \T}{\sqrt{\mX+\tilde \T+p^{*2}\tilde \T^3}}.
$$
We change variables and set $\tilde \T=zp^{*-1/3}\Psi_1(p^*\mX)$, and use $\mX=-\Psi_1(p^*\mX)/p^*-\Psi_1^3(p^*\mX)/p^*$ so that after some rewriting:
\bee
&& \frac{\Psi_1^{\frac 12}(p^*\mX) }{p^{*\frac 16}}\mY(\T)=\frac 12 \int_{\frac{p^{*1/3}\T}{\Psi_1(p^*\mX)}}^{\infty} \frac{dz}{\sqrt{-1+p^{*2}z^3}}\left(1+\frac{1}{z^2p^{*\frac 43}+zp^{*\frac 23}+1}\Psi_1^{-2}(p^*\mX)\right)^{-\frac 12}\\
&=& \psi_-\left(\frac{p^{*\frac{1}{3}}\T}{\Psi_1(p^*\mX)}\right)-\frac 12 \int_{\frac{p^{*1/3}\T}{\Psi_1(p^*\mX)}}^{\infty} \frac{dz}{\sqrt{-1+p^{*2}z^3}}\left(1-\left(1+\frac{1}{z^2p^{*\frac 43}+zp^{*\frac 23}+1}\Psi_1^{-2}(p^*\mX)\right)^{-\frac 12}\right)\\
&=& \psi_-\left(\frac{p^{*\frac{1}{3}}\T}{\Psi_1(p^*\mX)}\right)+O\left(|\mX|^{-\frac 23}\left(\frac{\T}{\Psi_1(p^*\mX)} \right)^{-\frac 52}\right)
\eee
where we used \fref{id:psi1infinity} and that $\T\geq \Psi_1(p^*\mX)/p^*$ for all $\mY<\mY^*(\mX)$. Given that uniformly on $[p^{*-2/3},\infty)$ one has $|\psi_-'|\approx (z-p^{*-2/3})^{-1/2}z^{-1}$, we can invert the above equation for all $\T$ such that $\frac{p^{*1/3}\T}{\Psi_1(p^*\mX)}-p^{*-2/3}\gg |\mX|^{-4/3}$, with:
$$
\frac{p^{*\frac{1}{3}}\T}{\Psi_1(p^*\mX)}=\psi^{-1}\left(\frac{\Psi_1^{\frac 12}(p^*\mX) }{p^{*\frac 16}}\mY(\T) \right)+O\left(|\mX|^{-\frac 23}\left(\frac{\T}{\Psi_1(p^*\mX)}\right)^{-1} \right)
$$
which, given that $\psi_-^{-1}(\tilde \T)\approx \tilde \T^{-2}$ uniformly on $(0,C_-]$, and that $\Psi_1(\mX)=|\mX|^{1/3}+O(|\mX|^{-1/3})$ gives finally:
\bee
\T&=& p^{*\frac{1}{3}}\Psi_1(p^*\mX) \psi^{-1}\left(\frac{\Psi_1^{\frac 12}(p^*\mX) }{p^{\frac 16}}\mY \right)+p^{*\frac{1}{3}}\Psi_1(p^*\mX)O\left(|\mX|^{-\frac 13}|\mY|^{2} \right) \\
&=&|\mX|^{\frac 13}\varphi_-\left(|\mX|^{\frac 13}\mY \right)+O\left(|\mX|^{-\frac 23}\mY^{-2} \right)
\eee
where we used the fact that $|\mY|\lesssim |\mX|^{-1/6}$ as $\mX\rightarrow -\infty$ and $\mY\leq \mY^*$. For $\mY^*<\mY\leq (2-\epsilon)\mY^*$ we write:
\bee
\mY & = & \mY^*+\frac{p^{*\frac 16}}{ \Psi_1^{\frac 12}(p^*\mX) } \frac 12 \int_{\frac{\Psi_1(p^*\mX)}{p^*}}^{\T} \frac{d\tilde \T}{\sqrt{\mX+\tilde \T+p^{*2}\tilde \T^3}}
\eee
so that changing variables with $\tilde \T=zp^{*-1/3}\Psi_1(p^*\mX)$ as previously:
\bee
\frac{\Psi_1^{\frac 12}(p^*\mX) }{p^{*\frac 16}}\mY&=&\frac{\Psi_1^{\frac 12}(p^*\mX) }{p^{*\frac 16}} \mY^*+\frac 12 \int_{1}^{\frac{p^{*1/3}\T}{\Psi_1(p^*\mX)}} \frac{dz}{\sqrt{-1+p^{*2}z^3}}\left(1+\frac{1}{z^2p^{*\frac 43}+zp^{*\frac 23}+1}\Psi_1^{-2}(p^*\mX)\right)^{-\frac 12} \\
&=& C_-+\frac 12 \int_{1}^{\frac{p^{*1/3}\T}{\Psi_1(p^*\mX)}} \frac{dz}{\sqrt{-1+p^{*2}z^3}}+O(|\mX|^{-\frac 23}).
\eee
As we are restricting to the range $\mY^*<\mY\leq (2-\epsilon)\mY^*$ the above equation, using the asymptotic behaviour of $\varphi_-$, gives:
\bee
\T=\frac{\Psi_1(p\mX)}{p^{\frac 13}}\varphi_-\left( \frac{\Psi_1^{\frac 12}(p\mX) }{p^{\frac 16}}\mY  \right)+O\left(|\mX|^{-1}\T^{2}\right)=|\mX|^{\frac 13} \varphi_-\left(|\mX|^{\frac 16}\mY  \right)+O\left(\frac{\mX^{-1}}{(2\mY^*(\mX)-\mY)^4}\right).
\eee
This shows the desired asymptotic behaviour (v) in Proposition \ref{pr:selfsimfunda} at $-\infty$. The behaviour at $\infty$ can be proved along similar lines.

\end{proof}

\subsection{Degenerate self-similar profile} \lab{subsec:degen}

In this subsection, we prove Proposition \ref{pr:selfsimdegen}. In order to simplify notations, we drop the prime ' notation for $\Phi'$, $\T'$ and $\mY^{'*}$ and simply write $\Phi$, $\T$ and $\mY^*$ instead.

We construct an odd in $\mX$ self-similar profile that is the two dimensional version of the profile found for the full viscous Prandtl system in \cite{CGIM} on the transversal axis.  Again, we proceed in two parts. In the first part, we use a Lagrangian approach, and the explicit formula \fref{selfsimtildetheta2} to study the curve $\mY^*$, the symmetries of $\T$, and the set of zero vorticity, and to prove it defines a self-similar profile.

In the second part, we perform an Eulerian study of the self-similar equation. Namely, we solve \fref{eq:autosimstattransformee2}. This allows us to prove the analyticity at the boundary of $\T$, and to study the derivatives on the axis. This shows how the two studies of \fref{eq:autosimstattransformee1} and \fref{eq:autosimstattransformee2} can complement one another.

\begin{proof}[Proof of (i) in Proposition \ref{pr:selfsimdegen}]

The fact that $\Phi$ preserves volume is a direct computation from the formula \fref{def:tildetheta}. Its analyticity and that of $\mathcal Y^*$ are direct consequences of the analyticity of $\Psi_1$. By fixing $\mX \in \mathbb R$, and so fixing the relation $\ma+\ma^3+\mb^2\ma/4=\mX$, the vertical component of the image is:
$$
\mY(\ma,\mb)=2\int_{-\infty}^{\frac \mb2} \left(1+\tmb^2\right)^{-1}\left(1+3\Psi_1^2 \left(\frac{\mX}{(1+\tmb^2)^{3/2}}\right)\right)^{-1}d\tmb .
$$
Hence, using the formula \fref{def:tildeY*}, the set $\{ \Phi(\ma,\mb), \ \Phi_1=\mX\}$ consists of the interval $(0<\mY<2\mY^*(\mX))$ where $\mY^*$ is indeed defined by \fref{def:tildeY*}. The range of $\Phi$ is thus the set $\{0<\mY<2 \mY^*(\mX) \}$.

For $\mX>0$, we change variables twice in \fref{def:tildeY*}, first with $\Theta=(1+\tmb^2)^{1/2}\Psi_1 ( \mX/(1+\tmb^2)^{3/2})$ using \fref{id:psi1implicit}, and then with $\Theta=z\Psi_1(\mX)$ to get:
$$
\lab{idY*}  \mY^*(\mX) =\int_{\Psi_1(\mX)}^{0} \frac{dz}{\sqrt{|\Theta|} \sqrt{\mX+\Theta+\Theta^3}}= \sqrt{\frac{-\Psi_1(\mX)}{\mX}} \int_0^1 \frac{d z}{\sqrt{z}\sqrt{1+ z \frac{\Psi_1(\mX)}{\mX}+ z^3 \frac{\Psi_1^3(\mX)}{\mX}}} .
$$
The expansion of $\mY^*$ near the origin \fref{exp tildeY 0} then comes as a direct consequence of \fref{id:psi1origin} and of:
$$
\int_0^1 \frac{dz}{\sqrt{z}\sqrt{1-z}} =B(\frac 12,\frac 12)=\pi, \ \ \int_0^1 \frac{\sqrt{z}(1+z)}{\sqrt{1-z}}dz =\frac{7}{8}\pi.
$$
We now turn to the expansion at infinity. We write:
$$
\mY^*(\mX) = \sqrt{\frac{-\Psi_1(\mX)}{\mX}} \int_0^1 \frac{dz}{\sqrt{z}\sqrt{1+z \frac{\Psi_1(\mX)}{\mX}+z^3 \frac{\Psi_1^3(\mX)}{\mX}}} = \sqrt{\frac{-\Psi_1(\mX)}{\mX}} \int_0^1 \frac{(1+g(\mX,z))^{-1/2} }{\sqrt{z}\sqrt{1-z^3}}dz
$$
where using that $\Psi_1(\mX)+\Psi_1^3(\mX)+\mX=0$:
$$
g(\mX,z)= \frac{z\frac{\Psi_1 (\mX)}{\mX}+z^3 \left( \frac{\Psi_1^3(\mX)}{\mX}+1\right)}{1-z^3}= \frac{z-z^3}{1-z^3}\frac{\Psi_1 (\mX)}{\mX}=\frac{z (1+z)}{1+z+z^2}\frac{\Psi_1 (\mX)}{\mX} \ = \ O(\mX^{-\frac 23})
$$
uniformly on $[0,1]$ as $\mX \rightarrow \infty$ from \fref{id:psi1infinity}. This, the integral value $\int_0^1 \frac{dz}{\sqrt{z}\sqrt{1-z^3}} = 2\sqrt{\pi}\frac{\Gamma (7/6)}{\Gamma (2/3)}$ and \fref{id:psi1infinity} then imply \fref{exp Y infty} as $\mX \rightarrow \infty$. The same limit holds at $ -\infty$ as $\mY^*$ is an even function.
\end{proof}

Now that $\T$ is well defined, we can study some properties directly from its formula.

\begin{proof}[Proof of (iv) and (v) in Proposition \ref{pr:selfsimdegen}]

\textbf{Step 1} \emph{Symmetries}. We have the first symmetry that $\Phi_1(\ma,-b)=\Phi_1(\ma,\mb)$ and that $\Phi_2(\ma,\mb)-\mathcal Y^*(\Phi_1(\ma,\mb))=\mathcal Y^*(\Phi_1(\ma,\mb))-\Phi_2(\ma,-\mb)$, from \fref{def:tildetheta} and \fref{def:tildeY*}. This implies $\Phi^{-1}(\mX,\mY^*(\mX)+\mY)=(\Phi_1^{-1}(\mX,\mY^*(\mX)-\mY),-\Phi^{-1}_2(\mX,\mY^*(\mX)-\mY))$, hence $\Theta(\mX,\mY^*-\mY)=\Theta(\mX,\mY^*+\mY)$ for $\mY<\mY^*$.

We have the second symmetry that $\Phi_1(-\ma,\mb)=-\Phi_1(\ma,\mb)$ and that $\Phi_2(-\ma,\mb)=\Phi_2(\ma,\mb)$. This implies that $\Phi^{-1}(-\mX,\mY)=(-\Phi_1^{-1}(\mX,\mY),\Phi^{-1}_2(\mX,\mY))$, so that $\Theta$ is odd in $\mX$.

Finally, $\Phi_1=\mX$ is negative on $\{\ma<0\}$, positive on $\{\ma>0\}$. Hence, on the set $\{0<\mY<\mY^*(\mX)\}$, $\T=-\ma$ is positive for $\mX<0$ and negative for $\mX>0$. This concludes the proof of (iv).\\

\noindent \textbf{Step 2} \emph{The set of zero vorticity}. We assume the regularity properties of (iii) and the identity \fref{eq:formulederiveeaxe} of (vi) in Proposition \ref{pr:selfsimdegen}, which are proved later on in this Subsection. Since $\T=0$ on the set $\{\mY\geq 2\mY^*(\mX)\}$, and is $C^1$, we deduce that $\pa_\mY \T=0$ on $\{\mY\geq 2\mY^*(\mX)\}$. Using the symmetry property (iv) proved in Step 1, this shows that  $\pa_\mY \T=0$ as well at the boundary $\{\mY=0\}$.

We now restrict ourselves to the set $\{0<\mY<2\mY^*(\mX)\}$. Since the mapping $(\ma,\mb)\mapsto (\mX,\mY)$ is volume preserving, one inverts the Jacobian matrix to find:
\be \label{id:jacobianvolumepres}
\begin{pmatrix} \frac{\pa \mY}{\pa \mb} &- \frac{\pa \mX}{\pa \mb} \\ -\frac{\pa \mY}{\pa \ma} & \frac{\pa \mX}{\pa \ma}  \end{pmatrix}=\begin{pmatrix} \frac{\pa \ma}{\pa \mX} & \frac{\pa \ma}{\pa \mY} \\ \frac{\pa \mb}{\pa \mX} & \frac{\pa \mb}{\pa \mY}  \end{pmatrix}.
\ee
One has $\pa_\mb \mX=\mb\ma/2$ from \fref{def:tildetheta}, so that $\pa_\mY \Theta (\mX,\mY^*(\mX))=-\mb\ma/2$. Hence the set where $\pa_\mY \Theta$ is zero is equal to $\{\ma=0\}\cup \{\mb=0 \}$, which corresponds to $\{\mX=0\}\cup\{\mY=\mY^*(\mX)\}$. On the set $\{\mY=\mY^*(\mX)\}=\{\mb=0\}$, we have $\ma+\ma^3=\mX$ so that $\T(\mX,\mY^*(\mX))=-\ma=\Psi_1(\mX)$ using \fref{id:psi1implicit}.

On the line $\{\mY=\mY^*(\mX)\}$, as $\T(\mX,\mY^*(\mX))=\Psi_1(\mX)$ and $\pa_\mY \Psi_1(\mX,\mY^*(\mX))=0$, one gets that $\pa_\mX \T(\mX,\mY^*(\mX))=\pa_\mX \Psi_1(\mX)$. Hence, as $\pa_\mX \Psi_1$ attains its minimum at the origin, the minimum of $\pa_\mX \T$ on $\{\mY =\mY^*\}$ is attained at $(0,\mY^*(0))=(0,\pi)$. From \fref{eq:formulederiveeaxe}, as $\pa_\mX\T=-\sin^2(\mY/2)1_{0\leq \mY\leq 2\pi}$, the minimum on the vertical axis is also attained at $(0,\pi)$.

We now show the Taylor expansion \fref{eq:taylortildetheta}. We write $\Theta_{ij}=\pa_\mX^i\pa_\mY^j \Theta (0,\pi)$. As $\T$ is odd in $\mX$, we obtain that:
$$
\T_{00}=\T_{01}=\T_{02}=\T_{03}=0, \ \ \T_{20}=\T_{21}=0
$$
We proved above that $\pa_\mY \Theta (\mX,\mY^*(\mX))=0$. Differentiating with respect to $\mX$ this identity, setting $\mX=0$, using the Taylor expansion \fref{exp tildeY 0} of $\mY^*$ and the coefficients computed above:
$$
\Theta_{11}=0.
$$
We differentiate once and three times this identity with respect to $\mX$, and set $\mX=0$. Using $\Psi_1(\mX)=-\mX+\mX^3+O(|\mX|^5)$, the coefficients computed above and the Taylor expansion of $\mY^*$ at $0$ \fref{exp tildeY 0} this gives:
$$
\Theta_{10}=-1, \ \ \Theta_{30}=6.
$$
We now set $\ma=0$. Then:
$$
\mY(0,\mb)=\int_{-\infty}^\mb \frac{d\tmb}{1+\frac{\tmb^2}{4}}
$$
so that $\pa_\mb \mY(0,0)=1$, $\pa_{\mb\mb}\mY(0,0)=0$ and $\pa_{\mb\mb\mb}\mY(0)=-\frac 12$. Inverting this relation one gets $\pa_\mY \mb (0,\pi)=1$, $\pa_{\mY\mY} \mb (0,\pi)=0$. We now use the relation \fref{id:jacobianvolumepres} to write $\pa_{\mX\mY \mY}\ma(0,\pi)=\pa_{\mY\mY_{|\mX}}(\pa_\mb \mY)(0,0)$, which, injecting all values of the coefficients already found, gives:
$$
\T_{12}=2
$$
and ends the proof of \fref{eq:taylortildetheta}. This closes the proof of (v).

\end{proof}

\begin{proof}[Proof of (ii) in Proposition \ref{pr:selfsimdegen}]

This is, as in the previous Subsection, a consequence of (ii) in Lemma \ref{lem:resolutioneqprofiles}. First, an elementary computation gives that $\mX$ given by \fref{def:tildetheta} solves \fref{eq:autosimstattransformee1} with $\alpha=3/2$ and $\beta=-1/2$. We have already proved that the mapping $(\mX,\mY)\mapsto (a,\mb)$ is volume preserving. Applying Lemma \ref{lem:resolutioneqprofiles}, there exists a function $\Upsilon$ on the set $\{(\mX,\mY), \T(\mX,\mY)\neq 0\}$ such that $(\T,\Upsilon)$ solves \fref{eq:selfsimstationary2}. Moreover, from Step 1 in the proof of Lemma \ref{lem:resolutioneqprofiles}, we have that $\frac 12 \ma\pa_\ma +\frac 12 \mb\pa_\mb$ is the pushforward of $(\frac 32 \mX+\T)\pa_\mX  +(-\frac 12 \mY +\Upsilon )\pa_\mY)$ by the mapping $(\mX,\mY)\mapsto (\ma,\mb)$. We apply both vector fields to $\mY$, and the equality of the results shows
\be \label{id:Upsilonfunda2}
-\frac 12 \mY +\Upsilon=\frac 12 \ma\pa_\ma\mY +\frac 12 \mb\pa_\mb \mY.
\ee
The change of variables $(\ma,\mb)\mapsto (\mX,\mY)$ being analytic, $\Upsilon$ given by the above formula is also analytic. The identity \fref{eq:selfsimstationary2} is in fact satisfied everywhere on the set $\{0<\mY<2\mY^*(\mX)\}$. We now show $\Upsilon=-\pa_{\mY}^{-1}\pa_\mX \T$. Recall $\pa_\mX \T=-\pa_\mb \mY$ from \fref{id:jacobianvolume}. Differentiating \fref{def:tildetheta}, using \fref{id:psi1implicit}:
\begin{align} \label{id:paamYpabmY2}
&\pa_\ma \mY=-12(1+3\ma^2+\frac{\mb^2}{4})\int_{-\infty}^{\frac \mb2} \frac{d\tmb}{(1+\tmb^2)^{\frac 52}}g\left(\frac{\mX}{(1+\tmb^2)^{\frac 32}} \right)d\tmb,\\
\label{id:paamYpabmY3} & \pa_\mb \mY=\frac{1}{(1+\frac{\mb^2}{4})(1+3\Psi_1^2\left(\frac{\mX}{(1+\mb^2/4)^{3/2}}\right)}-6\mb\ma\int_{-\infty}^{\frac \mb2} \frac{d\tmb}{(1+\tmb^2)^{\frac 52}}g\left(\frac{\mX}{(1+\tmb^2)^{\frac 32}} \right)d\tmb
\end{align}
where $g(z)=(\frac{\Psi_1 \Psi_1'}{(1+3\Psi_1^2)^2})(z)$. We now fix $\mX \in \mathbb R$ and let $\mY\rightarrow 0$. The constraint $\ma+\ma^3+\mb^2\ma/4=\mX$ and \fref{def:Phi} imply that $\mb\rightarrow -\infty$ and $\ma\rightarrow 0$ with $|\mb|\approx |\ma|^{-1/2}$. Using that $g$ is bounded, we get from \fref{id:paamYpabmY} that $|\ma\pa_\ma \mY|+|\mb\pa_\mb\mY|\rightarrow 0$ and $\pa_\mX \T=\pa_\mb \mY \rightarrow 0$ as $\mY \rightarrow 0$. Thus $\pa_{\mY}^{-1}\pa_\mX \T$ is well defined, and injecting this in \fref{id:Upsilonfunda2} gives $\Upsilon(\mX,\mY)\rightarrow 0$ as $\mY \rightarrow 0$. Since $\pa_\mY \Upsilon=-\pa_\mX \T$, we obtain
$$
\Upsilon=\pa_{\mY}^{-1}\pa_\mX \T.
$$
The fact that $(\T,\Upsilon)$ solves \fref{eq:selfsimstationary2} with $\Upsilon$ given as above implies that $(T-t)^{1/2}\T(x/(T-t)^{3/2},y/(T-t)^{-1/2})$ solves \fref{2DPrandtlphom}, and (ii) in Proposition \ref{pr:selfsimdegen} is proved.

\end{proof}

We can now end the proof of Proposition \ref{pr:selfsimdegen} which we will do by solving equation \fref{eq:autosimstattransformee2}.

\begin{lemma}
Keeping writing $\T$ instead of $\tilde \T$ for the function defined by \fref{selfsimtildetheta2}, it satisfies:
\be \lab{odetildetheta}
\left\{\begin{array}{lll}
& \pa_\mY \T = \sqrt{\T}\sqrt{-\T^3-\T-\mX}  & \mbox{ for }  \mX<0 \ \text{and} \ 0\leq \mY\leq \mY^*(\mX) \\
& \pa_\mY \T = -\sqrt{\T}\sqrt{-\T^3-\T-\mX}  & \mbox{ for }  \mX<0 \ \text{and} \  \mY^*(\mX)\leq \mY \leq 2\mY^*(\mX) \\
& \pa_\mY \T = -\sqrt{-\T}\sqrt{\T^3+\T+\mX}  & \mbox{ for }  \mX>0 \ \text{and} \ 0\leq \mY\leq \mY^*(\mX) \\
& \pa_\mY \T = \sqrt{\T}\sqrt{\T^3+\T+\mX}  & \mbox{ for }  \mX>0 \ \text{and} \  \mY^*(\mX)\leq \mY \leq 2\mY^*(\mX) \\
\end{array} \right.
\ee
\end{lemma}

\begin{proof}

From \fref{id:jacobianvolume} we have that $\pa_\mY \T=\pa_\mb \mX$. From \fref{def:Phi}, $\pa_\mb \mX=\mb\ma/2$ and $\ma^2\mb^2=4\ma(\mX-\ma-\ma^3)$. Since $\ma=-\T$ this gives:
$$
(\pa_\mY \T)^2=-\T (\mX-\T-\T^3).
$$
Using $\pa_\mY \T=\mb\ma/2$ and $\ma=-\T$, we get the signs of $\ma$ and $\pa_\mY \T$ via (iv) in Proposition \ref{pr:selfsimdegen} that has already been proved. Combined with the fact that $\mb<0$ and $\mb>0$ correspond respectively to $\mY^*(\mX)<2\mY^*$ and $0<\mY<\mY^*(\mX)$ from \fref{def:tildetheta} and \fref{def:tildeY*}, we obtain \fref{eq:Thetabis1} by solving the above equation for $\pa_\mY \T$.

\end{proof}

With the local ODEs \fref{odetildetheta}, we can now end the proof of Proposition \ref{pr:selfsimdegen}.

\begin{proof}[Proof of (iii), (vi) and (vii) in Proposition \ref{pr:selfsimdegen}]

\noindent \textbf{Step 1} \emph{Analyticity at the boundary}. The analyticity of $\T$ in the interior of its domain could be proved by studying the ODEs \fref{odetildetheta}, but note that it is a direct consequence of the formula \fref{selfsimtildetheta2}, as the diffeomorphism defined by \fref{def:tildetheta} is analytic. We now prove the existence of an analytic extension by solving periodic manner in $\mY$ the ODEs \fref{odetildetheta}. To do this, we will show that these ODEs can be used to prove the analyticity at the boundary $\{\mY=0\}$ (where the natural extension of $\T$ is $0$).

We will establish that an extension of the mapping $(\mX,\sqrt{-\Theta})\mapsto (\mX,\mY)$ is an analytic diffeomorphism, implying the result by taking the inverse transformation. Without loss of generality we consider the case $\mX>0$. From \fref{odetildetheta} we infer that there is a one-to-one relation between $(\mX,\mY)$ and $(\mX,\Theta(\mX,\mY))$ between the sets $\{\mX>0, \ 0\leq \mY \leq \mY^*\}$ and $\{ \mX>0, \ \Psi_1(\mX)\leq \Theta \leq 0\}$ with the formula
$$
\mY(\mX,\Theta(\mX,\mY)) = \int_0^{-\Theta (\mX,\mY)}\frac{dz}{\sqrt{z}\sqrt{\mX-z-z^3}}.
$$
We change variables and define $U=\sqrt{-\Theta}$. Then the above formula becomes
$$
\mY(\mX,U) = 2 \int_0^{U}\frac{du}{\sqrt{\mX-u^2-u^6}}.
$$
This formula also makes sense for $-\sqrt{-\Psi_1(\mX)}\leq U<0$, since $\mX-u^2-u^6>0$ for $|u|\leq \sqrt{-\Psi_1(\mX)}$, and the mapping $(\mX,U)\mapsto (\mX,\mY)$ is one-to-one from $\{\mX>0, \ |U|< \sqrt{-\Psi_1(\mX)} \}$ onto $\{\mX>0, \ |\mY|< \mY^*\}$, such that its restriction to nonnegative $\mY$ satisfies $U(\mX,\mY)=\sqrt{-\Theta (\mX,\mY)}$. Let $\mX_0>0$. For $(\mX,u)$ close to $(\mX_0,0)$, the function under the integral sign is analytic
$$
\frac{1}{\sqrt{\mX-u^2-u^6}}=\sum_{\alpha \in \mathbb N^2} a_\alpha (\mX-\mX_0)^{\alpha_1}u^{2\alpha_2}
$$
with $a_{(0,0)}=1/\sqrt{\mX_0}>0$ and $a_{(0,1)}=1/(2\mX_0^{3/2})>0$, and hence the integral is also analytic
$$
\mY(\mX,U) =2 \sum_{\alpha \in \mathbb N^2} \frac{a_\alpha}{2\alpha_2+1} (\mX-\mX_0)^{\alpha_1}U^{2\alpha_2+1}.
$$
Therefore, the mapping $(\mX,U)\mapsto (\mX,\mY)$ is an analytic diffeomorphism near $(\mX_0,0)$. Its inverse is then analytic of the form (since odd in $\mY$):
$$
U(\mX,\mY)=\sum_{\alpha \in \mathbb N^2}b_\alpha (\mX-\mX_0)^{\alpha_1} \mY^{1+2\alpha_2}
$$
Since $U$ coincides with $\sqrt{-\Theta}$ for $\mY\geq 0$ this means that $\Theta$ is analytic of the form:
$$
\Theta=-U^2=\sum_{\alpha \in \mathbb N^2}c_\alpha (\mX-\mX_0)^{\alpha_1} \mY^{2\alpha_2+2}.
$$
This proves that $\tilde \T$, restricted to the set $\{0\leq \mY\leq 2\tilde \mY^*(\mX)\}$ is analytic in its interior, and also at the lower boundary $\{\mY=0\}$. The symmetry property (iv) in Proposition \ref{pr:selfsimdegen} implies that this restriction is also analytic at the upper boundary $\{\mY=2\mY^*(\mX)\}$. From the ODEs \fref{odetildetheta}, this restriction satisfies $\pa_{\mX}\T=\pa_{\mY}\T=0$ at the upper boundary $\{\mY=2\mY^*(\mX)\}$. Hence, extending $\T$ by $0$ on the set $\{\mY\geq 2\mY^*(\mX)\}$ maintains the $C^1$ regularity.

Finally, the other extension \fref{eq:analyticextensiondegenerate} is, by the above analysis of the ODEs \fref{odetildetheta}, analytic on the whole upper half space $\mH$. Since it solves \fref{eq:selfsimstationary2} with $\alpha=3/2$ and $\beta=-1/2$ on $\{0<\mY<2\mY^*(\mX) \}$ it thus solves it on the whole upper half space by uniqueness of analytic expansions, and is therefore also a self-similar profile. This ends the proof of (iii).\\

\noindent \textbf{Step 2} \emph{First derivative on the axis}. Let us consider the zone $\mX>0$ and $0<\mY<\mY^*$. One has the following convergence result
$$
\mY(\mX,\Theta ) = \int_0^{-\Theta }\frac{dz}{\sqrt{z}\sqrt{\mX-z-z^3}}=\int_0^{-\frac{\Theta}{\mX}} \frac{du}{\sqrt{u}\sqrt{1-u-\mX^2u^2}}\rightarrow \int_0^{-\pa_\mX \Theta (0,\mY)} \frac{du}{\sqrt{u}\sqrt{1-u}}
$$
as $\mX\rightarrow 0$. Hence on the vertical axis $\mY$ and $\pa_\mX\Theta (0,\mY)$ are linked by
$$
\mY=\int_0^{-\pa_\mX \Theta (0,\mY)} \frac{du}{\sqrt{u}\sqrt{1-u}}= 2\arcsin (\sqrt{-\pa_\mX \Theta (0,\mY)}).
$$
This proves (vi) upon inverting the $\arcsin$ function.\\

\noindent \textbf{Step 3} \emph{Third derivative on the axis} Let $\phi (\mY)=-\pa_\mX\Theta (0,\mY)$ and $\psi(\mY)=\pa_{\mX}^3 \Theta (0,\mY)$. Since $\Theta$ solves \fref{eq:selfsimstationary2} with $\alpha=3/2$ and $\beta=-1/4$ and vanishes on the vertical axis $\{ \mX=0\}$, differentiating 3 times yields the following equation for $\psi$:
$$
4(1-\phi)\psi+\left(\pa_\mY^{-1}\phi-\frac 12 \mY\right) \pa_\mY \psi+3\pa_{\mY}^{-1}\psi \pa_\mY\phi=0, \ \ \psi(0)=0
$$
One computes from the formula of $\phi$:
\bee
&0=& 4\psi \left(\frac 12 -\frac 12 \cos(\mY-\pi)\right) +\frac 12 \sin(\mY-\pi) \pa_\mY \psi-\frac 32 \pa_\mY^{-1}\psi \sin (\mY-\pi)\\
\Leftrightarrow &0=& 4\psi \sin^2\left(\frac{\mY-\pi}{2}\right)+\frac 12 \sin(\mY-\pi) \pa_\mY \psi-\frac 32 \pa_\mY^{-1}\psi \sin (\mY-\pi) \\
\Leftrightarrow &0=& 4 \psi \tan \left(\frac{\mY-\pi}{2}\right)+ \pa_\mY \psi- 3 \pa_\mY^{-1}\psi .
\eee
We change variables and set $\psi(\mY)=\tilde \psi ((\mY-\pi)/2)$, with $r=(\mY-\pi)/2$. One then has:
\be \lab{eq:tildepsi}
8 \tilde \psi \tan(r)+ \pa_{r} \tilde \psi- 12 \int_{-\frac \pi 2}^{r} \tilde \psi=0.
\ee
We finally change again variables by setting
$$
\int_{-\frac \pi 2}^{r} \tilde \psi= f(r)
$$
and $f$ finally solves 
\be \lab{eq:f}
f''+8 \tan(r)f'-12 f=0, \ \ f\left(-\frac{\pi}{2}\right)=f'\left(-\frac{\pi}{2}\right)=0.
\ee
A first solution (forgetting about the boundary conditions) is
$$
f_1=\sin^2 r+\frac 16
$$
The wronskian $W=f_1'f_2-f_2'f_1$ between solutions solves
$$
W'=-8W\tan (r)
$$
and therefore is equal up to renormalisation to
$$
W=\left(\cos (r) \right)^8
$$
From the Wronskian relation we deduce that a second solution is given by
$$
f_2(r)=\left(\sin^2(r)+\frac 16\right)\left(540r+80\sin (2r)+3\sin (4r)+\frac{686}{3}\frac{\sin (2r)}{\sin^2(r)+\frac 16}\right).
$$
The set of all solutions satisfying \fref{eq:f} with the boundary conditions is spanned by
$$
f=\left( \sin^2(r)+\frac 16 \right) U(r)
$$
where
$$
U(r):=\left(540\left( r+\frac \pi 2\right)+80\sin (2r)+3\sin (4r)+\frac{686}{3}\frac{\sin (2r)}{\sin^2(r)+\frac 16}\right)
$$
satisfies
$$
\frac{d}{dr}\left(U\right)=96\frac{\cos^8(r)}{(\sin^2(r)+\frac 16)^2}
$$
so that the function $f$ and its derivatives up to order $8$ vanish at $-\pi/2$. The solution to \fref{eq:tildepsi} is then $\tilde \psi (r) =  \frac{d}{dr}f (r)$ hence
\bee
\tilde \psi (r) &=& 96 \frac{\cos^8(r)}{\sin^2(r)+\frac 16}+\sin(2r)\left(540\left(r+\frac \pi 2\right)+80\sin(2r)+3\sin(4r)+\frac{686 \sin (2r)}{3 \sin^2(r)+\frac 12}\right)
\eee
The original solution, using standard trigonometric identities is then:
\bee
\psi (\mY) = \frac{96\sin^8\left(\frac{\mY}{2}\right)}{\cos^2\left(\frac{\mY}{2}\right)+\frac 16}-\sin (\mY)V(\mY), \quad V(\mY):= 270\mY-80\sin (\mY)+3\sin (2\mY)-\frac{686\sin (\mY)}{3\cos^2\left(\frac{\mY}{2}\right)+\frac 12}.
\eee
The strict positivity of $\psi$ comes from the equation. Indeed, assume that $\psi (\mY_0)=0$ for some $0<\mY_0<2\pi$, and that $\psi>0$ on $(0,\mY_0)$. Then at $\mY_0$ there holds
$$
 \pa_\mY \psi (\mY_0)= 3 \int_0^{\mY_0}\psi>0
$$
which is a contradiction. The regularity properties and the limited development come from direct computations, and end the proof of (vi), namely
$$
\pa_\mY^j \psi(0)=0, \ \ j=0,...,7, \ \ \pa_\mY^8 \psi(0)>0, \ \ \psi(2\pi)=0, \ \ \pa_\mY \psi(2\pi)=-V(2\pi)=-540 \pi<0.
$$

\end{proof}


\section{Generic singularity for the characteristics} \label{sec:characteristics}

\noindent We describe here a property that the characteristic map has generically at the singularity time. Definition \ref{def:generic} below can be found in the analysis of \cite{van2,van1990lagrangian}, but no argument is given to support that it would hold generically. We establish this in Lemma \ref{lem:nondegeneracy}, and the ideas in its proof are new. In what follows, $\nabla$ is the gradient in Lagrangian $(X,Y)$ variables. We denote by $\nabla^{\perp}$ the orthogonal gradient, $H$ the Hessian matrix and $v^t$ the transposition:
$$
\nabla^{\perp} f= (-\pa_y f,\pa_x f), \ \ H f=\begin{pmatrix} \pa_{xx}f & \pa_{xy}f \\ \pa_{xy} f & \pa_{yy}f \end{pmatrix}, \ \ (v_1,v_2)^t=(v_2,v_1).
$$
We recall that $T$ and $T_b$ are defined by \fref{def:T1}, and that $x[t]$ is defined by \fref{eq:charx}.

\begin{definition}[Generic condition] \label{def:generic}
Let $(p^E_x,u^E)\in \mathcal F^4$ and $u_0\in \mathcal E^4$ be such that  $T<\infty$ and $T<T_b$. We say that \emph{the singularity of the characteristics is generic if} the mapping $x[t]:(X,Y)\mapsto x(t,X,Y)$ satisfies the following:
\begin{itemize}
\item[(i)] \emph{Uniqueness of the singular point}. At time $T$, there exists a unique point $(X_0,Y_0)\in \mathbb R \times (0,\infty)$ such that $\nabla x[T](X_0,Y_0)=0$. For any $(t,X,Y)\in [0,T]\times \mH$ such that $(t,X,Y)\neq (T,X_0,Y_0)$ and $x_Y(t,X,Y)=0$, there holds $x_X(t,X,Y)>0$.
\item[(ii)] \emph{Local nondegeneracy of the set of zero vorticity}. There holds:
\begin{eqnarray} \lab{lem:gen:xY1}
&&x_Y(t,X_0,Y_0) =0 \qquad \mbox{for all }t\in [0,T],\\
&& \lab{lem:gen:xY2} \nabla x_Y (T,X_0,Y_0)\neq 0.
\end{eqnarray}
\item[(iii)] \emph{Local nondegenerate minimality for $x_X$ on the set of zero vorticity}. There holds, where all quantities below are evaluated at $(T,X_0,Y_0)$:
\begin{eqnarray} \lab{conservationlaw2}
&&(\nabla x_Y)^t \nabla^{\perp} x_X=0,\\
&&\lab{conservationlaw3} p_0^2=(\nabla^{\perp } x_Y)^t\left( H x_X-\frac{(\nabla x_Y)^t\nabla x_X}{|\nabla x_Y|^2}Hx_Y \right)\nabla^{\perp}x_Y> 0.
\end{eqnarray}
\item[(iv)] \emph{First order Taylor expansion of $u$}. There holds:
\be \lab{lem:gen:tayloru}
u_X(T,X_0,Y_0)<0, \quad \mbox{and} \quad u_Y(t,X_0,Y_0)=0 \ \ \text{for} \ \text{all} \ t\in [0,T].
\ee
\end{itemize}
\end{definition}

\begin{remark}

In the case of a trivial outer Eulerian flow $p^E=u^E=0$, since $x=X+tu_0$, the above generic condition can be read on the initial datum $u_0$. It is equivalent to the fact that at a point $(X_0,Y_0)$ with $Y_0>0$ and $u_{0Y}=0$, the restriction of $u_{0X}$ to the set $\{ u_{0Y}=0\}\cup \{Y=0\}$ attains a unique negative global minimum, that moreover $\nabla u_{0y}(X_0,Y_0)\neq 0$ so that the set $\{ u_{0y}=0\}$ is locally nondegenerate, and that:
$$
p^2=\frac{p_0^2}{T^3}=(\nabla^{\perp } u_{0Y})^t\left(Hu_{0X}-\frac{(\nabla u_{0Y})^t\nabla u_{0X}}{|\nabla u_{0Y}|^2}Hu_{0Y} \right)\nabla^{\perp}u_{0Y}> 0
$$
when evaluated at $(X_0,Y_0)$, so that this minimum is nondegenerate. The quantity $p^2$ is well defined in Eulerian coordinates, and a computation shows that it satisfies the following identity, when evaluated at the characteristics $(x(t,X_0,Y_0),y(t,X_0,Y_0))$ starting from $(X_0,Y_0)$:
$$
p^2(t)=(\nabla^{\perp } u_{y})^t\left(Hu_{x}-\frac{(\nabla u_{y})^t\nabla u_{x}}{|\nabla u_{y}|^2}Hu_{y} \right)\nabla^{\perp}u_{y}=\frac{p^2}{\left(1-\frac tT \right)^2}.
$$
Hence $p^2$ should be thought of in this case as a conserved\footnote{Up to a fixed self-similar factor $\left(1-\frac tT \right)^{-2}$} quantity that will help determine the parameters $(\mu,\nu)$ of the asymptotic profile $\Theta_{\mu,\nu,\iota}$.

\end{remark}

Certain conditions in the above Definition \ref{def:generic} are always met, as the next Lemma shows.

\begin{lemma} \lab{lem:conditionspasimportantes}

Let $(p^E_x,u^E)\in \mathcal F^4$, $u_0\in \mathcal E^4$ and $(X_0,Y_0)\in \mH$ be such that $T<\infty$, $T<T_b$, and $\nabla x(T,X_0,Y_0)=0$. Then for all $t\in [0,T]$ one has $x_Y(t,X_0,Y_0) =0$ and $u_Y(t,X_0,Y_0)=0$, and there holds $u_X(T,X_0,Y_0)<0$ and $(\nabla x_Y(T,X_0,Y_0))^t \nabla^{\perp} x_X(T,X_0,Y_0)=0$.

\end{lemma}

\begin{proof}

We define for each $(X,Y)\in \mH$ the $2\times 2$ matrices $M=M(t,X,Y)$ and $A=A(t,X,Y)$ by:
\be \lab{gen:ode:defM}
\dot M=A M, \ \ M(0,X,Y)=Id, \ \ A(t,X,Y):= \begin{pmatrix} 0  & 1 \\ -p^E_{xx}(t,x(t,X,Y)) & 0 \end{pmatrix}.
\ee
From the above equation \fref{gen:ode:defM}, one obtains that $M(t)$ is of determinant $1$ (a consequence of the fact that the ODE \fref{eq:charx} is volume preserving in the phase space $(x,u)$), and in particular invertible. From the ODE \fref{eq:charx}, one deduces that at any $(X,Y)\in \mH$, $x_X$, $x_Y$, $u_X$ and $u_Y$ solve:
$$
\dot{\begin{pmatrix} x_X \\ u_X \end{pmatrix}} =A \begin{pmatrix} x_X \\ u_X \end{pmatrix}, \ \ \dot{\begin{pmatrix} x_Y \\ u_Y \end{pmatrix}} =A \begin{pmatrix} x_Y \\ u_Y \end{pmatrix},
$$
and therefore that they are given by the following formula
\be \lab{gen:ode:paXpaYid}
\begin{pmatrix} x_X (t) \\ u_X (t) \end{pmatrix} =M (t) \begin{pmatrix} 1 \\ u_{0X} \end{pmatrix}, \ \ \begin{pmatrix} x_Y (t) \\ u_Y (t) \end{pmatrix} =M (t) \begin{pmatrix} 0 \\ u_{0Y} \end{pmatrix}.
\ee
We recover from this that the set $\{ u_Y=0\}$ is preserved with time $t$. At $(T,X_0,Y_0)$ we get:
\bea \label{gen:intervanishing2}
&&\begin{pmatrix} 0 \\ u_X (T,X_0,Y_0) \end{pmatrix} =M(T,X_0,Y_0) \begin{pmatrix} 1 \\ u_{0X}(X_0,Y_0) \end{pmatrix},\\
&& \label{gen:intervanishing}
\begin{pmatrix} 0  \\ u_Y (T,X_0,Y_0)  \end{pmatrix} =M (T,X_0,Y_0) \begin{pmatrix} 0 \\ u_{0Y}(X_0,Y_0) \end{pmatrix},
\eea
since $\nabla x(T,X_0,Y_0)=0$. The first identity implies $u_X(T,X_0,Y_0)\neq 0$ using the invertibility of $M$. Next, $(1, u_{0X}(X,Y))$ and $ (0 ,u_{0Y} (X_0,Y_0))$ must be collinear since their images by $M$ are collinear and $M$ is invertible. This is only possible if $u_{0Y} (X_0,Y_0)=0$, and the second identity in \fref{gen:ode:paXpaYid} then implies $x_Y(t,X_0,Y_0)=u_Y(t,X_0,Y_0)=0$ for all $t\in [0,\infty)$.

We recall that from Step 1 in the proof of Theorem \ref{th:main}, $T$ is characterised as the first time at which $x_X$ touches zero on the set $\{u_{0Y}=0\}$, which happens away from the boundary as $T<T_b$. Hence $x_X\geq 0$ on $\{u_{0Y}=0\}\cap \{0\leq t\leq T\}$, and so since $x_X(T,X_0,Y_0)=0$ and $u_{0Y}(X_0,Y_0)=0$, $(T,X_0,Y_0)$ is a global minimizer of $x_X$ on this set. As $u_Y(t,X_0,Y_0)= 0$ for all $t\in [0,T]$, $\dot x_X=u_X$ from \fref{eq:charx}, and $u_X(T,X_0,Y_0)\neq 0$, necessarily $u_X(T,X_0,Y_0)<0$. Finally, if one of the quantities in the product $(\nabla x_Y(T,X_0,Y_0))^t \nabla^{\perp} x_X(T,X_0,Y_0)$ is zero, then this product is zero. If both are nonzero, since $x_X[T]$ restricted to $\{x_Y[T]=0 \}$ attains a minimum at $(X_0,Y_0)$ then this quantity is zero as well by Fermat's Theorem. This ends the proof of the Lemma.

\end{proof}

The terminology of "generic" singularity is due to the next Lemma.

\begin{lemma} \lab{lem:nondegeneracy}
Let $(p^E_x,u^E)\in \mathcal F^4$. In the subset of $\mathcal E^4$ of initial data $u_0$ such that $T<\infty$ and $T<T_b$, there exists a dense open set such that the singularity of the characteristics is generic in the sense of Definition \ref{def:generic}.
\end{lemma}

\begin{proof}[Proof of Lemma \ref{lem:nondegeneracy}]

Fix $u_0\in \mathcal E^4$ such that $T<\infty$ and $T<T_b$ a reference initial datum. Let $(\bar x,\bar u)$ denote the solution to the ODE \fref{eq:charx} with initial datum $\bar u(0,X,Y)=u_0(X,Y)$. For any $\delta>0$ and $(u_0+v)\in \mathcal E^4$ with $d_{\mathcal E^4}(v,0)\leq \delta$ we let $( x,u)$ denote the solution of \fref{eq:charx} with initial datum $u(0,X,Y)=(u_0+v)(X,Y)$.\\

\noindent \textbf{Step 1} \emph{Evolution of the perturbation and of the derivatives}. Differentiating \fref{eq:charx} we find that the solution $(x,u)$ can always be written under the following linearised form:
\be \lab{gen:ode:linearisation}
\begin{pmatrix} x(t) \\ u(t) \end{pmatrix} =\begin{pmatrix} \bar x(t) \\ \bar u(t) \end{pmatrix}+M(t)\begin{pmatrix} 0 \\ v \end{pmatrix}+F(t)v^2,
\ee
where the volume preserving matrix $M$ is defined by \fref{gen:ode:defM}, and where at a given $\delta$, $F:[0,T]\times \mH\rightarrow \mathbb R^2$ is a function such that $\| F\|_{C^3}\leq C$ is bounded independently of $v$.
As at $(T,X_0,Y_0)$, $M$ is invertible and sends the vector $(1,u_{0X})$ onto the vector $(0,u_X)$ from \fref{gen:intervanishing2}, one has that
\be \lab{gen:ode:defm}
M\begin{pmatrix} 0 \\ 1 \end{pmatrix} =\begin{pmatrix} m \\ n \end{pmatrix} \ \ \text{with} \ \ m(T,X_0,Y_0)\neq 0.
\ee
Hence, we can rewrite the first line of the identity \fref{gen:ode:linearisation} as:
\be \lab{nondegen:expansionx}
x(t)=\bar x(t)+mv+fv^2,
\ee
where $m(T,X_0,Y_0)\neq 0$, $\bar x$ is $C^4$, and $m$ and $f$ are bounded $C^3$ functions (with $C^3$ norm independent of $v$). We end this first step by writing the ODEs for higher order derivatives:
\be \lab{gen:ode:paXYpaYY}
\dot{\begin{pmatrix} x_{YY} \\ u_{YY} \end{pmatrix}} =A\begin{pmatrix} x_{YY} \\ u_{YY} \end{pmatrix}+\begin{pmatrix}0 \\ -(x_Y)^2p^E_{xxx}(x)\end{pmatrix}, \ \ \dot{\begin{pmatrix} x_{XY} \\ u_{XY} \end{pmatrix}} =A\begin{pmatrix} x_{XY} \\ u_{XY} \end{pmatrix}+\begin{pmatrix}0 \\ -x_Xx_Y p^E_{xxx}(x)\end{pmatrix}.\\
\ee

\noindent \textbf{Step 2} \emph{Generic singularities define an open set}. We show in this step that if $u_0$ is such that the properties (i) to (iv) of Definition \ref{def:generic} are satisfied, then this is also the case for $u_0+v$ for $\delta$ small enough. We use bars to denote quantities related to the initial datum $u_0$, and no bars for those related to $u_0+v$.

First, we make some computations for the unperturbed solution $\bar u$. At $(\bar X_0,\bar Y_0)$ one has $\bar u_{Y}=\bar x_Y=0$ for all times $t\in [0,\infty)$ from \fref{gen:intervanishing}. Next, one gets from \fref{gen:ode:paXpaYid} and \fref{gen:ode:defm} that:
$$
\bar x_{Y}(t, X_0, Y_0)= \bar m(t, X_0, Y_0) u_{0Y}( X_0, Y_0),
$$
and so since $\bar m(\bar T,\bar X_0,\bar Y_0)\neq 0$ from \fref{gen:ode:defm}, one gets that the sets $\{\bar x_Y=0\}$ and $\{\bar u_{0Y}=0\}$ coincide locally near $(\bar T,\bar X_0,\bar Y_0)$. Since $\bar x_Y(t,\bar X_0,\bar Y_0)=0$ for all times and $\nabla \bar x_Y(\bar T,\bar X_0,\bar Y_0)\neq 0$  from \fref{lem:gen:xY2}, from standard parametrisation methods we have that for $t$ close to $T$ the set $\{\bar x_{Y}(t)=0\}$ is locally near $(\bar X_0,\bar Y_0)$ a curve $\bar \Gamma$ that can be parametrised as $\bar \Gamma=\{(\bar X(s),\bar Y(s)), \ |s|<\epsilon\}$ for some small enough $\epsilon$ where (the derivatives being computed at $(t,\bar X_0,\bar Y_0)$):
$$
(\bar X(s),\bar Y(s))=(\bar X_0,\bar Y_0)+s\nabla \bar x_{Y}^{\perp}+cs^2\nabla \bar x_{Y}+O(|s|^3), \ \ c=-\frac{1}{2|\nabla \bar x_{Y}|^2}(\nabla \bar x_{Y}^\perp)^t H \bar x_{Y} \nabla \bar x_{Y}^{\perp}.
$$
A direct consequence of the above parametrisation is that one has the following Taylor expansion for $\bar x_X$ on $\bar \Gamma$ for $t$ close to $T$, after some elementary computations, using that $\bar x_X(\bar T,\bar X_0,\bar Y_0)=0$, $\dot x_X=u_X$, the identity \fref{conservationlaw2} and the definition of $\bar p_0$:
\be \label{gen:taylorexpbaru}
\bar x_{X}(t,\bar X(s),\bar Y(s))=-(\bar T-t)\bar u_X(\bar T,\bar X_0,\bar Y_0)(1+O(|\bar T-t|+|s|))+ \frac 12 s^2 \bar p^2_0(1+O(|\bar T-t|))+O(|s|^3),
\ee
$$
\frac{d^2}{ds^2} \bar x_X(t)=\bar p^2_0(1+O(|\bar T-t|+|s|)),
$$
where $\bar p^2_0>0$ is defined by \fref{conservationlaw2}.

Next, we study the perturbed solution $u$. First, note that for the unperturbed solution, because of (i), outside any neighbourhood of $(\bar T,\bar X_0,\bar Y_0)$, there exists $c>0$ such that $|\nabla \bar x|>c$ is uniformly far away from $0$. Hence, for $\delta$ small enough, it at any time $T$ close to $\bar T$ the solution is such that $\nabla x[T]=0$ somewhere, this has to be at a point near $(\bar X_0,\bar Y_0)$.\\

From the condition \fref{lem:gen:xY2} for $\bar u$, for $\delta$ small enough, for all $t$ close to $\bar T$, the set $\{x_Y=0\}$ is near $(\bar X_0,\bar Y_0)$ a non-empty curve $\Gamma$. Since $x_{Y}(t, X_0, Y_0)= m(t, X_0, Y_0) u_{0Y}( X_0, Y_0)$, and that $\bar m(\bar T, \bar X_0, \bar Y_0)\neq 0$, we obtain moreover that for $\delta$ small enough and $t$ close to $\bar T$, $\Gamma$ is independent of time. From \fref{gen:taylorexpbaru}, for $\delta$ small enough, there exists on $\Gamma$, at time $\bar T$, a unique point $(X_0,Y_0)$ at which $x_X$ is minimal (non necessarily $0$). From the nondegeneracy condition \fref{lem:gen:xY2}, one obtains that for $t$ close to $\bar T$, the curve $\Gamma$ can be parametrised by $ \Gamma=\{(X(s), Y(s)), \ |s|<\epsilon'\}$ for some small enough $\epsilon'$ where (the quantities below being computed at $(t,X_0,Y_0)$):
$$
(X(s),Y(s))= (X_0,Y_0)+s\nabla  x_{Y}^{\perp}+cs^2\nabla x_{Y}+O(|s|^3), \ \ c=-\frac{1}{2|\nabla x_{Y}|^2}(\nabla x_{Y}^\perp)^t H x_{Y} \nabla x_{Y}^{\perp},
$$
A direct consequence of the above parametrisation and of the definition of $(X_0,Y_0)$ is that on $ \Gamma$, close to $(\bar T,X_0,Y_0)$:
$$
x_{X}(t)=x_X(\bar T,X_0,Y_0)-(\bar T-t)\bar u_X(\bar T,\bar X_0,\bar Y_0)(1+O((\bar T-t)+|s|+\delta))+ \frac 12 s^2 \bar p^2_0(1+O((\bar T-t)+|s|+\delta)).
$$
As $x_X(\bar T,X_0,Y_0)=O(\delta)$ and $\bar u_X(\bar T,\bar X_0,\bar Y_0)<0$, we infer from the above expansion that the first time $x_X$ touches zero on the curve $\{x_Y=0 \}$ is $T=\bar T+O(\delta)$, at least at one location $s=O(\delta)$. The condition 
$$
\frac{d^2}{ds^2} x_X(t,X(s),Y(s))=\bar p^2_0(1+O(|\bar T-t|+|s|+\delta),
$$
ensures that this happens at a unique location. Hence (i) is satisfied. The inequality in (ii) is satisfied as at this point $\nabla x_Y=\nabla \bar x_Y +O(\delta)\neq 0$. The inequality in (iii) is satisfied as $p_0^2=\bar p^2_0+O(\delta)>0$. Finally, all other conditions in Definition \ref{def:generic} are satisfied from Lemma \ref{lem:conditionspasimportantes}. Hence for $\delta$ small enough, the solution $u$ has a singularity of the characteristics that is generic in the sense of Definition \ref{def:generic}.\\

\noindent \textbf{Step 3} \emph{Density of generic singularities}. We now assume that $u_0$ is such that $T<\infty$ and $T<T_b$, but that $\bar u$ fails to meet the conditions of Definition \ref{def:generic}. For any $\delta>0$, we will find a suitable $v$ with $d_{\mathcal E^4}(v,0)\leq \delta$ such that $u$ does satisfy these conditions. As the generic condition involves several requirements, there are several cases to consider. We start with the first case: we assume that for $\bar u$ the condition (ii) holds true (the equality is always holds true from Lemma \ref{lem:conditionspasimportantes}, so only $\nabla \bar x_Y(\bar T,\bar X_0,\bar Y_0)\neq 0$ is real assumption), and that (iv) fails (the first equality always holds from Lemma \ref{lem:conditionspasimportantes} so the real assumption is that $\bar p^2\leq 0$). We do not assume anything regarding (i), and we recall that (iv) is always true from Lemma \ref{lem:conditionspasimportantes}.\\

\noindent From $\nabla \bar x_Y(\bar T,\bar X_0,\bar Y_0)\neq 0$, as in Step 2 we infer that for all $t$ near $T$, near $(\bar X_0,\bar Y_0)$, the set $\{x_Y[t]=0\}$ is a curve $\bar \Gamma$ that we can still parametrise as $\bar \Gamma=\{(\bar X(s),\bar Y(s)), \ |s|<\kappa\}$ for some small $\kappa>0$ with
\be \lab{nondegen:paramgamma}
(\bar X(s),\bar Y(s))=(\bar X_0,\bar Y_0)+s\nabla \bar x_{Y}^{\perp}+cs^2\nabla \bar x_{Y}+O(s^3), \ \ c=-\frac{1}{2|\nabla \bar x_{Y}|^2}(\nabla \bar x_{Y}^\perp)^t J \bar x_{0Y} \nabla \bar x_{Y}^{\perp},
\ee
but this time:
\be \label{id:barp0iszero}
\frac{d^2}{ds^2} \left(\bar x_X(\bar T,\bar X(s),\bar Y(s))\right)_{|s=0}=\bar p^2_0=0,
\ee
this quantity cannot be negative by minimality of $\bar x_X$ at $(\bar X_0,\bar Y_0)$, so indeed $\bar p^2_0=0$. We assume without loss of generality that the function $m$ defined by \fref{gen:ode:defm} satisfies $\bar m(\bar T,\bar X_0,\bar Y_0)>0$. We then perturb the initial datum $u_0$ by
$$
v(\bar X_0+X,\bar Y_0+Y)=\epsilon^4 \Psi \left(\frac{K(X-\bar X_0)}{\epsilon} \right)\chi \left(\frac{K'(Y-\bar Y_0)}{\epsilon} \right), \ \ 0<\epsilon \ll 1, \ \ K\gg K'\gg 1,
$$
where
\begin{itemize}
\item $\Psi$ is odd in $X$ and compactly supported in $[-4,4]$.
\item $\Psi'\geq 0$ on $[-4,-2]$, $\Psi'\leq 0$ on $[-2,0]$, $\Psi'<0$ on $[-1,0]$ with $\Psi'$ attaining its minimum at the origin where $\Psi'(0)=-1$.
\item $\Psi'''\geq 1$ on $[-1,0]$ with $\Psi'''(0)=1$.
\item $\chi$ is a standard smooth cutoff function $\chi=1$ on $[-1,1]$ and $\chi=0$ outside $[-2,2]$.
\end{itemize}
Note that for any value of $K$ and $K'$, $v$ is indeed small in $C^3$ for $\epsilon$ small enough. Note also that as $v$ vanishes for $|X-X_0|\geq 4\epsilon/K$ or $|Y-Y_0|\geq 2\epsilon/K'$ the solution $(x,u)$ remains unchanged there. In the modification zone, note first that one has the formula from \fref{nondegen:expansionx}:
\bea
\non x(t,X,Y)&=&\bar x(\bar T)-(\bar T-t)\bar u(T)+m(\bar T)v+fv^2+gv(T-t)+h(T-t)^2, \\
\lab{degen:exptaylorx} &=&\bar x(\bar T,X,Y)-(\bar T-t)\bar u(T,X,Y)+\bar m(\bar T,\bar X_0,\bar Y_0)v\\
\non &&+v^2f_1+v(T-t)f_2+v(X-X_0)f_3+v(Y-Y_0)f_4+(T-t)^2f_5,
\eea
where all the functions $f$, $g$, $h$, $f_1,...,f_5$ are $C^3$ functions with $O(1)$ size uniformly in $\delta$. Note that for $t$ close to $T$, the curve $\Gamma:=\{x_Y[t]=0 \}$ is unchanged for $|Y-Y_0|\leq \epsilon/K'$, hence it is still parametrised by \fref{nondegen:paramgamma}. There holds on this part of the curve for $t=T+O(K \epsilon^3)$ and $|s|\lesssim \epsilon/K$, writing $(X,Y)$ instead of $(\bar X(s),\bar Y(s))$ to ease notation:
\begin{align*}
x_X(t,X,Y) & = \bar x_X(\bar T,X,Y) -(\bar T-t)\bar u_X(\bar T,\bar X_0,\bar Y_0)\left(1+O(|X-\bar X_0|+|Y-\bar Y_0|)\right)\\
& \qquad  + K\epsilon^3 \Psi' \left(\frac{K(X-\bar X_0)}{\epsilon} \right) \chi \left(\frac{K'(Y-\bar Y_0)}{\epsilon} \right)\bar m(\bar T,\bar X_0,\bar Y_0)+ O(\epsilon^4)\\
 &=  O(|s|^3) -(\bar T-t)\bar u_X(\bar T,\bar X_0,\bar Y_0)\left(1+O(|s|)\right)\\
& \qquad  +K\epsilon^3 \Psi' \left(\frac{K(X-\bar X_0)}{\epsilon} \right) \chi \left(\frac{K'(Y-\bar Y_0)}{\epsilon} \right) \bar m(\bar T,\bar X_0,\bar Y_0)+ O(\epsilon^4) \\
&=  -(\bar T-t)\bar u_X+K\epsilon^3 \Psi' \left(\frac{K(X-\bar X_0)}{\epsilon} \right) \chi \left(\frac{K'(Y-\bar Y_0)}{\epsilon} \right) \bar m+ O\left( \frac{\epsilon^3}{K^3} \right)
\end{align*}
where we wrote $\bar u_X<0$ and $\bar m>0$ for $\bar u_X(\bar T,\bar X_0,\bar Y_0)$ and $\bar m(\bar T,\bar X_0,\bar Y_0)$ to ease notation. We now restrict further to the zone of this curve for which $|X-\bar X_0|\ll \epsilon/K $, hence $\chi=1$ and we Taylor expand $\Psi$:
$$
x_X=\left(-1+\frac{t-\bar T}{\epsilon^3K}\frac{\bar u_X}{\bar m}+ \left(\frac{K(X-\bar X_0)}{\epsilon} \right)^2+O\left( \left|\frac{K(X-\bar X_0)}{\epsilon} \right|^3\right)+O\left(K^{-4} \right)\right) \epsilon^3 K \bar m.\\
$$
The above identity implies that the first time $\pa_X x$ touches $0$ in the zone of the curve $\Gamma$ for which $|X-X_0|\ll \epsilon/K$ is at a spacetime point $(T,X_0,Y_0)$ with:
$$
T=\bar T+\frac{\epsilon^3 K\bar m}{\bar u_X}+O\left(\epsilon^3K^{-3} \right)<\bar T, \ \ X_0=\bar X_0+O\left(\frac{\epsilon}{K^3} \right).
$$
At such a time, on the part of the curve for which $|X-\bar X_0|\geq \epsilon/K$ there holds for some $c>0$ from the condition on $\Psi$:
$$
\pa_X \tilde x(T)\geq  (T-\bar T)\bar u_X+\epsilon^3 K \bar m(-1+c)+O\left(\frac{\epsilon^3}{K^3}\right)\geq c'\epsilon^3K, \ \ c'>0.
$$
Hence at time $T$ on $\{u_Y=0\}$, $x_X$ did not touch zero outside the part of $\Gamma$ in the zone $|X-X_0|\leq \epsilon/K$. From Lemma \ref{lem:conditionspasimportantes}, the vectors $\nabla  x_Y$ and $\nabla x_X$ are collinear at $(T,X_0,Y_0)$. From Step 2, the second derivative of $x_X(T,\bar X(s),\bar Y(s))$ with respect to $s$ is positively proportional to $p^2$ computed at this point. The collinearity of $\nabla  x_Y$ and $\nabla x_X$ implies the simplification for $p_0$ defined by \fref{conservationlaw2}:
\be \lab{id:p}
\frac{p^2_0}{x_{YY}}=x_{XXX}x_{YY}-3x_{XXY}x_{XY}+3x_{XYY}x_{XX}-\frac{x_{XY}^3}{x_{YY}^2}x_{YYY}.
\ee
From \fref{degen:exptaylorx} we infer that, introducing the notation $\bar P=(\bar T,\bar X_0,\bar Y_0)$:
\bee
x_{XXX} & = & \bar x_{XXX}+O(|T-\bar T|)+\bar m K^3\epsilon \Psi'''\left(\frac{K(X-\bar X_0)}{\epsilon} \right)+O(\epsilon^2)\\
&=& \bar x_{XXX}(\bar P)+\bar m K^3\epsilon \Psi'''\left(\frac{K(X-\bar X_0)}{\epsilon} \right)+O(\epsilon/K)
\eee
$$
x_{YY}=\bar x_{YY}(\bar P)+O(\frac \epsilon K), \ \ x_{XXY}= \bar x_{XXY}(\bar P)+O(\frac \epsilon K), \ \ x_{YYY}=\bar x_{XXY}(\bar P)+O(\frac \epsilon K)
$$
$$
x_{XY}=\bar x_{XY}(\bar P)+O(\frac \epsilon K), \ \ x_{XYY} =\bar x_{XYY}(\bar P)+O(\frac \epsilon K), \ \ x_{XX}=\bar x_{XX}(\bar P)+O(\frac \epsilon K).
$$
The condition $\nabla \bar x_Y\neq 0$ and the fact that $\nabla^\perp \bar x_Y.\nabla \bar x_X=0$ ensure $\bar x_{YY}\neq 0$ and therefore:
\bee
p^2_0  & = & \bar p^2_0+(\bar x_{YY})^2\bar m K^3\epsilon \Psi'''\left(\frac{K(X-\bar X_0)}{\epsilon} \right)+O(\epsilon/K)\\
&=&(\bar x_{YY})^2\bar m K^3\epsilon \Psi'''\left(\frac{K(X-\bar X_0)}{\epsilon} \right)+O(\epsilon/K) \geq  cK^3\epsilon
\eee
for some $c>0$, where we used \fref{id:barp0iszero}. This ensures that the zero we found for $x_X$ on the curve was unique in the part $|X-\bar X_0|\leq \epsilon/K$, and hence is globally unique, proving (i) in Definition \ref{def:generic}. The second inequality \fref{lem:gen:xY2} is true as $\nabla x_Y=\nabla \bar x_Y +O(\epsilon)$ and since we assumed $\nabla \bar x_Y(\bar T,\bar X_0,\bar Y_0)\neq 0)$. We proved that $p^2_0>0$ above. Using Lemma \ref{lem:conditionspasimportantes} $u$ thus meets the requirements of Definition \ref{def:generic}.\\

\noindent \textbf{Step 4} \emph{Other cases} Step 3 does not cover all cases. One also has to treat the case $\nabla \bar x_Y=0$, for which (ii) fails. The set $\{\bar x_Y=0\}$ is degenerate. As the parameter $p$ depends on third order derivatives, one has to consider three subcases, wether the symmetric Hessian matrix $H\bar x_Y(\bar T,\bar X_0,\bar Y_0)$ has eigenvalues with the same sign, different signs, or if it is degenerate. In each of these cases we can perform a similar analysis as in Step 3, so we only explain the strategy and leave the details to the reader.

In the case for which $H\bar x_Y(\bar T,\bar X_0,\bar Y_0)$ has two eigenvalues with different signs, the set $\{\bar x_Y[\bar T]=0 \}$ is locally two crossing curves. Hence as at time $\bar T$, on the set $\{\bar x_Y=0\}$, as the quantity $\bar x_X(\bar T)$ has to be minimal at $(\bar X_0,\bar Y_0)$, one gets that $Hu_X(\bar T,\bar X_0,\bar Y_0)$ must be a nonnegative matrix. We can thus perturb the initial datum to separate the two crossing curves in two non-crossing curves, while making a minimum for $\bar x_X$ appear on one of these curves.

 In the case for which $H\bar x_Y(\bar T,\bar X_0,\bar Y_0)$ has two eigenvalues with the same sign, the set $\{\bar x_Y[\bar T]=0 \}$ is locally a point. As $\bar u_X(\bar T,\bar X_0,\bar Y_0)<0$, by perturbing the initial datum one can transform this point into a circle, which makes us go back to Step 3. In the case of a degenerate matrix $H\bar x_Y(\bar T,\bar X_0,\bar Y_0)$ with at least a zero eigenvalue, one perturbs the initial data to make these eigenvalues nonzero, which makes us go back to the two previous cases.

\end{proof}

\section{Generic self-similar separation} \lab{sec:gen}

\subsection{Strategy of the proof and renormalised variables}

This section is devoted to the proof of Theorem \ref{th:main2}. We fix throughout this section $(u^E,p^E_x)\in \mathcal F^4$ and $u_0\in \mathcal E^4$ such that the solution $u$ to \fref{2DPrandtlp} has a singularity of the characteristics that is generic in the sense of Definition \ref{def:generic}. We recall that this happens for a dense subset of initial data in $\mathcal E^4$ such that $T<\infty$ and $T<T_b$ thanks to Lemma \ref{lem:nondegeneracy}. We will show that $u$ satisfies the conclusion of Theorem \ref{th:main}. The proof is decomposed in several steps:
\begin{itemize}
\item We define below self-similar Lagrangian and Eulerian variables\footnote{The mapping $(\mX,\mY)\mapsto (a,b)$ will be showed to be close to the mapping $(\mX,\mY)\mapsto (\ma,\mb)=\Phi^{-1}(\mX,\mY)$ given by the inverse of $\Phi$ defined in \fref{def:Phi}, hence this $(a,b)$ notation.}, $(a,b)$ and $(\mX,\mY)$.
\item We compute $\mX[t](a,b)$ in Subsection \ref{subsec:mX}, and show that $\mX[t](a,b)\approx \Phi_1 (a,b)$ (where $\Phi=(\Phi_1,\Phi_2)$ is defined in \fref{def:Phi}).
\item We compute $\mY[t](a,b)$ in Subsection \ref{subsec:mY} using incompressibility, from the formula for $\mX$ and the relation \fref{id:intevolume1}. We show that $\mY[t](a,b)\approx \Phi_2 (a,b)$. The technical difficulty is the parametrisation of the level curves $\{\mX=c\}$, which is done differently in the three zones $Z_0^c$, $Z_1$ and $Z_2$ defined by \fref{def:Z0}, \fref{def:Z1} and \fref{def:Z2}.
\item Theorem \ref{th:main2} is proved in Subsection \ref{subsec:proofthmain} in two parts. In the first one, we invert the characteristics and compute $(a[t](\mX,\mY),b[t](\mX,\mY))$ and show that $(a,b)\approx \Phi^{-1}(\mX,\mY)$. The technical difficulty is that one performs this inversion simultaneously at all points in an unbounded zone, and that $\Phi$ becomes degenerate at infinity and at the boundary of its support. We introduce several zones, each with an associated renormalisation, so that the error terms are uniformly estimated in each zone.
In the second part, the explicit formula for the mapping $(\mX,\mY)\mapsto (a,b)$ allows to retrieve $u$ and to end the proof of Theorem \ref{th:main2}.
\end{itemize}

The identity \fref{conservationlaw2} implies at $(T,X_0,Y_0)$ that $x_{XY}^2=x_{XX}x_{YY}$, and \fref{conservationlaw3} implies that $x_{XY}\neq 0$. There are thus four cases to consider, depending on the sign of $x_{XX}$ (which is the same as that of $x_{YY}$), and the sign of $x_{XY}$. We make the hypothesis that:
\be \label{id:assumptioniota}
x_{XX}(T,X_0,Y_0)>0 \qquad \mbox{and} \qquad x_{XY}(T,X_0,Y_0)<0.
\ee
It is clear from the proof that the sign of $x_{XY}$ does not matter for the result (i.e. does not change the parameters $\mu,\nu,\iota$ of the profile), and that the case $x_{XX}<0$ would lead to a singularity with profile $-\Theta(-\mX,\mY)$, i.e. $\iota=-1$. First, we define centred variables, for notational convenience:
$$
\mathsf{t}=T-t >0, \qquad \sfX=X-X_0, \qquad \sfY=Y-Y_0.
$$
Second, we define self-similar Lagrangian variables\footnote{We abuse notations, since variables also called $(a,b)$ were introduced in Subsection \ref{sec:selfsim} to study the profile $\T$. Our proof shows that these variables become asymptotically equivalent as $t\uparrow T$, justifying this abuse.} $(a,b)$ near $(X_0,Y_0)$:
\be \label{def:abrenorm}
\begin{pmatrix}a \\b \end{pmatrix} = \begin{pmatrix} k_1 \mathsf{t}^{-\frac 12}  & k_2 \sft^{-\frac 12}  \\ -k_3 \sft^{-\frac 34}  &k_4 \sft^{-\frac 34} \end{pmatrix} \begin{pmatrix} \sfX  \\ \sfY \end{pmatrix}, \qquad \begin{pmatrix}X \\Y \end{pmatrix} = \begin{pmatrix}X_0 \\Y_0 \end{pmatrix} +\begin{pmatrix} \bar k_1 \mathsf{t}^{\frac 12}  & -\bar k_2 \sft^{\frac 34}  \\ \bar k_3 \sft^{\frac 12}  & \bar k_4 \sft^{\frac 34} \end{pmatrix} \begin{pmatrix} a \\ b \end{pmatrix},
\ee
where the constants are given by (all expressions below being computed at $(T,X_0,Y_0)$ and $p_0$ being given by \fref{conservationlaw3}):
\begin{align}
\non &k_1=\frac{1}{\sqrt{-u_X}\sqrt{6}(x_{XX}+x_{YY})}\frac{p_0}{p^*}, \quad k_2=\sqrt{\frac{x_{XX}}{x_{YY}}}k_1, \quad k_4=\sqrt{\frac{1}{2\sqrt{6}(-u_X)^{\frac 32}}\frac{p_0}{p^*}}, \quad k_3=\sqrt{\frac{x_{XX}}{x_{YY}}}k_4 \\
\label{def:ks} & \bar k_1=x_{YY} \sqrt 6 \sqrt{-u_X} \frac{p^*}{p_0}, \quad \bar k_3=\sqrt{\frac{x_{XX}}{x_{YY}}}\bar k_1, \quad \bar k_2=\frac{\sqrt{2\sqrt{6}(-u_X)^{\frac 32}x_{XX}x_{YY}}}{x_{XX}+x_{YY}}\sqrt{\frac{p^*}{p_0}}, \quad \bar k_4=\sqrt{\frac{x_{YY}}{x_{XX}}}\bar k_2.
\end{align}

The reader should keep in mind that there are the variables $(a,b)$ given by \fref{def:abrenorm}, and that there are the variables $(\ma,\mb)$ of Proposition \ref{pr:selfsimfunda} (that will be showed to be asymptotically equivalent). We denote by $(\mX^\T,\mY^\T)$ and $(\ma,\mb)$ the change of variables related to the unperturbed profile $\T$:
\be \label{def:mXmYT}
\begin{array}{l l}
&(\ma,\mb)\mapsto (\mX^\T(\ma,\mb),\mY^\T(\ma,\mb))=\Phi(\ma,\mb) \mbox{ is the mapping }\Phi \mbox{ defined in \fref{def:Phi}},\\
&(\mX,\mY)\mapsto (\ma(\mX,\mY),\mb(\mX,\mY)) \mbox{ is the inverse of the mapping }\Phi \mbox{ defined in \fref{def:Phi}}.
\end{array}
\ee

 The two line vectors in the above matrices in \fref{def:abrenorm} are orthogonal. We renormalise the Eulerian side and use Eulerian self-similar variables $(\mX,\mY)$ according to
\be \label{def:mathcalXmathcalYrenorm}
x^*(t)=x(t,X_0,Y_0), \qquad \mX = k_5 \frac{x-x^*(t)}{\sft^{\frac 32}}, \qquad \mY = k_6 \frac{y}{\sft^{-\frac 14}},
\ee
where
$$
k_5=\frac{1}{-u_X \bar k_1},\quad k_6=\sqrt{\frac{\sqrt{-u_X}p_0}{2\sqrt{6}p^*}}
$$
so that the mapping $(a,b)\mapsto (\mX,\mY)$ preserves volume. Note that in \fref{def:mathcalXmathcalYrenorm} a slight abuse of notation is made, as these are the $(\mX,\mY)$ variables of Theorem \ref{th:main2}, but up to some fixed renormalisation factors. This is only to simplify notations in the forthcoming analysis.

\begin{center}
\includegraphics[width=14cm]{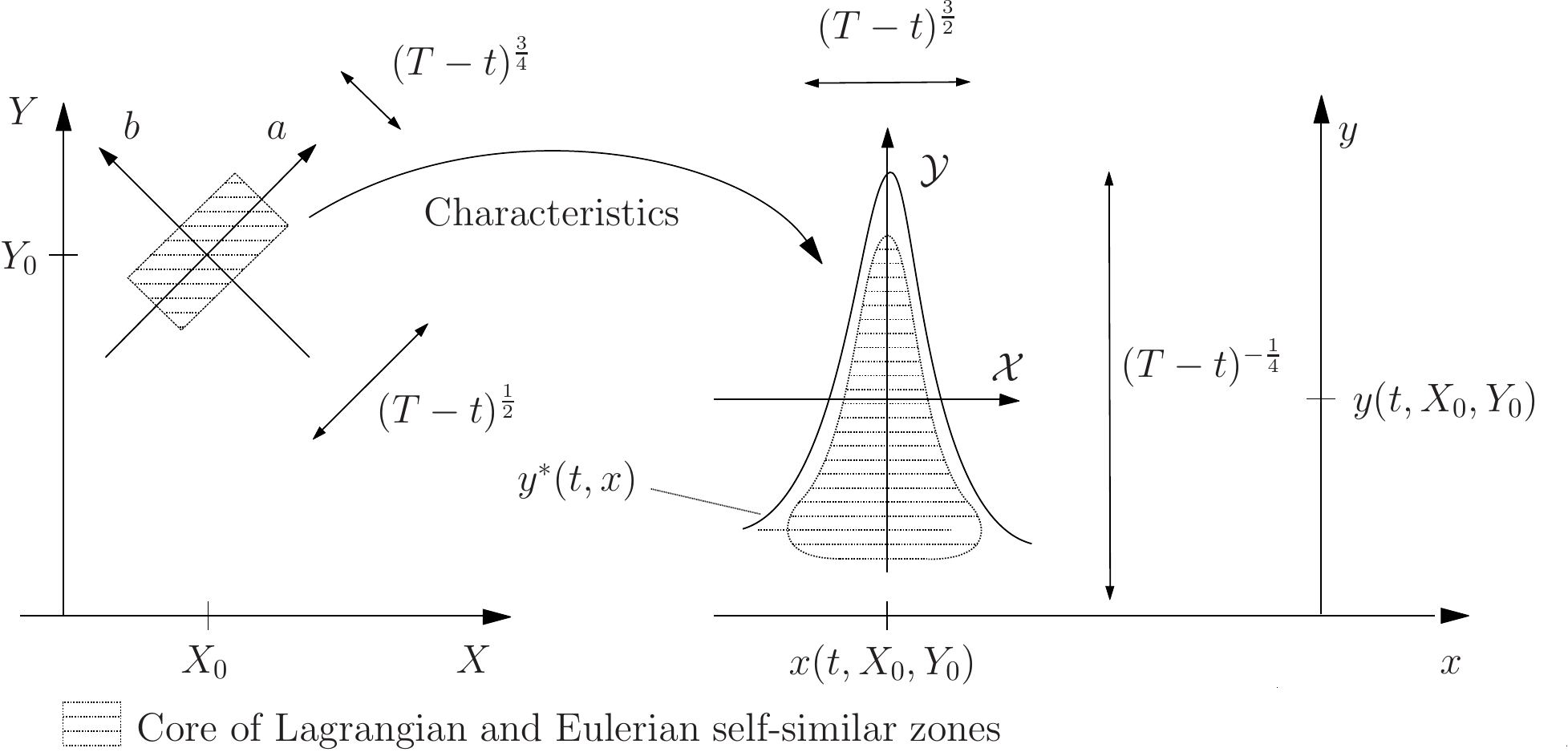}
\end{center}

Our analysis will prove that the parameters $(\mu,\nu,\iota)$ of Theorem \ref{th:main2} are\footnote{We recall that the value $\iota=1$ has been ensured by the sign assumption \fref{id:assumptioniota}, see the comments there.}:
\be \label{id:defparametersgeneric}
\mu=-u_X \bar k_1, \qquad \nu =k_6^{-1}, \qquad \iota=1.
\ee

\subsection{Tangential displacement} \label{subsec:mX}

We compute here $(a,b)\mapsto \mX(t,a,b)$, with detailed estimates to parametrise the level curves $\{\mX=Cte\}$ in the next Subsection. We describe two zones $Z_1$ and $Z_2$ near $(X_0,Y_0)$, in $(a,b)$ variables, for which the curve $\mX=Cte$ can be parametrised either with the variable $b$ in $Z_1$, or the variable $a$ in $Z_2$ (the Taylor expansion is not precise enough to allow for a single parametrisation). For this purpose, we assume for a small parameter $0<\delta \ll 1$ that:
\be \label{bd:seftdelta}
0<T-t \leq \delta^2,
\ee
(equivalently $0<\sft \leq \delta^2$) and for a large constant $L\gg 1$ we define:
\be \lab{def:Z0}
Z_0:= \mH \backslash \left\{|a|\leq \delta^{\frac 13}\sft^{-\frac 12} \ \text{and} \ |b|\leq \delta \sft^{-\frac 34} \right\},
\ee
\be \lab{def:Z1}
Z_1=Z_1[L]:= Z_0^c \cap \{ |b|\leq L(1+|a|^{\frac 32})\},
\ee
\be \lab{def:Z2}
Z_2=Z_2[L]:= Z_0^c \cap Z_1^c =Z_0^c \cap \{ |b|> L(1+|a|^{\frac 32})\}.
\ee
The zone $Z_0^c$ defines a size one neighbourhood of $(X_0,Y_0)$ in Lagrangian variables $(X,Y)$, and so $Z_0$ is an exterior zone where the dynamics remains regular. The following two Lemmas give all estimates we will use later on. The precise value of the constants $L$ above do not matter, we simply will take them large enough to ensure a dichotomy for the level curves $\{x=Cte \}$, see next Lemma \ref{lem:lagrangian}. We introduce:
$$
r=\frac{\bar k_2}{\bar k_1}
$$
\begin{lemma}  \lab{ap:calculsx}
For $L$ large enough, then $\delta$ small enough, the following holds true. For $(X,Y)\in Z_0^c$, there holds, with constants in the $O()$'s depending on $L$:
\bea
\non \mX &=& a+p^{*2}a^3\left(1+O\left(\sft+|a|\sft^{\frac 12}\right)\right) +b^2 \left(1+O\left(\sft+|b|\sft^{\frac 34}+|a|\sft^{\frac 12}\right)\right) \\
\lab{bd:xvanshen} &&+\sft^{\frac 14}rb+O\left(a^2|b|\sft^{\frac 14}+a^2\sft^{\frac 12}+|b|\sft^{\frac 12+\frac 34} \right).
\eea
If moreover $(X,Y)\in Z_1$ then there holds:
\begin{eqnarray} \lab{bd:Z1:x}
\mX&=&\left(a+p^{*2}a^3\right)\left(1+O\left(\sft^{\frac 14}+|a|^{\frac 12}\sft^{\frac 14}\right)\right) +b^2 +O(\sft^{\frac 14}),\\
 \lab{bd:Z1:a}
a&=&-\frac{1}{p^*} \Psi_1\left(p^*\left(\mX-b^2 \right) \right)\left(1+O\left(\sft^{\frac 14}+|\mX|^{\frac 16}\sft^{\frac 14}+|b|^{\frac{1}{3}}\sft^{\frac 14}\right)\right)+O(\sft^{\frac{1}{12}}).
\end{eqnarray}
Or if, moreover, $(X,Y)\in Z_2$ then there holds:
\begin{eqnarray} \lab{bd:Z2:x}
\mX&=& \left(a+p^{*2}a^3+b^2\right)\left(1+O\left(\sft^{\frac 14}|b|^{\frac 13}\right)\right),\\
 \lab{bd:Z2:a}
a&=&-\frac{1}{p^*}\Psi_1\left(p^*(\mX-b^2)\right)+O\left(|\mX|^{\frac 13+\frac{1}{18}}\sft^{\frac{1}{12}}\right), \\
\lab{bd:Z2:b}
b&=&\pm  \sqrt{\mX-a-p^{*2}a^3}\left(1+O\left(|b|^{\frac 13}\sft^{\frac 14}\right)\right)=\pm  \sqrt{\mX-a-p^{*2}a^3}\left(1+O\left(|\mX|^{\frac 16}\sft^{\frac 14}\right)\right).
\end{eqnarray}

\end{lemma}

\begin{proof}

The proof uses the Taylor expansion of $x$ near $(T,X_0,Y_0)$ and the algebraic identities of Definition \ref{def:generic}. In what follows, all symbols of functions without arguments denote the evaluation of these functions at the point $(T,X_0,Y_0)$. From Lemma \ref{lem:nondegeneracy}, for all $t$ there holds $x_{Y}(t,X_0,Y_0)=0$ and at the point $(T,X_0,Y_0)$ there holds $x_{X}=0$, $x_{Xt}=u_X$, and $x_{XX}x_{YY}=x_{XY}^2$ with the assumption $x_{XX}>0$ and $x_{XY}<0$ from \fref{id:assumptioniota}. Thus the Taylor expansion of $x$ at $(T,X_0,Y_0)$ gives:
\bea
\non x(T-\sft,X_0+\sfX,Y_0+\sfY)&=&x^*(t)-u_X\sft\sfX+\left(\sqrt{\frac{x_{XX}}{2}}\sfX-\sqrt{\frac{x_{YY}}{2}}\sfY\right)^2 \\
\non  && +\frac{x_{XXX}}{6}\sfX^3+\frac{x_{XXY}}{2}\sfX^2\sfY+\frac{x_{XYY}}{2}\sfX\sfY^2+\frac{x_{YYY}}{6}\sfY^3\\
\lab{id:taylor:xt} &&+O\Bigl( |\sfX|^4+|\sfY|^4+\sft|\sfX|^2+\sft|\sfY|^2+\sft^2|\sfX|\Bigr) .
\eea
We use now the variables $(a,b)$, for $(X,Y)\in Z_0^c$, which means that $|a|\leq \delta^{1/3} \sft^{-1/2}\ll \sft^{-1/2}$ and $|b|\leq \delta \sft^{-3/4}\ll\sft^{-3/4}$. From the identities \fref{id:p}, \fref{def:abrenorm} and \fref{def:ks}, the above expression becomes:
\begin{align*}
x-x^*(t)&= (-u_X)\bar k_1 \Bigl(\sft^{\frac 32}a+\sft^{\frac{3}{2}} p^{*2}a^3\left(1+O\left(\sft+|a|\sft^{\frac 12}\right)\right)+\sft^{\frac 32} b^2 \left(1+O\left(\sft+\sft^{\frac 34}|b|+\sft^{\frac 12}|a|\right)\right) \\
&\qquad \qquad \qquad +\sft^{\frac 74}rb+O\left(a^2|b|(1-t)^{\frac 74}+\sft^{2}|a|^2+\sft^{2+\frac 34}|b| \right)\Bigr)
\end{align*}
which, using \fref{def:mathcalXmathcalYrenorm}, shows \fref{bd:xvanshen}. Consider now \fref{bd:xvanshen} in the zone $Z_1$ where $|b|\leq L( 1+|a|^{3/2})$. As $a$ and $a^3$ share the same sign, and since $|a|\sft^{1/2}\ll 1$ and $|b|\sft^{3/4}\ll1$ we can write:
\bee
\non \mX &=& \left(a+p^{*2}a^3\right)\left(1+O\left(\sft+|a|\sft^{\frac 12}\right)\right) +b^2 \left(1+O\left(\sft^{\frac 34}+\sft^{\frac 12}|a|\right)\right) \\
&&+\sft^{\frac 14}rb+O\left(|a|^{3+\frac 12}\sft^{\frac 14}+|a|^2\sft^{\frac 14} +\sft^{\frac 12+\frac 34}|a|^{\frac 32}+\sft^{\frac 12 +\frac 34} \right)\\
&=& \left(a+p^{*2}a^3\right)\left(1+O\left(\sft^{\frac 14}+|a|^{\frac 12}\sft^{\frac 14}\right)\right) +b^2+O(\sft^{\frac 14}).
\eee
where we used $|a|^2+|a|^{3/2}\lesssim |a+a^3|$. This proves \fref{bd:Z1:x}. We rewrite the above identity as:
$$
\left(\mX-b^2+O(\sft^{\frac 14})\right)\left(1+O\left(\sft^{\frac 14}+|a|^{\frac 12}\sft^{\frac 14}\right)\right)= a+p^{*2}a^3
$$
Note that the solution to $a+p^{*2}a^3=f$ is $a=-\Psi_1(p^*f)/p^*$. Hence:
$$
a=-\frac{1}{p^*} \Psi_1\left(p^*\left(\mX-b^2+O(\sft^{\frac 14})\right)\left(1+O\left(\sft^{\frac 14}+|a|^{\frac 12}\sft^{\frac 14}\right)\right)\right)
$$
Now, for $t$ close to $T_0$, and $(X,Y)$ in $Z_0^c$, there holds $O\left(\sft^{\frac 14}+|a|^{\frac 12}\sft^{\frac 14}\right)=o(1)$, hence using \fref{bd:Psi11} the above identity gives:
$$
a=-\frac{1}{p^*} \Psi_1\left(p^*\left(\mX-b^2+O(\sft^{\frac 14})\right)\right)\left(1+O\left(\sft^{\frac 14}+|a|^{\frac 12}\sft^{\frac 14}\right)\right).
$$
Hence, from \fref{bd:Psi11}  $|a|\lesssim |\Psi_1\left(p^*\left(\mX-b^2+O(\sft^{\frac 14})\right)\right)|\lesssim |\mX|^{1/3}+|b|^{2/3}+\sft^{1/12}$. This, reinjected in the above identity and using \fref{bd:Psi11}, gives \fref{bd:Z1:a}. Consider now \fref{bd:xvanshen} in the zone $Z_2$ where $|a|\leq L^{-2/3} |b|^{2/3}\ll b^{2/3}$ and $|b|\geq L\gg 1$. Then the dominant term in the right hand side of \fref{bd:xvanshen} is $b^2$ and we infer that $\mX\approx b^2\gg 1$ and $\mX\gg |a|^{3}$. Injecting these bounds in \fref{bd:xvanshen} gives the desired identity \fref{bd:Z2:x} for $\mX$ in $Z_2$. Then we rewrite \fref{bd:Z2:x} as:
$$
b^2=\mX (1+O(|b|^{\frac 13}\sft^{\frac 14}))-a-p^{*2}a^3=(\mX -a-p^{*2}a^3)(1+O(|b|^{\frac 13}\sft^{\frac 14}))
$$
as $\mX^{1/3}\gg |a|$. Hence the solution is:
\bee
b&=&\pm  \sqrt{(\mX -a-p^{*2}a^3)(1+O(|b|^{\frac 13}\sft^{\frac 14}))}=\pm \sqrt{\mX -a-p^{*2}a^3}(1+O(|b|^{\frac 13}\sft^{\frac 14})) \\
&&\quad =\pm  \sqrt{\mX -a-p^{*2}a^3}(1+O(|\mX|^{\frac 16}\sft^{\frac 14})) 
\eee
which shows \fref{bd:Z2:b}. We also rewrite \fref{bd:Z2:x} as $ a+p^{*2}a^3=\mX-b^2+O\left(|\mX|^{1+1/6}\sft^{1/4}\right)$, so that, using \fref{bd:Psi11}:
$$
a=-\frac{1}{p^*} \Psi_1\left(p^*\left(\mX-b^2+O\left(|\mX|^{1+\frac 16}\sft^{\frac 14}\right)\right)\right)=-\frac{1}{p^*} \Psi_1\left(p^*\left(\mX-b^2\right)\right)+O\left(|\mX|^{\frac 13+\frac 16}\sft^{\frac 14}\right)
$$
This shows \fref{bd:Z2:a} and ends the proof of the Lemma.

\end{proof}

The next Lemma adapts Lemma \ref{ap:calculsx} to higher order derivatives.

\begin{lemma}  \lab{ap:calculsx2}
For $L$ large enough, then $\delta$ small enough, the following holds true. For $(X,Y)\in Z_0^c$, there holds, with constants in the $O()$'s depending on $L$:
\begin{eqnarray} \lab{bd:paaMxvanshen}
\pa_a \mX & = & 1+3p^{*2}a^2+O\left(|a||b|\sft^{\frac 14}+|b^2|\sft^{\frac 12}+|a|^3\sft^{\frac 12}+\sft+\sft^{\frac 12}|a|+\sft^{\frac 14}|b| \right),\\
 \lab{bd:pabmXvanshen}
\pa_{b}\mX & = & 2b+r\sft^{\frac 14}+O\left(\sft^{\frac 34}|b|^2+\sft^{\frac 12}|b||a|+\sft|a|^2+\sft^{\frac 54}+\sft|b|+\sft^{\frac 34}|a| \right).
\end{eqnarray}
If moreover $(X,Y)\in Z_1$ then there holds:
\begin{eqnarray} \lab{bd:Z1:paamX}
\pa_a\mX &=& (1+3p^{*2}a^2)(1+O(\sft^{\frac 14}+|a|^{\frac 12}\sft^{\frac 14})) \\
\non&&= \left(1+3\Psi_1^2\left(p^*\left(\mX-b^2 \right)\right)\right)\left(1+O(\sft^{\frac{1}{12}}+|\mX|^{\frac 16}\sft^{\frac 14}+|b|^{\frac 13}\sft^{\frac 14})\right),\\
  \lab{bd:Z1:paaamX}
\pa_{aa}\mX&=& 6p^{*2}a\left(1+O\left((|a|^{\frac 12}\sft^{\frac 14}) \right)\right)+O\left(\sft^{\frac 14}\right) \\
\non & & = -6p^* \Psi_1\left(p^*\left(\mX-b^2 \right) \right)\left(1+O\left(\sft^{\frac 14}+|\mX|^{\frac 16}\sft^{\frac 14}+|b|^{\frac{1}{3}}\sft^{\frac 14}\right)\right)+O(\sft^{\frac{1}{12}}),\\
 \lab{bd:Z1:pabmX}
\pa_b \mX & = & 2b\left(1+O(\sft^{\frac 34}+|a|\sft^{\frac 12})\right)+O\left(\sft^{\frac 14}+\sft|a|^2 \right).
\end{eqnarray}
Or if, moreover, $(X,Y)\in Z_2$ then there holds:
\begin{eqnarray} \lab{bd:Z2:pax}
\pa_b \mX&=& 2b\left(1+O\left(|b|^{\frac 13}\sft^{\frac 14}\right)\right) \ = \ \pm 2\sqrt{\mX-a-p^{*2}a^3} \left(1+O\left(|\mX|^{\frac 16}\sft^{\frac 14}\right)\right),\\
\lab{bd:Z2:pabbx}
\pa_{bb} \mX & = & 2\left(1+O\left(|b|^{\frac 23}\sft^{\frac 12}\right)\right)=2\left(1+O\left(|\mX|^{\frac 13}\sft^{\frac 12}\right)\right),\\
 \lab{bd:Z2:paax}
\pa_a \mX & = & 1+3p^{*2}a^2+O\left(|b|^{\frac 53}\sft^{\frac 14}\right)=1+3p^{*2}a^2+O\left(|\mX|^{\frac 56}\sft^{\frac 14}\right).
\end{eqnarray}

\end{lemma}

\begin{proof}

We omit the proof, relying only on explicit manipulation of the Taylor expansion of $x$, and which is verbatim the same as that of Lemma \ref{ap:calculsx}.

\end{proof}

\subsection{Normal displacement} \label{subsec:mY}

The function $(a,b)\mapsto \mX(t,a,b)$ has been computed in the previous Subsection, what allows us to compute $(a,b)\mapsto \mY(t,a,b)$, relying on Lemma \ref{lem:paramincompressibility}. Let $\Gamma [x]$ denote the curve $\{x[t](X,Y)=x \}$ in Lagrangian variables, with starting point at $\{Y=0\}$. The proof of Lemma \ref{lem:lagrangian} below ensures $\Gamma[x]$ enters $Z_0^c$ either from the bottom side $\{b=-\delta \sft^{-3/4}\}$ or from the upper side $\{b=\delta \sft^{-3/4}\}$. By a continuity argument, for $\delta$ small, then $\epsilon>0$ small enough, all such curves with $|x-x^*|\leq \epsilon$ enter from the same side. Without loss of generality we assume it to be the bottom side. This assumption is harmless: if they enter from the upper side the same conclusions would hold from the symmetry (iii) in Proposition \ref{pr:selfsimfunda}.

With the control of the behaviour of $\mX$ in the zones $Z_1$ and $Z_2$ obtained in Lemma \ref{ap:calculsx}, a parametrisation of these level curves as graphs over the $b$ variable is possible most of the time, but there is a case for which this is impossible, and we then cut in several zones, each parametrised with either the variable $a$ or the variable $b$. The following picture is described by Lemma \ref{lem:lagrangian} below.

\begin{center}
\includegraphics[width=12cm]{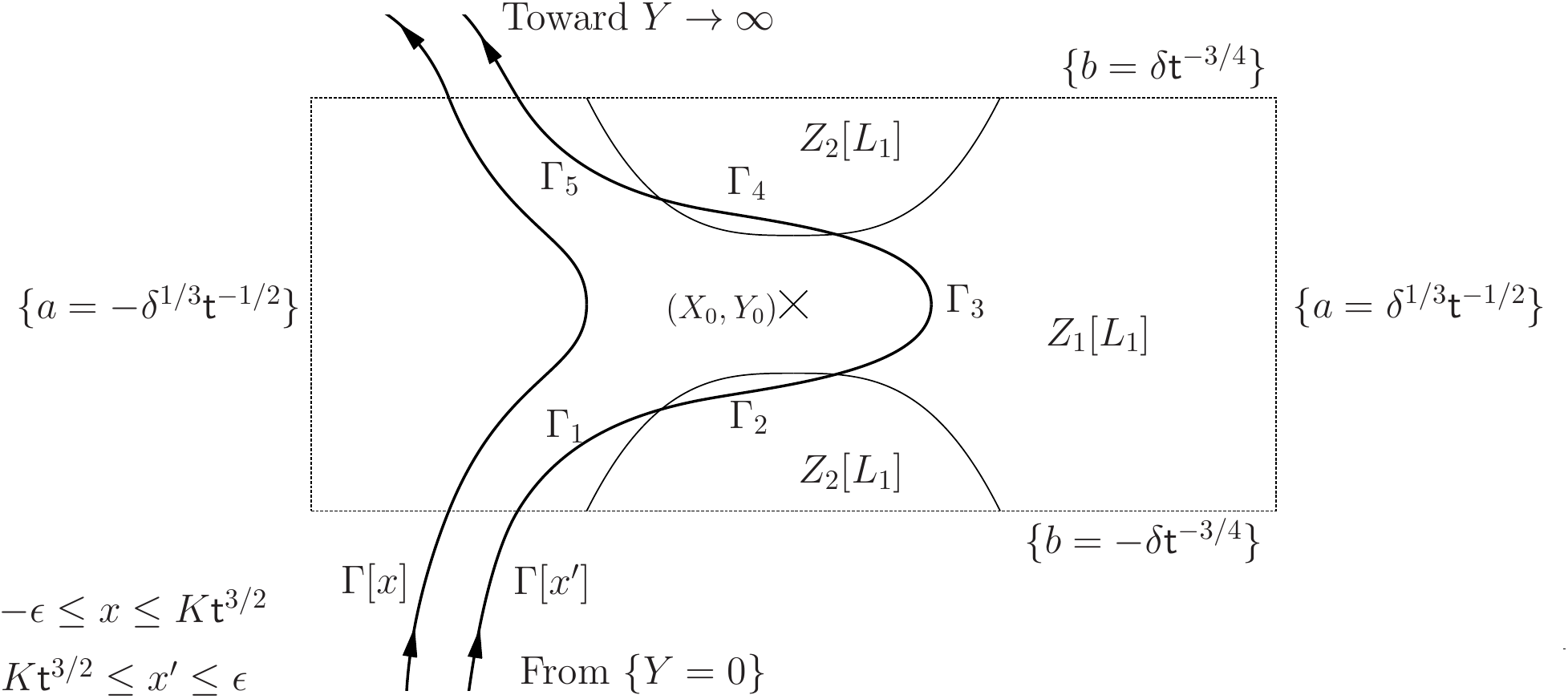}
\end{center}

\begin{lemma}[Lagrangian structure of the curves $\{x=Cte\}$] \lab{lem:lagrangian}

For all $\tilde L$ large enough, for any $L_1\geq \tilde L$, there exists $K^*\geq \tilde L$ such that for $K\geq K^*$, there exists $L_2^*\geq \tilde L$ such that for $L_2\geq L_2^*$, for $\delta$ small enough and then $\epsilon$ small enough, the following holds. For any $| x-x^*|\leq \epsilon $ the set $\Gamma [x] \cap Z_0^c$ consists of a curve entering and exiting at the points $(a_{in},b_{in})$ and $(a_{out},b_{out})$ with $b_{in}=-\delta \sft^{-3/4}$ and $b_{out}=\delta \sft^{-3/4}$. In addition,
\begin{itemize}
\item[(i)] For $-\epsilon\leq x-x^*\leq K\sft^{3/2}$ this set can be parametrised with the variable $b$, via $a=a(x,b)$, and is contained in $Z_1[L_2]$.
\item[(ii)] For $K\sft^{3/2}\leq x-x^* \leq \epsilon$ this set consists on the concatenation of five curves $\Gamma_{1},\Gamma_{2},\Gamma_3,\Gamma_{4},\Gamma_{5}$ with starting and end points $(a_{in},b_{out}),(a_{1},b_{1}),(a_{2},b_{2}),(a_{3},b_{3}),(a_{4},b_{4})$, and $(a_{out},b_{out})$ where
$$
b_{1}=-L_1(1+|a_{1}|^{\frac 32}), \ \ b_{2}=-L_1(1+|a_{2}|^{\frac 32}), \ \ b_{3}=L_1(1+|a_{3}|^{\frac 32}), \ \ b_{4}=L_1(1+|a_{4}|^{\frac 32}).
$$
$\Gamma_{1}$ lies in $Z_1[L_1]$ with $a,b\ll -1$, and can be parametrised as a curve $a=a(x,b)$. $\Gamma_{2}$ lies in $Z_2[L_1]$ with $b<0$, and can be parametrised as a curve $b=b(x,a)$. $\Gamma_{3}$ lies in $Z_1[L_1]$ with $a>0$, and can be parametrised as a curve $a=a(x,b)$. $\Gamma_{4}$ lies in $Z_2[L_1]$ with $b>0$, and can be parametrised as a curve $b=b(x,a)$. $\Gamma_{5}$ lies in $Z_1[L_1]$ with $a<0$ and $b>0$, and can be parametrised as a curve $a=a(x,b)$.
\end{itemize}

\end{lemma}

\begin{proof}

\noindent \textbf{Step 1} \emph{The in and out points} Fix $|x|\leq \epsilon$, we show the existence of the in and out points and show that they are the only points at the boundary of $Z_0^c$. The in point is the point of the curve belonging to the lower boundary $\{b=-\delta\sft^{-3/4} \}$ of $Z_0^c$. We fix $b_{in}=-\delta \sft^{-3/4}$ and look for the corresponding parameter $a_{in}$. For $(a,b)\in Z_2[L_1]$ one has $|b|\gtrsim L_1$, $|a|\ll |b|^{2/3}$ and $\sft^{\frac{3}{4}}|b|=o(1)$ as $\delta\rightarrow 0$. Hence from \fref{bd:Z2:x}:
$$
\mX \gtrsim |b|^2.
$$
In particular, if $(a,b_{in})\in Z_2[L_1]$ then $x(a,b_{in})-x^*\gtrsim \delta^2 >\epsilon\geq x-x^*$. Hence there are no solutions in $Z_2[L_1]$. As $(0,b_{in})\in Z_2[L_1]$, one has from the inequality above $x(0,b_{in})-x^*>x-x^*$ and as $(-\delta^{1/3}\sft^{-1/2},b_{in})\in Z_1[L_1]$ one has from \fref{bd:Z1:x} 
\be \lab{bd:x:leftboundary}
x (-\delta^{1/3}\sft^{-1/2},b_{in})-x^*=-p^{*2} \delta (1+o(1))<-\epsilon\leq x-x^*
\ee
Hence by the intermediate value Theorem there exists a solution $a_{in}$ to $x(a_{in},b_{in})=x$, which is unique as $\pa_ax>0$ on $Z_1[L_1]$ from \fref{bd:Z1:paamX}. One can show the existence and uniqueness of the out point $(a_{out},b_{out})$ at the north boundary $\{b=\delta \sft^{-3/4} \}$ with the same argument. The computation \fref{bd:x:leftboundary} holds true at the left boundary $\{a=-\delta^{1/3}\sft^{-1/2} \}$ of $Z_0^c$, and a similar computation shows that $x-x^*>\epsilon$ at the right boundary $\{a=\delta^{1/3}\sft^{-1/2} \}$. Hence the two points $(a_{in},b_{in})$ and $(a_{out},b_{out})$ are the only points of the curve at the boundary of $Z_0^c$.\\

\noindent \textbf{Step 2} \emph{Proof of (i)}. From the first step, the strict monotonicity $\pa_ax>0$ in $Z_1[L_2]$ from \fref{bd:Z1:paamX}, and since $\Gamma \cap Z_0^c\subset Z_1[L_2]$, a standard application of the implicit function theorem gives that the curve $x(a,b)=x$ can be parametrised with the variable $b$ as $a=a(x,b)$ for $b\in [-\delta \sft^{-3/4},\delta \sft^{-3/4}]$.\\

\noindent \textbf{Step 3} \emph{Proof of (ii)}. Fix $K\sft^{3/2}\leq x-x^* \leq \epsilon$. We first prove the existence and uniqueness of the point $(a_1,b_1)$ on the curve $b_1=-L_1(|a_1|^{3/2}+1)$ with $a_1<0$. On this curve, from \fref{bd:Z2:x}, using the fact that $|a_1|\ll |b_1|$:
$$
\mX(a_1,b_1)=b_1^2(1+o(1))
$$
where the $o(1)$ is as $L_1\rightarrow \infty$. Hence, as for $a_1=0$, $\mX(0,-L_1)=O(L_1^2)<K\leq \mX$ and for $b_1=-\delta \sft^{-3/4}$:
$$
x\left(- \left(\frac{\delta}{\sft^{3/4}L_1}-1\right)^{2/3},-\delta \sft^{-3/4}\right)-x^*=\delta^2(1+o(1))>\epsilon\geq x-x^*,
$$
there exists a solution $(a_1,b_1)$ with $b_1=-L_1(|a_1|^{3/2}+1)$ and $a_1<0$ to $x(a_1,b_1)=x$. One can also show that $\pa_{b_1} \mX=2b_1(1+o(1))$ on this curve which implies the uniqueness of the point $(a_1,b_1)$. The existence and uniqueness of $(a_2,b_2)$, $(a_3,b_3)$ and $(a_4,b_4)$ can be done similarly.

We now show that $(a_{in},b_{in})$ is connected to $(a_1,b_1)$ by the part of the curve staying in $Z_1[L_1]$ with $a,b<0$. Indeed, note that for $b=0$ and $a<0$, on has $x<0$ from \fref{bd:Z1:x}. Also, for $|a|,|b|\leq L_1$ one has from \fref{bd:Z1:x} that $x(a,b)\leq C(L_1)\sft^{3/2}$ so $x(a,b)<x$ for $K$ large enough. Consider the part of $\Gamma$ which is in the zone $Z_1$ with $a,b<0$. This proves that the only points of this set at the boundary of $Z_1\cap \{a,b<0 \}$ are precisely $(a_{in},b_{in})$ and $(a_1,b_1)$. From a direct check on the gradient of $x$, the curve $\Gamma$ indeed penetrates the zone $Z_1\cap \{a,b<0 \}$ at these two points, and so the part of $\Gamma$ in the zone $Z_1\cap \{a,b<0 \}$ consists of a curve joining these two points. Moreover, as in $Z_1$ there holds $\pa_a x>0$ from \fref{bd:Z1:paamX}, it can be parametrised with the variable $b$.

The proof of the properties for $\Gamma_i$ for $i\geq 2$ can be proved following the same ideas. We just mention that in $Z_2[L_1]$ there holds $\pa_b \mX >0$ for $b>0$ and $\pa_b \mX<0$ for $b<0$ from \fref{bd:Z2:pax}, which shows that the curves $\Gamma_2$ and $\Gamma_4$ can indeed be parametrised with the variable $a$.

\end{proof}

The knowledge of $(a,b)\mapsto \mX$ and of its level curves (Lemmas \ref{ap:calculsx}, \ref{ap:calculsx2} and \ref{lem:lagrangian}), allows to retrieve the value of $\mY$ by incompressibility using \fref{id:intevolume2}.

\begin{lemma} \label{lem:computationmathcaly}

For any $N_2\geq 1$, for $N_1$ large enough and then $K$ large, the following holds true for $\delta>0$ small and then $\epsilon>0$ small enough. If, in the first case, $(X,Y)\in Z_0^c$ is such that either $-\epsilon \leq x(a,b)-x^*\leq K\sft^{3/2}$, or\footnote{Note that in this second subcase, for $K$ large, from \fref{bd:xvanshen} one has necessarily $|a|\approx |b|^{2/3}\gg 1$.}, $K\sft^{3/2}\leq x-x^*\leq \epsilon$ and $b<0$ and $N_2 |\mX(a,b)|\leq |b|^2$ then there holds:
\be \lab{id:mYleftcentral}
\mY(a,b)= \mY^\T(a,b) +\int_{-\infty}^b  \frac{O\left(\sft^{\frac{1}{12}}+|\tb|^{\frac 14}\sft^{\frac{3}{16}}+|\mX|^{\frac 16}\sft^{\frac{1}{4}} \right)}{1+3\Psi_1^2\left( p^* \left(\mX-\tb^2\right) \right)}d\tb+O(\sft^{\frac 14}),
\ee
\begin{eqnarray}
 \pa_{a}\mathcal Y(a,b)  \lab{id:paamYleftcentral}  &=& \pa_{a}\mathcal Y^\T(a,b)+ 6p^* \left(1+3\Psi_1^2\left(p^*\left(\mX-b^2 \right) \right) \right) \\
\non &&\qquad \qquad \qquad  \times \Biggl( \int_{-\infty}^{b}  \frac{\Psi_1\left(p^*\left(\mX-\tb^2 \right) \right)O\left(\sft^{\frac{1}{12}}+|\mX|^{\frac 16}\sft^{\frac{1}{4}}+|\tb|^{\frac 14}\sft^{\frac{3}{16}}+\sft^{\frac{1}{4}}|b|^{\frac 13} \right)}{\left(1+3\Psi_1^2\left(p^*(\mX-\tb^2) \right) \right)^3} d\tb \\
\non &&\qquad \qquad \qquad  \qquad \qquad \qquad +\int_{-\infty}^b \frac{O\left(\sft^{\frac{1}{12}} \right)}{\left(1+3\Psi_1^2\left(p^*(\mX-\tb^2) \right) \right)^3} d\tb+O(\sft^{\frac 74})\Biggr).
\end{eqnarray}
If, in the second case, $(X,Y)\in Z_0^c$ is such that:
\be \lab{autosimaabright2}
N_1\leq |p^{*2}a^3+b^2|, \ \ b^2\leq N_2 |p^{*2}a^3+b^2|,
\ee
then there holds:
\be \lab{bd:sides:mY}
\mY(a,b)= \mY^\Theta (a,b) \left(1+O(\sft^{\frac{1}{12}}|\mX|^{\frac{1}{18}}) \right), \ \ \pa_a \mY(a,b) = \pa_a\mY^{\Theta}(a,b) +O(\sft^{\frac{1}{12}}|\mX|^{-\frac 12+\frac{1}{18}}),
\ee
\be \lab{bd:sides:pabmY}
\pa_b \mY(a,b)=\pa_b\mY^\Theta (a,b)+O\left(|\mX|^{-\frac 23 +\frac{1}{18}}\sft^{\frac{1}{12}} \right).
\ee

\end{lemma}

\begin{proof}[proof of Lemma \ref{lem:computationmathcaly}]

The proof is relegated to Appendix \ref{ap:y2}.

\end{proof}


\subsection{Inverting the characteristics map} \label{subsec:proofthmain}

We now prove Theorem \ref{th:main2}. Our main strategy is to renormalise the characteristics and to invert them precisely. The following drawing sums up the transformations. In particular, we will show that the mapping $(a,b)\mapsto (\mX,\mY)$ is close to the mapping $\Phi$ defined by \fref{def:Phi}. We will use a perturbative argument to show the inverse is close to $\Phi^{-1}=(\ma,\mb)$. However, this inversion is done in an unbounded zone, which forces us to renormalise the perturbation problems when approaching infinity or the boundary of support of $\T$. Hence we need precise asymptotic estimates for $\ma$ and $\mb$. The Lemma below provides us with all the estimates that will be used to conclude the proof of Theorem \ref{th:main}.

\begin{lemma} \lab{ap:calculstheta}
One has the formulae for the mappings $(\mX^\T,\mY^\T)$ and $(\ma,\mb)$ defined by \fref{def:mXmYT}:
\be \lab{id:paamYTheta}
\pa_\ma\mY^\Theta (\ma,b)=6p^* \left(1+3p^{*2}\ma^2 \right) \int_{-\infty}^{\mb}  \frac{\Psi_1\left(p\left(\mX^\Theta(\ma,\mb)-\tmb^2 \right) \right)}{\left(1+3\Psi_1^2\left(p(\mX^\Theta(\ma,b)-\tmb^2) \right) \right)^3}d\tmb,
\ee
$$
\pa_b \mY^{\T}(\ma,\mb)=\frac{1}{1+3p^{*2}\ma^2}+12p^*\mb\int_{-\infty}^{\mb}  \frac{\Psi_1\left(p\left(\mX^\Theta(\ma,\mb)-\tmb^2 \right) \right)}{\left(1+3\Psi_1^2\left(p(\mX^\Theta(\ma,\mb)-\tmb^2) \right) \right)^3}d\tmb,
$$
\be \lab{id:pamXmYTheta}
\begin{pmatrix} \pa_\mX \ma & \pa_\mY \ma \\ \pa_\mX \mb & \pa_\mY \mb \end{pmatrix}=\begin{pmatrix} \pa_\mb \mY^\T &- \pa_\mb \mX^\T \\ -\pa_\ma \mY^\T & \pa_\ma \mX^\T \end{pmatrix}.
\ee
\begin{itemize}
\item[(i)] \emph{Bottom of the self similar zone.} There exists $ c_i\neq 0 \ \text{for} \ i=1,...,4$, such that as $\kappa \rightarrow 0$, for all $(\mX,\mY)\in \mH$ such that $0<\mY\leq \kappa \langle \mX \rangle^{-1/6}$, there holds $\mb<0$ and:
$$
|p^{*2}\ma^{3}+\mb^{2}|= o(\mb^{2}), \ \ |\mb|\approx \frac{1}{\mY^3}, \ \ |\ma |\approx \frac{1}{\mY^{2}}, \ \ \left|\int_{-\infty}^{\mb} \frac{\Psi_1(\mX-\tmb^2)}{\left(1+3\Psi_1^2(\mX-\tmb^2)\right)^3}d\tmb \right|\approx |\mb|^{-\frac 73},
$$
and, at the point $(\ma,\mb)=(\ma,\mb)(\mX,\mY)$:
\be \lab{bd:calculTheta:bottom:pamXmY}
|\pa_\ma \mX^\Theta| \sim c_1 |\mb|^{\frac 43}, \ \ |\pa_\mb \mX^\Theta|\sim c_2 |\mb|, \ \ |\pa_\ma\mY^\Theta|\sim c_3 |\mb|^{-1}, \ \  |\pa_\mb\mY^\Theta|\sim c_4 |\mb|^{-\frac 43}.
\ee
\item[(ii)] \emph{Sides of the self similar zone.} For $\kappa>0$ small, there exists $M^*(\kappa)$ with $M^*(\kappa)\rightarrow \infty$ such that for $M\geq M^*(\kappa)$, $|\mX|\geq M$ and $\kappa |\mX|^{-1/6}\leq \mY \leq (2-\kappa)C_{\pm}|\mX|^{-1/6}$ where $\pm=\text{sgn}(\mX)$:
$$
|\mb|\lesssim C(\kappa)|p^{*2}\ma^{3}+\mb^{2}|, \ \  |p^{*2}\ma^{3}+\mb^{2}|\geq \frac M2, \ \ |\mb|\approx |\mX|^{1/2}, \ \ |\ma|\approx |\mX|^{\frac 16},
$$
\be \lab{eq:selfsimsides}
\begin{array}{l l}
&\pa_\mX \ma\sim |\mX|^{-\frac 23}\varphi^1_{\pm}(\mY|\mX|^{1/6}), \ \ \pa_\mY \ma\sim |\mX|^{\frac 12}\varphi^2_{\pm}(\mY|\mX|^{1/6}),\\
&\pa_\mX \mb\sim |\mX|^{-\frac 12}\varphi^3_{\pm}(\mY|\mX|^{1/6}), \ \ \pa_\mY \mb\sim |\mX|^{\frac 23}\varphi^4_{\pm}(\mY|\mX|^{1/6}),
\end{array}
\ee
for some functions $\varphi^j_{\pm}\in C^{\infty}(0,2C_{\pm})$ where $\pm=\mbox{sgn}(\mX)$, such that $((\varphi^1_{\pm},\varphi^3_{\pm}),(\varphi^2_{\pm},\varphi^4_{\pm}))$ is volume preserving. For $N,\tilde N>0$, we have as $\tilde N \rightarrow \infty$ and $N^2\tilde N^{-1}\rightarrow \infty$ that for $(\ma,\mb)\in \mathbb R^2$ with $|p^{*2}\ma^3+\mb^2|\geq \max (N,\tilde N^{-1}\mb^{2})$ there holds:
\be \lab{bd:calculTheta:side:mY}
\mY^\Theta(\ma,\mb)\approx |\mX^{\Theta}(\ma,\mb)|^{-\frac 16}, \ \ \int_{-\infty}^{\mb} \frac{|\tmb|^{\frac{1}{9}}d\tmb}{1+3\Psi_1^2\left(p^*(\mX^\Theta(\ma,\mb)-\tmb^2)\right)}\lesssim |\mX^{\Theta}(\ma,\mb)|^{-\frac 16 +\frac{1}{18}},
\ee
\be \lab{bd:calculTheta:side:pamY}
\left| \int_{-\infty}^\mb \frac{ \Psi_1\left(p^*(\mX^\T-\tmb^2)\right)d\tmb}{\left(1+3\Psi_1^2\left(p^*(\mX^\T-\tmb^2) \right) \right)^3} \right| \lesssim |\mX^\T|^{-\frac 76}, \ \  \int_{-\infty}^\mb \frac{ |\Psi_1\left(p^*(\mX^\T-\tmb^2)\right)||\tmb|^{\frac{1}{12}} d\tmb }{\left(1+3\Psi_1^2\left(p^*(\mX^\T-\tmb^2) \right) \right)^3}\lesssim |\mX^\T|^{-\frac 76+\frac{1}{24}}.
\ee
\end{itemize}

\end{lemma}

\begin{proof}[Proof of Lemma \ref{ap:calculstheta}]

All identities in Lemma \ref{ap:calculstheta} are direct consequences of the formula for the mapping $\Phi$ defined by \fref{def:Phi}, and of the other properties of $\T$ listed in Proposition \ref{pr:selfsimfunda}. To avoid repetition, we do not give a proof and refer to Section \ref{sec:selfsim} for the study of these functions.

\end{proof}

We are now in position to end the proof of Theorem \ref{th:main}. The proof is lengthy and a bit repetitive. The computations are however detailed for clarity. For $|x-x^*(t)|\leq \epsilon$, let $\Gamma_{(X_b,0)}^{(X_{out},Y_{out})}$ be the part of the curve $\Gamma[x]=\{x[t](X,Y)=x\}$ joining the boundary of the upper half plane and the point $(a_{out},b_{out})$ defined in Lemma \ref{lem:lagrangian}. We define:
\be \lab{def:y*}
y^*(t,x)=\int_{\Gamma_{(X_b,0)}^{(X_{out},Y_{out})}} \frac{ds}{|\nabla x|}.
\ee
We introduce the following zones: the core, the sides, the bottom and the top of the self-similar zone, and the zones below and above, respectively:
\begin{align}
\label{autosimeulerstep1} &\mathcal Z_{co}[M_{co},\kappa_{co}]=\left\{ |x-x^*|\leq M_{co}\sft^{3/2}, \ \kappa_{co}\sft^{-1/4}\leq  y\leq 2 (\mY^*_{\mu,\nu,\iota}(\frac{ x-x^*}{\sft^{3/2}})-\kappa_{co})\sft^{-\frac 14} \right\},\\
&\mathcal Z_{si}[M_{si},\kappa_{si}]=\left\{ M_{si}\sft^{3/2}\leq |x-x^*|\leq \epsilon, \ \frac{\kappa_{si}}{|x-x^*|}\leq y\leq \frac{C_{\text{sgn}(x-x^*)}-\kappa_{si}}{|x-x^*|^{1/6}} \right\},\\
\label{autosimeulerbottom} &\mathcal Z_{bo}[K_{bo},\kappa_{bo}]=\left\{ |x-x^*|\leq \epsilon, \ K_{bo}\leq y\leq \frac{\kappa_{bo}}{|x-x^*|^{1/6}+\sft^{1/4}} \right\},\\
&\mathcal Z_{to}[M_{to},\kappa_{to}]=\left\{ |x-x^*|\leq \epsilon, \ y^*(t,x)(1-\kappa_{to})\leq y\leq y^*(t,x)-K_{to} \right\},\\
&\mathcal Z_{be}[K_{bo}]=\left\{ |x-x^*|\leq \epsilon, \ 0\leq y\leq K_{bo}\right\},\\
&\mathcal Z_{ab}[K_{to}]=\left\{ |x-x^*|\leq \epsilon, \ y^*(t,x)-K_{to}\leq y\right\}.
\end{align}
corresponding to the following picture:
\begin{center}
\includegraphics[width=11cm]{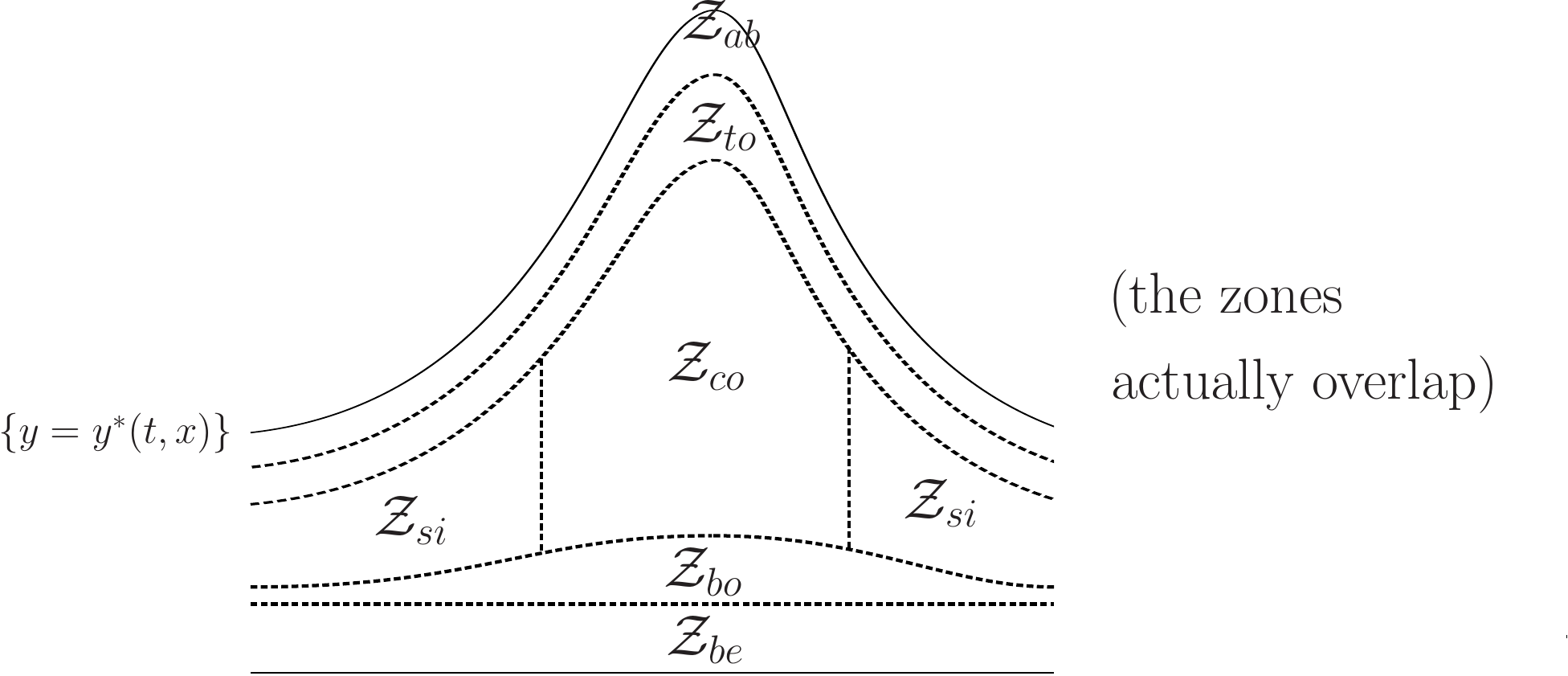}
\end{center}
Their Lagrangian counterparts are defined as follows:
\begin{align}
\label{autosimab1} & Z_{co}[N_{co}]=\left\{ |a|\leq N_{co}, \quad |b|\leq N_{co} \right\},\\
\label{autosimaabright} & Z_{si}[N_{si},\tilde N_{si}]=Z_0^c\cap \left\{ N_{si} \leq  |p^{*2}a^3+b^2|, \quad b^2\leq \tilde N_{si} |p^{*2}a^3+b^2|. \right\},\\
\label{autosimabbottom} &Z_{bo}[N_{bo},\tilde N_{bo}]=\left\{ -\delta \sft^{-3/4}\leq b\leq -N_{bo}, \quad \tilde N_{bo} |p^{*2}a^3+b^2|\leq b^2 \right\},\\
&Z_{to}[M_{to},\kappa_{to}]=\left\{ -\delta \sft^{-3/4}\leq b\leq -N_{to}, \quad \tilde N_{to} |p^{*2}a^3+b^2|\leq b^2 \right\}
\end{align}

\begin{proof}[Proof of Theorem \ref{th:main}]

\noindent We invert the characteristics, bearing in mind that the leading order term is $(\ma,\mb)$ as defined in \fref{def:mXmYT}. The change of variables preserving volume, one has:
$$
\begin{pmatrix} \pa_\mX a & \pa_\mY a \\ \pa_\mX b & \pa_\mY b \end{pmatrix}=\begin{pmatrix} \pa_b \mY &- \pa_b \mX \\ -\pa_a \mY & \pa_a \mX \end{pmatrix}
$$
Once the characteristics are inverted, we use the Taylor expansion of $u$ in Lagrangian variables near $(X_0,Y_0)$ (a direct consequence of \fref{def:abrenorm}) where $u_X$ below is evaluated at $(T,X_0,Y_0)$:
\be \lab{id:taylorugen}
u(t,X,Y)=u^*-(-u_X \bar k_1)a\sft^{\frac 12}+O\left(\sft^{\frac 34}|b|+\sft|a|^2+\sft^{\frac 32}|a| \right)
\ee
\be \lab{id:taylorpaugen}
\pa_a u =-(-u_X \bar k_1)\sft^{\frac 12}\left(1+O\left(T-t+\sft^{\frac 12}|a|+\sft^{\frac 34}|b|\right)\right), \qquad \pa_b u=O(\sft^{\frac 34}),
\ee
\be \lab{id:derivee}
\pa_\mX u =\pa_a u\pa_b\mY-\pa_b u \pa_a\mY, \qquad \pa_\mY u =-\pa_a u \pa_b \mX+\pa_b u \pa_a \mX.
\ee
We fix once for all the variables $L_1$ and $L_2$ such that Lemmas \ref{ap:calculsx}, \ref{ap:calculsx2} and \ref{lem:lagrangian} hold true.\\

\noindent \textbf{Step 1} \emph{The core of the self-similar zone}. We claim that for any $M_{co},\kappa_{co}>0$, for $\delta$ small and then $\epsilon$ small, the estimates \fref{bd:tildeu}, \fref{bd:paXtildeu} and \fref{bd:paYtildeu} hold true for $(\hat x,\hat y)\in \mathcal Z_{co}$. Fix then $(\hat x,\hat y)\in \mathcal Z_{co}$, and define $(\hat \mX,\hat \mY)$ by \fref{def:mathcalXmathcalYrenorm}. We want to invert the characteristics and to find the corresponding $\hat a$ and $\hat b$ such that $(x,y)(\hat a,\hat b)=(\hat x,\hat y)$.

Let $N_{co}\gg 1$ to be chosen later on and consider $(a,b)\in Z_{co}$. We will first show a priori estimates in $Z_{co}$, that will allow us to prove that $(\hat a,\hat b)\in Z_{co}$. This will justify the a priori estimates and prove the claim. The zone $Z_{co}$ lies in the zone $Z_1$ for $\delta$ small enough, so the corresponding estimates in Lemma \ref{ap:calculsx} apply. Moreover, as $|a|,|b|,|\mX|\lesssim 1$ one obtains from \fref{bd:xvanshen}, \fref{bd:Z1:paamX} and \fref{bd:Z1:pabmX}:
$$
\mX=a+p^{*2}a^3+b^2+O\left(\sft^{\frac{1}{4}}\right)=\mX^\Theta(a,b)+O\left(\sft^{\frac{1}{4}}\right),
$$
and similarly:
\be \lab{coreselfsim:paamX}
\pa_a \mX=1+3p^{*2}a^2+O(\sft^{\frac{1}{4}})=\pa_a\mX^\Theta+O(\sft^{\frac{1}{4}}), \ \ \pa_b \mX=2b+O(\sft^{\frac 14})=\pa_b\mX^\Theta+O(\sft^{\frac 14}).
 \ee
Since $|\mX|\lesssim 1$, one can use the identities \fref{id:mYleftcentral} and \fref{id:paamYleftcentral}. Injecting that $|\mX|,|b|\lesssim 1$ and that $1+3\Psi_1^2\left( p^* \left(\mX-b^2\right) \right)\gtrsim (1+|\tb|^{4/3})$ for $\tb\leq b$ from \fref{id:psi1infinity} in \fref{id:mYleftcentral} gives:
$$
\mY(a,b)  = 2 \int_{-\infty}^b \frac{d\tb}{1+3\Psi_1^2\left( p^* \left(\mX-b^2\right) \right)} +O\left(\int_{-\infty}^b \frac{|\tb|^{\frac 14}\sft^{\frac{3}{16}}}{1+|\tb|^{\frac 43}}d\tb  \right)+O(\sft^{\frac{1}{12}})=\mY^\Theta (a,b) +O(\sft^{\frac{1}{12}}).
$$
Injecting that $|b|,|\mX|\lesssim 1$ and $1+3\Psi_1^2\left( p^* \left(\mX-b^2\right) \right)\approx (1+|\tb|^{4/3})$ in \fref{id:paamYleftcentral} gives similarly:
\bee
 \pa_{a}\mathcal Y(a,b) &=& 6p^* \left(1+3\Psi_1^2\left(p^* \left(\mX-b^2 \right) \right) \right) \int_{-\infty}^{b}  \frac{\Psi_1\left(p^*\left(\mX-\tb^2 \right) \right)}{\left(1+3\Psi_1^2\left(p^*(\mX-\tb^2) \right) \right)^3}d\tb \\
&& +O\left(\int_{-\infty}^b \frac{\sft^{\frac{1}{12}}+|\tb|^{\frac 14}\sft^{\frac{3}{16}}}{1+\tb^{\frac{10}{3}}}d\tb \right)+O(\sft^{\frac 74})\\
&=&\pa_a \mY^\Theta (a,b)+O(\sft^{\frac{1}{12}}).
\eee
We retrieve the last partial derivative $\pa_b \mY(a,b)$ by incompressibility, using that $\pa_a \mX \approx 1$, \fref{coreselfsim:paamX} and the above identity:
$$
\pa_b \mY(a,b)=\frac{1+\pa_a\mY\pa_b\mX}{\pa_a\mX}=\frac{1+(\pa_a\mY^\Theta+O(\sft^{\frac{1}{12}}))(\pa_b\mX^\Theta +O(\sft^{\frac 14}))}{\pa_a\mX^\Theta (1+O(\sft^{\frac{1}{12}}))}=\pa_b \mY^{\Theta}+O\left(\sft^{\frac{1}{12}} \right).
$$
We are now ready to invert the characteristics. We look for a solution of the form $(a,b)=\left(\bar a+h_1,\bar b +h_2\right)$ to $(x,y)(a,b)=(x,y)$, where
$$
(\bar a,\bar b)=\left(\ma\left(\hat \mX,\hat \mY \right),\mb \left(\hat \mX,\hat \mY\right) \right)=\left(-\T \left(\hat \mX,\hat \mY \right),\mb\left(\hat \mX,\hat \mY\right) \right).
$$
As for $(\hat x,\hat y)\in \mathcal Z_{co}$, $(\hat \mX,\hat \mY)$ is in a compact zone inside the support of $\Theta$, then $(\bar a,\bar b)\in Z_{co}$ for $N_{co}$ large enough depending on $M_{co}$ and $\kappa_{co}$, so that our previous computations apply. Consider then the mapping:
$$
\Xi:(h_1,h_2)\mapsto \left(  \mX\left(\bar a+h_1,\bar b +h_2\right), \mY\left(\bar a+h_1,\bar b +h_2\right) \right).
$$
From the estimates on the derivatives done above, there holds for $|h_1|,|h_2|=O(\sft^{1/12}) $:
\be \lab{id:jacobiancenter}
\begin{pmatrix} \pa_{h_1}\Xi_1 & \pa_{h_2}\Xi_1 \\ \pa_{h_1}\Xi_2 & \pa_{h_2}\Xi_2 \end{pmatrix} = \begin{pmatrix} \pa_a \mX^{\Theta}(\bar a,\bar b) & \pa_b \mX^{\Theta}(\bar a,\bar b) \\  \pa_a \mY^{\Theta}(\bar a,\bar b) & \pa_b \mY^{\Theta}(\bar a,\bar b)  \end{pmatrix} +O(\sft^{\frac{1}{12}}).
\ee
Also, again from the computations performed above:
$$
\Xi (0,0)-\left(\hat \mX, \hat \mY \right)=O(\sft^{\frac{1}{12}}).
$$
Note that, as $\bar a$ and $\bar b$ vary in $Z_{co}$, the first term in the right hand side of \fref{id:jacobiancenter} belongs to a compact set of invertible matrices. Hence we can invert the above equation applying the implicit function Theorem, uniformly as $t \rightarrow T$ for all $(\hat x,\hat y)\in Z_{co}$: there exists $(h_1,h_2)=O(\sft^{\frac{1}{12}})$ such that $\Xi(h_1,h_2)=\left( \hat \mX,\hat \mY \right)$. Hence we inverted the characteristics and found:
$$
\hat a=\bar a+O\left(\sft^{\frac{1}{12}}\right)=-\T(\hat \mX,\hat \mY)+O\left(\sft^{\frac{1}{12}}\right), \qquad \hat b=\bar b+O(\sft^{\frac{1}{12}}).
$$
Using the Taylor expansion \fref{id:taylorugen} of $u$ and the fact that $|\hat a|,|\hat b|\lesssim 1$:
$$
u(t,\hat x,\hat y)=u(t,\hat X,\hat Y)=u(t,X_0,Y_0)-(-u_X \bar k_1)\sft^{\frac 12}(\hat a+O(\sft^{\frac 14})),
$$
so we infer from \fref{id:defparametersgeneric} that:
\be \lab{id:uthcoreselfsim}
u(t,\hat x,\hat y) =u^*+ \sft^{\frac 12} \left( \Theta_{\mu,\nu,\iota} \left(\frac{\hat x-x^*}{\sft^{\frac 32}},\frac{\hat y}{\sft^{-\frac 14}} \right)+O\left(\sft^{\frac{1}{12}} \right)\right),
\ee
which shows \fref{bd:tildeu}. Once the inversion is done, the estimates for the derivatives \fref{bd:paXtildeu} and \fref{bd:paXtildeu} follow naturally, as from \fref{id:taylorpaugen}, \fref{id:derivee}, \fref{id:pamXmYTheta} and the above estimates:
\be \lab{id:Xuthcoreselfsim}
\pa_\mX u(t,\hat x,\hat y)= -\sft^{\frac 12}\pa_b \mY^\Theta \left(1+O(\sft^{\frac{1}{12}})\right)+O(\sft^{\frac 34})=\sft^{\frac 12}\pa_\mX \Theta_{\mu,\nu,\iota}\left(\frac{\hat x-x^*}{\sft^{3/2}},\frac{\hat y}{\sft^{-1/4}}\right)+O\left(\sft^{\frac{1}{12}}\right)
\ee
\be \lab{id:Yuthcoreselfsim}
\pa_\mY u(t,\hat x,\hat y)= \sft^{\frac 12}\pa_b \mX^\Theta \left(1+O(\sft^{\frac{1}{12}})\right)+O(\sft^{\frac 34})=\sft^{\frac 12}\pa_\mY \Theta_{\mu,\nu,\iota}\left(\frac{\hat x-x^*}{\sft^{3/2}},\frac{\hat y}{\sft^{-1/4}}\right)+O\left(\sft^{\frac{1}{12}}\right)
\ee

\noindent \textbf{Step 2} \emph{The bottom of the self-similar zone}. We claim that there exists $K_{bo}*>0$ such that for all $K_{bo}\geq K_{bo}^*$, and $\kappa_{co}>0$, for $\delta$ small and then $\epsilon$ small, the estimates \fref{bd:tildeu}, \fref{bd:paXtildeu} and \fref{bd:paYtildeu} hold true for $(\hat x,\hat y)\in \mathcal Z_{bo}$.

We apply a similar strategy as in Step 1. However the leading order term of the characteristics, which is still $(\mX^\Theta,\mY^\Theta)$, becomes degenerate in this zone. Hence we first renormalise the characteristics, in order to prove that their invertibility is uniformly possible in $\mathcal Z_{bo}$. Fix $(\hat x,\hat y)\in \mathcal Z_{bo}$ meaning $|\hat{\mX} |\leq k_5\epsilon \sft^{-3/2}$ and $k_6K_{bo}\sft^{1/4} \leq \hat{\mY} \leq k_6\kappa_{bo}/ (k_5^{-1/6}|\hat \mX|^{1/6}+1)$, and consider $(a,b)\in Z_{bo}$ for $N_{bo},\tilde N_{bo}$ to be chosen later. One has from \fref{autosimabbottom} and from \fref{bd:xvanshen} that:
\be \lab{bd:amXautosimbottom}
a<0, \ \ |b|\approx |a|^{3/2}\gg 1 \ \ \text{and} \ \ |\mX|\ll |b|^2
\ee
for all $N_{bo},\tilde N_{bo}$ large enough, hence $(a,b)\in Z_1$. Consequently all the computations corresponding to the zone $Z_1$ in Lemma \ref{ap:calculsx} apply. The identities \fref{bd:Z1:x}, \fref{bd:Z1:paamX} and \fref{bd:Z1:pabmX} give, when injecting the bounds \fref{bd:amXautosimbottom}:
\be \lab{bd:bottom:x}
\mX=\left(a+p^{*2}a^3+b^2\right)+O\left(\sft^{\frac{1}{4}}|b|^{2+\frac{1}{3}}\right)=\mX^\Theta(a,b)+O\left(\sft^{\frac{1}{4}}|b|^{2+\frac{1}{3}}\right),
\ee
\be \lab{bd:bottom:paamX}
\pa_a \mX=(1+3p^{*2}a^2)\left(1+O\left(|b|^{\frac 13}\sft^{\frac 14}\right) \right)=\pa_a\mX^\Theta (a,b)\left(1+O\left(|b|^{\frac 13}\sft^{\frac 14}\right) \right),
\ee
\be \lab{bd:bottom:pabmX}
\pa_b \mX=2b\left(1+O\left(|b|^{\frac 13}\sft^{\frac 14} \right) \right)=\pa_b\mX^\Theta (a,b)\left(1+O\left(|b|^{\frac 13}\sft^{\frac 14} \right) \right).
\ee
We now compute $\mY$. Since $b<0$ and $|\mX|\ll |b|^2$, \fref{id:mYleftcentral} and \fref{id:paamYleftcentral} apply. As $|\mX|\ll |b|^2$ and $b<0$, then for $\tb\leq b$ one has from \fref{bd:Psi11} and \fref{bd:amXautosimbottom}:
\be \lab{bd:psi1bottom}
|\tb|\gg 1, \ \ \Psi_1^2(p^*(\mX-\tb^2))\approx |\tb|^{4/3} \ \ \text{and} \ \ |\mX|\ll |\tb|.
\ee
Injecting the above bound and \fref{bd:amXautosimbottom} in the identity \fref{id:mYleftcentral}, and using \fref{bd:calculTheta:bottom:pamXmY}, gives:
\bea
\non \mY &=& 2\int_{-\infty}^b \frac{d\tb}{1+3\Psi_1^2\left(p^*\left(\mX-\tb^2\right)\right)}+O\left(\int_{-\infty}^b \frac{\sft^{\frac{1}{12}}|\tb|^{\frac{1}{9}}d\tb}{|\tb|^{\frac 43}} \right)+O(\sft^{\frac 14}) \\
\lab{bd:bottom:mY} &=& \mY^\Theta(a,b)\left(1+O\left(\sft^{\frac{1}{12}}|b|^{\frac{1}{9}} \right) \right).
\eea
Similarly, injecting \fref{bd:psi1bottom} and \fref{bd:amXautosimbottom} in the integral \fref{id:paamYleftcentral} giving $\pa_a\mY$, and using \fref{id:paamYTheta} gives:
\begin{align*}
 & \pa_{a}\mathcal Y(a,b)  \\
 =& 6p^* \left(1+3\Psi_1^2\left(p^*\left(\mX-b^2 \right) \right) \right)\left( \int_{-\infty}^{b}  \frac{\Psi_1\left(p^*\left(\mX-\tb^2 \right) \right)}{\left(1+3\Psi_1^2\left(p^*(\mX-\tb^2) \right) \right)^3} +O\left(\int_{-\infty}^b \frac{\sft^{\frac{1}{12}}|\tb|^{\frac{1}{9}}}{|\tb|^{\frac{10}{3}}}d\tb \right) +O(\sft^{\frac 74})\right) \\
  =& \pa_a\mY^\Theta (a,b)\left(1+O(\sft^{\frac{1}{12}}|b|^{\frac{1}{9}}) \right)
\end{align*}
We compute the last partial derivative $\pa_b \mY(a,b) $ via the incompressibility, using \fref{bd:bottom:paamX}, \fref{bd:bottom:pabmX} and the above bound, and \fref{bd:calculTheta:bottom:pamXmY}:
$$
\pa_b \mY(a,b) = \frac{1+\pa_a \mY \pa_b \mX}{\pa_a \mX}=  \frac{1+\pa_a \mY^\Theta \pa_b \mX^\Theta \left(1+O(|b|^{\frac{1}{9}}\sft^{\frac{1}{12}})\right)}{\pa_a \mX^\Theta \left(1+O(|b|^{\frac{1}{9}}\sft^{\frac{1}{12}})\right)} = \pa_b \mY^\Theta (a,b) \left(1+O(\sft^{\frac{1}{12}}|b|^{\frac{1}{9}})\right).
$$
We can now invert the characteristics. We look for a solution of $(x,y)(\hat a,\hat b)=(\hat x,\hat y)$ of the form $(a,b)=(\bar a+h_1|\bar b|^{\frac 23},\bar b (1+h_2))$ close to
$$
(\bar a,\bar b)=\left(\ma\left(\hat \mX,\hat \mY \right),\mb\left(\hat \mX,\hat \mY\right) \right)=\left(-\T \left(\hat \mX,\hat \mY \right),\mb\left(\hat \mX,\hat \mY\right) \right).
$$
and with a priori bound $|h_1|,|h_2|\lesssim |\bar b|^{1/12}\sft^{1/16}$. From (i) in Lemma \ref{ap:calculstheta}, since $(\hat x,\hat y)\in \mathcal Z_{bo}$:
\be \lab{bd:barab:bottom}
\bar b<0, \ \ |\bar b|\approx \frac{1}{\sft^{\frac 34}\hat y^3}, \ \ |\bar a|\approx \frac{1}{\sft^{\frac 12}\hat y^2}, \ \  |p^{*2}\bar a^{3}+\bar b^{ 2}|\leq \frac{1}{\tilde N_{bo}^*} \bar b^{2}
\ee
for all $\kappa_{bo}$ small enough, for some $\tilde N_{bo}^*(\kappa_{bo})\rightarrow \infty$ as $\kappa_{bo}\rightarrow 0$. Since $\hat \mY\lesssim  \kappa_{bo} \langle \mX \rangle^{-1/6}$ we obtain from \fref{bd:barab:bottom} that $\bar b\rightarrow -\infty$ as $\kappa_{bo}\rightarrow 0$. Since $\hat y \geq  K_{bo}$ we get from \fref{bd:barab:bottom} that $\bar b\geq -\delta \sft^{-3/4}$ if $\delta\geq K_{bo}^{-3} $. Thus, by taking $K_{bo}$ large and $\kappa_{bo}$ small, there always exist some large $N_{bo},\tilde N_{bo}$ and small $\delta$ such that $(\bar a,\bar b)\in Z_{bo}$. Consider now the mapping:
$$
\Xi:(h_1,h_2)\mapsto \left( \frac{\mX\left(\bar a+h_1|\bar b|^{\frac 23},\bar b (1+h_2)\right)}{\bar b^2}, \mY\left(\bar a+h_1|\bar b|^{\frac 23},\bar b (1+h_2)\right)|\bar b|^{\frac 13} \right).
$$
From the computations on the derivatives above, \fref{bd:calculTheta:bottom:pamXmY} and \fref{bd:barab:bottom}:
$$
\begin{pmatrix} \pa_{h_1}\Xi_1 & \pa_{h_2}\Xi_1 \\ \pa_{h_1}\Xi_2 & \pa_{h_2}\Xi_2 \end{pmatrix} = \begin{pmatrix} |\bar b|^{-\frac 43}\pa_a \mX^{\Theta}(\bar a,\bar b) & |\bar b|^{-1} \pa_b \mX^{\Theta}(\bar a,\bar b) \\ |\bar b| \pa_a \mY^{\Theta}(\bar a,\bar b) & |\bar b|^{\frac 43}\pa_b \mY^{\Theta}(\bar a,\bar b)  \end{pmatrix} +O(\hat y^{-\frac 13}).
$$
From \fref{bd:calculTheta:bottom:pamXmY} the above matrix is bounded, and the leading order term in the right hand side is close to a fixed invertible matrix in the whole zone under consideration. From \fref{bd:bottom:x} and \fref{bd:bottom:mY}:
$$
\Xi (0,0)-\left( \frac{\hat \mX}{|\bar b|^2},\hat \mY|\bar b|^{\frac 13}\right)=O(|\bar b|^{\frac{1}{9}}\sft^{\frac{1}{12}})=O(\hat y^{-\frac 13}).
$$
Therefore, one can apply the implicit function Theorem: there exists $(h_1,h_2)=O(\hat y^{-1/3})$ such that $\Xi(h_1,h_2)=\left( \frac{\hat \mX}{|\bar b|^2},\hat \mY|\bar b|^{\frac 13} \right)$ or equivalently for the characteristics: 
$$
\hat a=\bar a+h_1|\bar b|^{\frac 23}=\bar a\left(1+O(\hat y^{-\frac 13})\right), \ \ \hat b=\bar b(1+h_2))=\bar b \left(1+O(\hat y^{-\frac 13}) \right),
$$
where we used that $|\bar a|\approx |\bar b|^{2/3}$. Hence, since, using the Taylor expansion \fref{id:taylorugen} of $u$ and the fact that $|\hat a|^3\approx |\hat b|^2\gg 1$:
$$
u(t,\hat x,\hat y)=u(t,\hat X,\hat Y)=u(t,X_0,Y_0)-(-u_X \bar k_1)\sft^{\frac 12}\hat a(1+O(|\hat a|\sft^{1/2})),
$$
so from \fref{id:defparametersgeneric} we infer that:
\be \lab{id:uthbottom}
u(t,\hat x,\hat y) =u^*+ \sft^{\frac 12} \Theta_{\mu,\nu,\iota} \left(\frac{\hat x-x^*}{\sft^{\frac 32}},\frac{\hat y}{\sft^{-\frac 14}} \right)\left(1+O\left(\frac{1}{\hat y^{\frac{1}{3}}}\right)\right),
\ee
which shows \fref{bd:tildeu}. Again, the estimates for the derivatives \fref{bd:paXtildeu} and \fref{bd:paYtildeu} follow from \fref{id:taylorpaugen}, \fref{id:derivee}, \fref{id:pamXmYTheta}, \fref{bd:calculTheta:bottom:pamXmY} and the above estimates:
\bea
\non \pa_{\mX} u &=& \pa_a u \pa_b\mY-\pa_b u \pa_a\mY=-\sft^{\frac 12}\pa_b\mY^{\Theta}\left(1+O\left(\frac{1}{|\hat y|^\frac 13} \right)\right)+O\left(\sft^{\frac 34}|\hat b|^{-1} \right)\\
\lab{id:Xuthbottom} &=& \sft^{\frac 12} \pa_\mX \T_{\mu,\nu,\iota}\left(1+O\left(\frac{1}{|\hat y|^{\frac 13}} \right)\right),
\eea
\bea
\non \pa_{\mY} u &=&- \pa_a u \pa_b\mX+\pa_b u \pa_a\mX=\sft^{\frac 12}\pa_b\mX^{\Theta}\left(1+O\left(\frac{1}{|\hat y|^\frac 13} \right)\right)+O\left(\sft^{\frac 34}|\hat b| \right)\\
\lab{id:Yuthbottom} &=& \sft^{\frac 12} \pa_\mY \T_{\mu,\nu,\iota}\left(1+O\left(\frac{1}{|\hat y|^{\frac 13}} \right)\right).
\eea

\noindent \textbf{Step 3} \emph{The sides of the self-similar zone}. We claim that there exists $\kappa_{si}*>0$ such that for all $0<\kappa_{si}\leq \kappa_{si}^*$, for $M_{si}$ large enough, then for $\delta$ small and then $\epsilon$ small, the estimates \fref{bd:tildeu}, \fref{bd:paXtildeu} and \fref{bd:paYtildeu} hold true for $(\hat x,\hat y)\in \mathcal Z_{si}$.

Fix then $(\hat x,\hat y)\in \mathcal Z_{si}$ corresponding to $k_5 \epsilon   \sft^{-3/2}\leq \hat \mX \leq -k_5M_{si} $ and $k_6k_5^{1/6} \kappa_{si} |\hat \mX|^{-1/6}\leq \hat \mY \leq k_6k_5^{1/6}(C_--\kappa_{si}) |\hat \mX|^{-1/6}$ on the left side, or to $k_5 M_{si}\leq \hat \mX \leq k_5 \epsilon \sft^{-3/2} $ and $k_6k_5^{1/6}  \kappa_{si} \hat \mX^{-1/6}\leq \hat \mY \leq k_6k_5^{1/6} (C_+-\kappa_{si}) \hat \mX^{-1/6}$ on the right side. Let $(a,b)\in Z_{si}$ for $N_{si},\tilde N_{si}$ to be fixed large later on. In $Z_{si}$, from \fref{autosimaabright}, $|a|\lesssim (|p^{*2}a^3+b^2|+b^2)^{1/3}\lesssim \tilde N_{si}^{1/3}|p^{*2}a^3+b^2|^{1/3}$, so one infers from \fref{bd:xvanshen} that:
\bee
\mX&=&(p^{*2}a^3+b^2)\left(1+O\left(\frac{|a|}{|p^{*2}a^3+b^2|}+\frac{|b|\sft^{\frac 14}}{|p^{*2}a^3+b^2|}+|p^{*2}a^3+b^2|^{\frac{1}{24}}\sft^{\frac 16} \right) \right)\\
&=&(p^{*2}a^3+b^2)\left(1+O\left(\tilde N_{si}^{\frac 13}N_{si}^{-\frac 23}+|p^{*2}a^3+b^2|^{\frac{1}{24}}\sft^{\frac 16} \right) \right)
\eee
which implies that $|\mX|\approx |p^{*2}a^3+b^2|$ provided $N_{si}\ll \tilde N_{si}^2$. Hence the bounds in this zone:
\be \lab{bd:abmXautosimsides}
|\mX|\gtrsim N_{si}\gg 1, \ \ |b|\lesssim |\mX|^{\frac 12}, \ \ |a|\lesssim |\mX|^{\frac 13}.
\ee
Injecting these bounds in \fref{bd:xvanshen}, \fref{bd:paaMxvanshen} and \fref{bd:pabmXvanshen}, one gets:
$$
\mX =\left(a+p^{*2}a^3+b^2\right)\left(1+O(\sft^{\frac 14}|\mX|^{\frac 16})\right)=\mX^\Theta(a,b)\left(1+O(\sft^{\frac 14}|\mX|^{\frac 16})\right),
$$
\be \lab{bd:sides:paamX}
\pa_a \mX = 1+3p^{*2}a^2+O(|\mX|^{\frac 56}\sft^{\frac 14})=\pa_a \mX^\Theta+O(|\mX|^{\frac 56}\sft^{\frac 14}) ,
\ee
\be \lab{bd:sides:pabmX}
\pa_b \mX=2b+O(\sft^{\frac 14}|\mX|^{\frac 23})=\pa_b \mX^\Theta +O(\sft^{\frac 14}|\mX|^{\frac 23}).
\ee
In addition, the estimates \fref{bd:sides:mY} and \fref{bd:sides:pabmY} apply and the following estimates are true for $\mY$:
$$
\mY= \mY^\Theta (a,b) \left(1+O(\sft^{\frac{1}{12}}|\mX|^{\frac{1}{18}}) \right), \ \ \pa_a \mY = \pa_\ma\mY +O(\sft^{\frac{1}{12}}|\mX|^{-\frac 12+\frac{1}{18}}),
$$
$$
\pa_b \mY=\pa_b\mY^\Theta +O\left(|\mX|^{-\frac 23 +\frac{1}{18}}\sft^{\frac{1}{12}} \right).
$$
With these identities we can now invert the characteristics, and again, will renormalise in a suitable way the perturbation problem. We look for a solution $(\hat a,\hat b)$ of the form $(a,b)=\left(\bar a+h_1|\mX|^{1/3},\bar b +h_2|\mX|^{1/2}\right)$ to $(x,y)(\hat a,\hat b)=(\hat x,\hat y)$, where
$$
(\bar a,\bar b)=\left(\ma\left(\hat \mX,\hat \mY \right),\mb\left(\hat \mX,\hat \mY\right) \right)=\left(-\T \left(\hat \mX,\hat \mY \right),\mb\left(\hat \mX,\hat \mY\right) \right).
$$
As $\hat x,\hat y$ is in the zone \fref{autosimeulerbottom} this produces thanks to (ii) in Lemma \ref{ap:calculstheta}:
$$
|\bar b|^2\leq C(\kappa_{si})|p^{*2}\bar a^3+\bar b^2|, \ \ |p^{*2}\bar a^3+\bar b^2|\gtrsim M_{si}.
$$
Hence, for any $\kappa_{si}$ small enough, there exists $N_{si},\tilde N_{si}$ large with $N_{si}\gg \tilde N_{si}^{1/2}$ such that for $M_{si}$ large enough, if $(\hat a,\hat b)\in \mathcal Z_{si}$ then $(\bar a,\bar b)\in Z_{si}$, which shows the validity of all the computations done so far. From the above computations one infers that:
$$
\mX(\bar a,\bar b) = \hat \mX \cdot \left(1+O(|\hat x-x^*|^{\frac 16})\right), \ \ \mY(\bar a,\bar b)=\hat \mY\cdot \left(1+O\left(|\hat x- x^*|^{\frac{1}{24}} \right)\right).
$$
Consider the mapping:
$$
\Xi:(h_1,h_2)\mapsto \left( \frac{\mX\left(\bar a+h_1|\bar \mX|^{\frac 13 },\bar b +h_2|\bar \mX|^{\frac 12}\right)}{|\bar \mX|}, \mY\left(\bar a+h_1|\bar \mX|^{\frac 13},\bar b +h_2|\bar \mX|^{\frac12}\right)|\bar \mX|^{\frac 16} \right)
$$
From the estimates on the derivatives done above, there holds for $|h_1|,|h_2|=O(|\hat x|^{1/18}) $:
$$
\begin{pmatrix} \pa_{h_1}\Xi_1 & \pa_{h_2}\Xi_1 \\ \pa_{h_1}\Xi_2 & \pa_{h_2}\Xi_2 \end{pmatrix} = \begin{pmatrix} \frac{1}{|\bar \mX|^{\frac 23}}\pa_a \mX^{\Theta}(\bar a,\bar b) &\frac{1}{|\bar \mX|^{\frac 12}} \pa_b \mX^{\Theta}(\bar a,\bar b) \\ |\bar \mX|^{\frac 12} \pa_a \mY^{\Theta}(\bar a,\bar b) & |\bar \mX|^{\frac 23} \pa_b \mY^{\Theta}(\bar a,\bar b)  \end{pmatrix} +O(|\hat x|^{\frac{1}{18}}),
$$
and the Jacobian matrix of $\Xi$ must preserve volume. From the computations above,
$$
\Xi (0,0)-\left( \frac{\hat \mX}{|\bar \mX| }, \hat \mY |\bar \mX|^{\frac 16} \right)=O(|\hat x|^{\frac{1}{18}}).
$$
Note that, due to \fref{eq:selfsimsides}, as $\mX \rightarrow +\infty$, in the zone that we consider \fref{autosimaabright}, the leading order matrix in the identity giving the Jacobian of $\Xi$ in the right hand side belongs to a compact set of matrices with determinant 1. Hence, applying the inverse function Theorem, we can invert the above equation, uniformly as $\epsilon \rightarrow 0$ in the zone that we consider: there exists $(h_1,h_2)=O(|\hat x|^{\frac{1}{24}})$ such that $\Xi(h_1,h_2)=\left( \frac{\hat \mX}{|\bar \mX|}, \hat \mY|\bar \mX|^{\frac 16} \right)$. Hence we inverted the characteristics and found:
$$
\hat a=\bar a\left(1+O\left(|\hat x|^{\frac{1}{18}}\right)\right), \ \ \hat b=\bar b\left(1+O\left(|\hat x|^{\frac{1}{18}}\right)\right).
$$
Hence, since, using the Taylor expansion \fref{id:taylorugen} of $u$ and the fact that $|\hat a|^3,|\hat b|^2\lesssim |\hat \mX|$:
$$
u(t,\hat x,\hat y) = u(t,X,Y)=u^*-(-u_X \bar k_1)\sft^{\frac 12}\hat a+O\left(|\hat b|\sft^{\frac 34}+|\hat a|^{3}\sft^{\frac 32} \right) = -(-u_X \bar k_1)\sft^{\frac 12}\left(\hat a+O\left(|\hat \mX|^{\frac 13+\frac{1}{6}}\sft^{\frac 14}\right) \right)
$$
we infer from \fref{id:defparametersgeneric} that:
\be \lab{id:uthside}
u(t,\hat x,\hat y)  = u^*+ \sft^{\frac 12} \left(\Theta_{\mu,\nu,\iota} \left(\frac{\hat x-x^*}{\sft^{\frac 32}},\frac{\hat y}{\sft^{-\frac 14}} \right)+O\left( \frac{|\hat x-x^*|^{\frac{1}{3}+\frac{1}{18}}}{\sft^{\frac 12}}\right) \right)\\
\ee
which shows \fref{bd:tildeu}. The estimates for the derivatives \fref{bd:paXtildeu} an \fref{bd:paYtildeu} follow from \fref{id:taylorpaugen}, \fref{id:derivee}, \fref{id:pamXmYTheta}, \fref{bd:calculTheta:bottom:pamXmY} and the above estimates:
\bea
\non \pa_{\mX} u &=& \pa_a u \pa_b\mY-\pa_b u \pa_a\mY\\
\non &=&-\sft^{\frac 12}\pa_b\mY^{\Theta}+O\left(\sft^{\frac 12+\frac{1}{12}}|\mX|^{-\frac 23+\frac{1}{18}}\right)\left(1+O\left(\frac{1}{|\hat y|^\frac 13} \right)\right)+O\left(\sft^{\frac 34}|\mX|^{-\frac 12} \right)\\
\lab{id:Xuthside} &=& \sft^{\frac 12} \pa_\mX \T_{\mu,\nu,\iota}+O\left(\sft^{\frac 32}|x|^{-\frac 23+\frac{1}{18}} \right)
\eea
\bea
\non \pa_{\mY} u &=&- \pa_a u \pa_b\mX+\pa_b u \pa_a\mX=\sft^{\frac 12}\pa_b\mX^{\Theta}+O\left(\sft^{\frac 34} |\mX|^{\frac 23} \right)+O\left(\sft^{\frac 34}|\mX|^{\frac 23}\right)\\
\lab{id:Yuthside} &=& \sft^{\frac 12} \pa_\mY \T_{\mu,\nu,\iota}+O\left(\sft^{-\frac 14}|x|^{\frac 12+\frac 16} \right)
\eea

\noindent \textbf{Step 4} \emph{End of the proof of \fref{bd:tildeu}, \fref{bd:paXtildeu} and \fref{bd:paYtildeu}.} These stability estimates concerning the self-similar zone and the bottom, are direct consequences of Step 1, 2, and 3, upon choosing first the parameter $\kappa_{bo}$ and $K_{bo}$, then $\kappa_{si}$ and $M_{si}$, and finally $M_{co}$, and then taking $\delta$ small and $\epsilon $ small, to ensure the zones overlap and any point $(\hat x,\hat y)$ belongs to at least one zone. \\

\noindent \textbf{Step 5} \emph{The displacement line, proof of \fref{th:y*}}. We invest here the properties of $y^*$ defined by \fref{def:y*}. The analogue of \fref{id:mYsideselfsim} adapts, namely:
$$
y^*(x)=k_6^{-1}\frac{\mY^{\Theta}(a_{out},b_{out})}{\sft^{\frac 14}}\left(1+O\left(\sft^{\frac{1}{12}}+|\mX|^{\frac{1}{18}}\sft^{\frac{1}{12}} \right)\right).
$$
By definition, as $b_{out}=\delta \sft^{-3/4}$:
$$
\mY^{\Theta}(a_{out},b_{out})  = \int_{-\infty}^{\delta \sft^{-3/4}} \frac{d\tb}{1+3\Psi_1^2\left(p^*\left(\mX-\tb^2\right) \right)} = 2\mY^*(\mX)+O(\sft^{\frac 14}),
$$
as for $b\geq \delta \sft^{-3/4}$ there holds $|\mX|\ll b$ for $\epsilon$ small enough, so that $\Psi_1^2\left(p^*\left(\mX-\tb^2\right) \right)\approx b^{4/3}$ from \fref{id:psi1infinity}. Therefore, from \fref{def:mXmYT}:
$$
y^*(x)=k_6^{-1}\frac{2\mY^*\left(\frac{k_5(x-x^*)}{\sft^{\frac 32}} \right)}{\sft^{\frac 14}}\left(1+O\left(\sft^{\frac{1}{12}}+|x|^{\frac{1}{18}}\right)\right)= \frac{2\mY^*_{\mu,\nu,\iota}\left(\frac{x-x^*}{\sft^{\frac 32}} \right)}{\sft^{\frac 14}}\left(1+O\left(\sft^{\frac{1}{12}}+|x|^{\frac{1}{18}}\right)\right).
$$
This shows \fref{th:y*}.\\

\noindent \textbf{Step 6} \emph{Reconnection functions: proof of \fref{id:asymptotiquefetg}}. The reconnection functions $f$ and $g$ are defined the following way. Fix $t=T$, and consider the curve $\Gamma:=\{x(T,X,Y)=x^*(T) \}$. We split it in two parts. $\Gamma_{bot}$ is the bottom part which goes from the boundary of the upper half plane to $(X_0,Y_0)$, and $\Gamma_{top}$ which is the top part from $(X_0,Y_0)$ and beyond. We change variables, using \fref{def:ks}:
$$
\begin{pmatrix}A \\B \end{pmatrix} =  \begin{pmatrix} k_1  & k_2  \\ -k_3  &k_4  \end{pmatrix}  \begin{pmatrix}X-X_0 \\Y-Y_0 \end{pmatrix},
$$
so that from the Taylor expansion of $x$ \fref{id:taylor:xt}, for $|A|,|B|\ll1$:
$$
x(T,X,Y)=x^*(T)+p^{*2}A^3+B^2+O(|A|^4+|A||B|^2+|B|^3)
$$
which extends naturally to higher order derivatives. Therefore, near $(X_0,Y_0)$, on $\Gamma$ one has:
\be \lab{id:A:bottomreconnect}
A=-\frac{1}{p^{*\frac 23}}B^{\frac 23}\left(1+O(|B|^{\frac 23}) \right).
\ee
Denote by $(X_0^*,0)$ the point of the boundary such that $x(T,X_0^*,0)=x^*(T)$, by $(X_{in},Y_{in})$ that for which $B=-\delta$, and by $(X_{out},Y_{out})$ that for which $B=\delta$. On $\Gamma_{bot}$ for $|B|\ll 1$ and $B<0$ we compute $y$ using \fref{id:intevolume1} and split:
\bee
y & = & \int_{\Gamma_{(X_0^*,0)}^{(X_1,Y_1)}} \frac{ds}{|\nabla x(T)|}+k_6^{-1}\int_{-\delta}^B \frac{d\tilde B}{\pa_A x} \ = \ O(1)+\frac{1}{k_63p^{*\frac 23}}\int_{-\delta}^B \frac{(1+O(|\tilde B|^{\frac 23}))d\tilde B}{\tilde B^{\frac 43}}\\
&=& O(1)-\frac{1}{k_6 p^{*\frac 23}} \frac{1}{B^{\frac 13}}.
\eee
This shows that $y\rightarrow \infty$ as $(X,Y)$ approaches $(X_0,Y_0)$ on $\Gamma_b$. Given any $y\geq 0$, there thus exists a unique point $\phi(y)$ on $\Gamma_b$ between the boundary of $\mH$ and $(X_0,Y_0)$ such that $y(X,Y)=y$ and we define:
$$
f(y)=u(T,\phi(y)).
$$
From the above expansion for $y$ one obtains that, writing $\phi(y)$ in $(A,B)$ variables:
$$
B(\phi(y))\sim -\frac{1}{y^3k_6^3p^{*2}} \ \ \text{as} \ \ y\rightarrow \infty,
$$
which, from \fref{id:taylorugen}, \fref{id:A:bottomreconnect} and \fref{id:defparametersgeneric} gives:
$$
f(y)-u ^*=u(T,\phi(y))-u ^*\sim \frac{-u_X \bar k_1}{y^2k_6^2p^{*2}}=\frac{1}{y^2} \frac{\mu \nu^2}{p^{*2}} \ \ \text{as} \ y\rightarrow \infty
$$
and the first part of \fref{id:asymptotiquefetg} is proved. For $(X,Y)\in \Gamma_{top}$ now, we define:
$$
\tilde y(X,Y)=\int_{\Gamma_{(X_{out},Y_{out})}^{(X,Y)}} \frac{ds}{|\nabla x[T]|}
$$
where this integral is taken with positive sign if $(X,Y)$ is after $(A_{out},B_{out})$ on $\Gamma$, and with negative sign if it is between $(A_{out},B_{out})$ and $(X_0,Y_0)$. It follows from \fref{id:A:bottomreconnect} that:
$$
\tilde y(A_{out},B_{out})=0, \qquad \lim_{(X,Y)\in \Gamma_{top}, \ |(X,Y)|\rightarrow \infty} \tilde y(X,Y)=\infty, \qquad \tilde y \sim -\frac{1}{k_6p^{*\frac 23}} \frac{1}{B^{\frac 13}} \ \text{as} \ B\rightarrow 0^+.
$$
Hence $(X,Y)\mapsto \tilde y$ defines a $C^3$ diffeomorphism between $\Gamma_{top}$ and $\mathbb R$. Given any $\tilde y\in \mathbb R$, there exists a unique point $\tilde \phi(\tilde y)$ on $\Gamma_{top}$ such that $\tilde y(\tilde \phi (\tilde y))=\tilde y$ and we thus define:
\be \lab{eq:deftildephi}
g(\tilde y)=u(T,\tilde \phi(\tilde y)).
\ee
From \fref{id:taylorugen} and \fref{id:A:bottomreconnect} one obtains:
$$
g(\tilde y)-u ^*\sim  \frac{-u_X \bar k_1}{y^2k_6^2p^{*2}}=\frac{1}{y^2} \frac{\mu \nu^2}{p^{*2}} \ \ \text{as} \ y\rightarrow -\infty.
$$
Denoting by $X^E\mapsto x^E$ the Lagrangian to Eulerian map of the Bernouilli equation \fref{eq:bernouilli} (with $\dot x^E=u^E(t,x^E)$ and $x^E(0,X^E)=X^E$) then, at time $T$, $x^E[T]$ is a diffeomorphism since $(u^E,p^E)$ is a globally regular solution of \fref{eq:bernouilli}. From the second equation in \fref{2DPrandtlp}, we get that for $\tilde y\rightarrow \infty$, $X(\tilde \phi(\tilde y))\rightarrow X^E(T,x^*(T))$. Hence $g(\tilde y)\rightarrow u^E(x^*(T))$ as $\tilde y \rightarrow \infty$. This and the above equation proves the second identity in \fref{id:asymptotiquefetg}.\\

\noindent \textbf{Step 7} \emph{Reconnection: proof of \fref{id:recobottom} in $\mathcal Z_{be}$, \fref{id:recotop} in $\mathcal Z_{ab}$, \fref{id:topselfsim} and \fref{bd:tildevtop} in $\mathcal Z_{to}$}. For any $K_{bo}>0$, as $(t,x)\rightarrow (T,x^*)$, the characteristics map $(X,Y)\mapsto (x,y)$ still defines a diffeomorphism when restricted to the preimage of $y<K_{bo}$, since it avoids a size one zone near $(X_0,Y_0)$. Hence the convergence of $u$ to $f$ in $Z_{be}$ follows from a direct continuity argument. This shows the reconnection at the bottom \fref{id:recobottom}.\\

\noindent We now turn to the reconnection at the top. For any $K_{to}>0$, the zone $y^*-K_{to}\leq y \leq y^*+K_{to}$ corresponds in Lagrangian variable to a zone which stays at a distance $1$ away from $(X_0,Y_0)$, and hence where the parametrisation of the curves $\Gamma:=\{x(X,Y)=Cte \}$ and $\nabla x$ remain uniformly regular as $t\rightarrow T$. From this fact and the definition \fref{def:y*} of $y^*$, as $(x,t)\rightarrow (T,x^*)$, the inverse of the characteristics map $(x,y)\mapsto (X,Y)$ is such that:
$$
(X,Y)(x,y)\rightarrow \tilde \phi(y-y^*)
$$
where $\tilde \phi$ is defined by \fref{eq:deftildephi}. Hence from \fref{eq:deftildephi} one has $u(t,X,Y)\rightarrow g(y-y^*(x))$ in this zone which proves \fref{id:recotop} for $y^*-K_{to}\leq y \leq y^*+K_{to}$. For $y\geq y^*+K_{to}$, we have that $X\rightarrow X^E(T,x^*(T))$ where $X^E$ is defined in Step 6, and $Y\rightarrow \infty$, uniformly as $K_{to}\rightarrow \infty$ for $|x-x^*|\leq \epsilon$. Thus $u(t,x,y)\rightarrow u^E(T,x^*(T))$ uniformly. This proves \fref{id:recotop} for $y \geq y^*+K_{to}$ since $g(\tilde y)\rightarrow u^E(T,x^*(T))$ for $\tilde y \rightarrow \infty$. Hence \fref{id:recotop} is proved.\\

\noindent We now turn to the convergence in the top part of the self-similar zone $Z_{to}$. Let $K_{to}\gg 1$ be large, and then $0<\kappa_{to},\epsilon \ll 1$ be small, and consider $(\hat x,\hat y)\in \mathcal Z_{to}$, corresponding to
\be \lab{autosimeulertop}
|\hat \mX|\leq  k_5\epsilon\sft^{-\frac 32}, \ \ k_6\sft^{\frac 14}y^*(t,x)(1-\kappa_{to}) \leq \hat \mY \leq  k_6\sft^{\frac 14}(y^*(x)-K_{to}),
\ee
Define $\bar b$ the following way:
$$
\bar b=M(1+|\hat \mX|^{\frac 12})
$$
where $M\gg 1$ will be chosen later on. One has for $\epsilon$ small enough depending on $\delta$ that $|\bar b|\gg |\hat \mX|^{1/2}$. We now look for the solution $\bar a$ to $\mX(\bar a,\bar b)=\hat \mX$. In the zone $a<0$ with $|a|\approx |\bar b|^{2/3}$ one is in $Z_1$ and therefore from \fref{bd:Z1:paamX} one has $\pa_a \mX \approx |\bar b|^{4/3}$. Hence, as $|\mX|\ll |b|^{2}$ this condition on the derivative via the intermediate value Theorem implies that there exists $\bar a<0$ with $|\bar a|\approx |\bar b|^{2/3}$ such that $\mX(\bar a,\bar b)=\hat \mX$. From Lemma \ref{lem:lagrangian}, the curve $\Gamma[\hat x]$, before leaving $Z_0^c$ by the point $(a_{out},b_{out})$, is in $Z_1$ where it can be parametrised with the variable $\tb$ as $\ta=\ta(\mX,\tb)$. Either in case (i) of Lemma \ref{lem:lagrangian} it always stayed in $Z_1$, or in case (ii) before that stage the curve was in $Z_2$.

As $(\bar a,\bar b)$ belongs to $Z_1$, this means that in any case, one can parametrise with the variable $\tb$ the last part of $\Gamma$ before exiting $Z_0^c$, as a curve $\ta=\ta (\mX,\tb)$ for $\bar b\leq \tb \leq \delta \sft^{3/4}$. For $(a,b)$ along this curve there holds by definition of $y^*$ \fref{def:y*}:
$$
\mX(a,b)=\hat \mX, \ \ \mY(a,b)=k_6\sft^{\frac 14}y^*(\hat x)-\int_{b}^{\frac{\delta}{\sft^{3/4}}} \frac{d\tb}{\pa_a \mX(\ta(\hat \mX,b),\tb)}.
$$
As this curve lies in $Z_1$ from \fref{bd:Z1:paamX}, and from the inequality $|\hat \mX|\ll |\tb|$ for $b\leq \tb\leq \delta \sft^{-3/4}$:
$$
\mY(a,b)=k_6 \sft^{\frac 14}y^*(\hat x)-\int_{b}^{\frac{\delta}{\sft^{3/4}}} \frac{d\tb}{1+3\Psi_1^2\left(p^*(\hat \mX-\tb^2)\right)}\left(1+O(|\tb|^{\frac{1}{9}}\sft^{\frac{1}{12}}) \right).
$$
An expansion at infinity of $\Psi_1$ gives that:
$$
\frac{1}{1+3\Psi_1^2\left(p^*(\hat \mX-\tb^2)\right)}\left(1+O(|\tb|^{\frac{1}{9}}\sft^{\frac{1}{12}}) \right)=\frac{1}{p^{*\frac 23}3\tb^{\frac 43}}+O\left(\frac{\sft^{\frac{1}{12}}|\tb|^{\frac{1}{9}}}{|\tb|^{\frac 43}} +\frac{|\hat \mX||\tb|^{-2}}{|\tb|^{\frac 43}}\right),
$$
so that:
$$
\mY(a,b)=k_6\sft^{\frac 14}y^*(t,\hat x)-\frac{k_6}{p^{*\frac 23}b^{\frac 13}}+O\left(|b|^{-\frac 13+\frac{1}{9}}\sft^{\frac{1}{12}}+|b|^{-\frac 13-2}|\hat \mX| \right).
$$
We write $\mY=k_6\sft^{1/4}y^*(t,\hat x)-\Delta \mY$, and from \fref{autosimeulertop} and \fref{th:y*}:
$$
k_6\sft^{\frac 14}K \leq \Delta \mY \leq \frac{k_6\kappa}{1+|\hat \mX|^{\frac 16}}.
$$
Solving $\mY(a,b)=\mY$ then amount to solve:
$$
-\frac{1}{p^{*\frac 23}b^{\frac 13}}+O\left(|b|^{-\frac 13+\frac{1}{9}}\sft^{\frac{1}{12}}+|b|^{-\frac 13-2}|\hat \mX| \right)=-\Delta \mY.
$$
For $b=\delta \sft^{-3/4}$ one has $b^{-1/3}\gtrsim \sft^{1/4}\delta^{-1/3}\gg \sft^{\frac 14}K_{to}$ for $\delta$ small enough, and for $b=\bar b$ one has $\bar b^{-1/3}\gtrsim M^{-1/3}(1+|\mX|^{1/6})^{-1}\gg \kappa_{to} (1+|\mX|^{1/6})^{-1}$ for $M$ large enough. Hence, there exists indeed a solution $(\hat a,\hat b)$ to the above equation by the intermediate value Theorem, and bootstrapping information from the above equation and using that $|\hat \mX|\approx |\hat b|^2$ one obtains:
\be \lab{bd:b:topreconnect}
\hat b=\frac{1}{p^{*2}k_6^3\left(y^*(t,\hat x)-\hat y \right)^3\sft^{\frac 34}}\left(1+O\left(\frac{1}{(y^*(\hat x)-\hat y)^{\frac 13}}\right)\right).
\ee
Bootstrapping information from the equation \fref{bd:Z1:x} gives finally:
\bee
\hat a & = & -p^{*-\frac 23} \hat b^{\frac 23}\left(1+O\left(\frac{|\hat \mX|}{\hat b^2}+\sft^{\frac{1}{12}}|\hat b|^{\frac 19}\right)\right) \\
&=& -\frac{1}{k^2_6p^{*2}(y^*(t,\hat x)-\hat y)^2\sft^{\frac 12}}\left(1+O\left(\frac{|\hat y-y^*(t,\hat x)|}{y^*(t,\hat x)}+\frac{1}{(\hat y-y^*(t,\hat x))^{\frac 14}}\right)\right).
\eee
And hence from \fref{id:taylorugen} and \fref{id:defparametersgeneric}:
\begin{align*}
u(t,\hat x,\hat y)-u^*(t)&= \frac{\mu\nu^2}{p^{*2}(y^*(t,\hat x)-\hat y)^2}\left(1+O\left(\frac{|\hat y-y^*|}{y^*}+\frac{1}{(\hat y-y^*)^{\frac 14}}\right)\right)
\end{align*}
The estimates for the derivative proved in Lemma \ref{lem:computationmathcaly} adapt, namely at the point $(\hat a,\hat b)$ such that $(x,y)(\hat a,\hat b)=(\hat x,\hat y)$, using \fref{bd:b:topreconnect}:
\begin{align*}
&\pa_a \mX = 1+3p^{*2}\hat a^2+O(|\hat b|^{\frac 53}\sft^{\frac 14})=\pa_a \mX^\Theta+O\left( \frac{(y^*-\hat y)}{(y^*-\hat y)^4\sft} \right) ,\\
&\pa_b \mX=2b+O(\sft^{\frac 14}|\hat b|^{\frac 43})=\pa_b \mX^\Theta +O\left( \frac{(y^*-\hat y)^{-1}}{(y^*-\hat y)^3\sft^{\frac 34}} \right),\\
& \pa_a \mY = \pa_a \mY^{\Theta} +O(\sft^{\frac{1}{12}}|\hat b|^{-1+\frac{1}{9}})= \pa_a\mY^{\Theta} +O\left((y^*-\hat y)^{3-\frac 13}\sft^{\frac 34} \right),\\
&\pa_b \mY=\pa_b\mY^\Theta +O\left(|1+|\hat \mX||^{-\frac 76}|\hat b|^{1+\frac{1}{9}}\sft^{\frac{1}{12}} \right)=\pa_b\mY^\Theta +O\left((y^*-\hat y)^{4-\frac 13}\sft \right),
\end{align*}
so that from \fref{id:taylorpaugen}, \fref{id:derivee}, \fref{id:pamXmYTheta}, \fref{bd:calculTheta:bottom:pamXmY}, (iii) in Proposition \ref{pr:selfsimfunda} and the above estimates:
\bee
\pa_{\mX} u &=& \pa_a u \pa_b\mY-\pa_b u \pa_a\mY\\
&=&-\sft^{\frac 12}(2\pa_b\mY^{*}(\mX^\T)+\pa_b(2\mY^*(\mX^\T)-\pa_b\mY^\T))+O\left(\sft^{\frac 32}|\sft^{\frac 32}+|\hat x-x^*||^{-\frac 76}|y^*-\hat y|^{-3-\frac 13}\right)\\
&=& -\sft^{\frac 12}4 b\pa_\mX \mY^*+O(\sft^{-\frac 14}|y-y^*|^4)+O\left(\sft^{\frac 32}|\sft^{\frac 32}+|\hat x-x^*||^{-\frac 76}|y^*-\hat y|^{-3-\frac 13}\right),
\eee
\bee
\pa_{\mY} u &=&- \pa_a u \pa_b\mX+\pa_b u \pa_a\mX=-\frac{1}{\sft^{\frac 14}p^{*2}k_6^3(y^*-y)^3}+O\left(\frac{(y^*-\hat y)}{(y^*-\hat y)^3\sft^{\frac 14}} \right),
\eee
which using \fref{id:defparametersgeneric} completes the proof of \fref{bd:tildevtop}.

\end{proof}

\begin{appendix}

\section{Parametrisation and volume preservation} \label{ap:y}

\begin{lemma} \label{lem:paramincompressibility}
Let $\Omega \subset \mathbb R^2$ be open, and $f=(f_1,f_2)\in C^1(\Omega ,\mathbb R^2)$ be such that $\nabla f_1\neq 0$ on $\Omega$, and that for each $q$ in the range of $f_1$, the level set $\Gamma[q]=\{z\in \Omega, \ f_1(z)=q\}$ is diffeomorphic to $\mathbb R$.
\begin{itemize}
\item[(i)] Let $\gamma:\mathcal O\rightarrow \Omega$ be a diffeomorphism between an open set $\mathcal O\subset \mathbb R^2$ and $\Omega$, such that for each $q \in \mathbb R$, $\gamma(q,\cdot)$ is an arclength parametrisation of $\Gamma[q]$. Then $f$ preserves volume if and only if for all $q$ in the range of $f_1$, and any two $z=\gamma (q,s)$ and $z'=\gamma (q,s')$ there holds:
\be \lab{id:intevolume1}
|f_2(z)-f_2(z')|=\left| \int_{s}^{s'} \frac{d\tilde s}{|\nabla f_1(\gamma (q,\tilde s))|}\right|.
\ee
\item[(ii)] If $\pa_{z_1}f_1\neq 0$ on $\Omega$, let $\tilde z_1:\mathcal O\rightarrow \Omega $ be the only diffeomorphism between an open set $\mathcal O\subset \mathbb R^2$ and $\Omega$, such that for each $q\in \mathbb R$, $\tilde z_1(q,\cdot)$ is the parametrisation of $\Gamma[q]$ by the second component, i.e. $\Gamma[q]=\{(\tilde z_1(q,\tilde z_2),\tilde z_2), \ (q,\tilde z_2)\in \mathcal O \}$, then $f$ preserves volume if and only if for all $q$ in the range of $f_1$, and any two $z=(\tilde z_1(q,z_2),z_2)$ and $z'=(\tilde z_1(q,z_2'),z_2')$:
\be \lab{id:intevolume2}
|f_2(z)-f_2(z')| =\left| \int_{z_2}^{z_2'} \frac{d\tilde{z_2}}{\pa_{z_1} f_1(\tilde z_1(q,\tilde z_2),\tilde z_2)}\right|.
\ee
\end{itemize}

\end{lemma}

\begin{proof}
\emph{Proof of (i)}. We denote by $\gamma^{-1}=( q,s):\Omega \rightarrow \mathcal O$ the inverse of $ \gamma$, and note that $q=f_1$. At a point $z\in \Omega$, considering the orthonormal vectors $v_1=\nabla f_1(z) |\nabla f_1(z)|^{-1}$ and $v_2= \nabla^{\perp}f_1(z) |\nabla f_1(z)|^{-1}$ where $(e_1,e_2)^{\perp}=(-e_2,e_1)$ we get the identities $\nabla q.v_1=|\nabla f_1(z)|$, $\nabla q.v_2=0$ and $\nabla s. v_2=1$. This implies $|\text{Det}\left(J \gamma^{-1} \right)|=|\nabla  f_1(z)|.$ Consider now the mapping $\phi:(q,s)\mapsto (q,f_2(\gamma(q,s))$. Since $\frac{\pa q}{\pa s}_{|_q}=0$ and $\frac{\pa q}{\pa q}_{|s}=1$ we have $ |\text{Det}\left(J\phi \right)|=|\frac{\pa f_2}{\pa  s}_{|q}|. $
Since $f=\phi \circ \gamma^{-1}$, we get from the two previous computations that $f$ preserves volume, that is, $|\text{Det}Jf|=1$, if and only if:
$$
|\frac{\pa f_2}{\pa  s}_{|q}|=\frac{1}{|\nabla  f_1(z)|}.
$$
This is equivalent to (i) upon integrating, and noticing that these quantities cannot change sign.

\emph{Proof of (ii)}. From the assumptions on $f$, there always exists a parametrisation by arclength of the level sets $\gamma:\mathcal O' \rightarrow \Omega$ satisfying the properties of (i). We change variables $\tilde z_1^{-1}\circ \gamma:(q,s)\mapsto (q,\tilde z_2)$. Differentiating the identity $f_1(\tilde z_1(q,\tilde z_2),\tilde z_2)=q$ we find $\frac{\pa \tilde z_1}{\pa_{\tilde z_2}}_{|_{q}}=-\pa_{z_2}f_1/\pa_{z_1}f_1$. Hence:
$$
|\frac{\pa s}{\pa \tilde z_2}_{|_q}|=\sqrt{1+\left(\frac{\pa_{z_2}f_1}{\pa_{z_1}f_1} \right)^2}=\frac{|\nabla f_1|}{|\pa_{z_1}f_1|}.
$$
Changing variables $s \mapsto \tilde z_2$ in \fref{id:intevolume1}, then yields that $f$ preserves volume if and only if \fref{id:intevolume2} holds true.
\end{proof}

\section{Computing the normal component of the characteristics} \label{ap:y2}

\begin{proof}[proof of Lemma \ref{lem:computationmathcaly}]

At several moments in the proof, we will use the following. The function $\Psi_1$ enjoys:
\be \lab{bd:Psi11}
c<\left|\frac{X\pa_X\Psi_1(X)}{\Psi_1(X)} \right|\leq \frac 1c, \ \ \ \ |\Psi_1(X)|\lesssim |X|^{\frac 13}, \ \ \Psi_1(X+X')=\Psi_1(X)+O(|X'|^{\frac 13})
\ee
uniformly for $X,X'\in \mathbb R$. From the first bound above, given a small quantity $O(|z|)$ that has size $|z|$, as an application of the implicit function theorem it is true that:
\be \lab{bd:Psi12}
\Psi_1\left(X(1+O(|z|))\right)=\Psi_1\left(X\right) (1+O(|z|)).
\ee
where the "new" $O(|z|)$ in the right hand side enjoys the same differentiability properties and similar bounds, as the one in the left hand side. To rearrange the $O()$'s in what follows, we shall simplify $O(|z_1|)+O(|z_2|)=O(|z_2|)$ in zones where $|z_1|\lesssim |z_2|$, and $O(|z|^\alpha)=O(|z|^\beta)$ if $ 0< \beta \leq \alpha $ as all quantities in the $O()$'s will be small.

 To compute the vertical component, we will use the following result. Assume that $(a,b)$ and $(a',b')$ belong to the same level set curve $\Gamma[x]=\{x(\ta,\tb)=x \}$. Assume that, when ordering points on $\Gamma$ with their distance to the boundary, $(a',b')$ is after $(a,b)$, and that either $b\leq b'$ or $a\leq a'$. Assume moreover that $\Gamma_{(a,b)}^{(a',b')}$, the part of $\Gamma$ between $(a,b)$ and $(a',b')$, can be either parametrised with the variable $\tb$ as $\ta=\ta(\mX,\tb)$ for $b\leq \tb \leq b'$ or with the variable $\ta$ as $\tb=\tb (\mX,\ta)$ for $a\leq \ta\leq a'$. Then by applying Lemma \ref{lem:paramincompressibility} and \fref{def:mathcalXmathcalYrenorm} one obtains the following identities:
\be \lab{id:intchangevariable}
\int_{\Gamma_{(a,b)}^{(a',b')}} \frac{ds}{|\nabla x|}=\frac{1}{k_6\sft^{\frac 14}} \int_{b}^{b'} \frac{d\tb}{|\pa_a \mX|} \ \ \text{or} \ \ \int_{\Gamma_{(a,b)}^{(a',b')}} \frac{ds}{|\nabla x|}=\frac{1}{k_6\sft^{\frac 14}} \int_{a}^{a'} \frac{d\tb}{|\pa_b \mX|}.
\ee

\noindent \textbf{Step 1} \emph{The normal component for left part of the sides, the core, and the bottom, of the self-similar zone}. We first derive rough estimates that will be used in the next steps. Fix $(X,Y)\in Z_0^c$ such that either $-\epsilon \leq x(a,b)-x^*\leq K\sft^{3/2}$, or, $K\sft^{3/2}\leq x-x^*\leq \epsilon$ and $b<0$ and $N_2 |\mX|\leq |b|^2$. Note that in the second case,  from \fref{bd:xvanshen}, one has necessarily $|a|\approx |b|^{2/3}\gg 1$. Let $\Gamma$ denote the part of the curve $\Gamma[x]$ which joins the boundary of the upper half plane and $(X,Y)$. We decompose it in two parts:
\be \lab{id:mYstep1}
\Gamma_1:= \Gamma \cap Z_0, \ \ \Gamma_2=\Gamma \cap Z_0^c, \ \ \mY = k_6\sft^{\frac 14}\left(\int_{\Gamma_1} \frac{ds}{|\nabla x|}+\int_{\Gamma_2} \frac{ds}{|\nabla x|} \right),
\ee
The integral in $Z_0$ is at distance one to $(X_0,Y_0)$, and everything then remains regular:
\be \lab{id:mY1step1}
\int_{\Gamma_1} \frac{ds}{|\nabla x|}=O(1), \ \ \pa_{\mX}\left(\int_{\Gamma_1} \frac{ds}{|\nabla x|}\right)=O(\sft^{\frac 32}).
\ee
In $Z_0^c$, from Lemma \ref{lem:lagrangian} the curve $\Gamma [x]$ lies in $Z_1$, so it can be parametrised with the variable $\tb$. Also, from \fref{id:intchangevariable}, \fref{bd:Z1:paamX} and as $|\tb|\sft^{3/4}\ll 1$ for $-\delta\sft^{-3/4}\leq \tb \leq b$:
\bea
\non  k_6\sft^{\frac 14}\int_{\Gamma_2} \frac{ds}{|\nabla x|}  &=&  \int_{-\frac{\delta}{\sft^{-\frac 34}}}^b \frac{d\tb}{\pa_a\mX} \ =\ \int_{-\frac{\delta}{\sft^{-\frac 34}}}^b \frac{1+O\left(\sft^{\frac{1}{12}}+|\tb|^{\frac 13}\sft^{\frac{1}{4}}+|\mX|^{\frac 16}\sft^{\frac{1}{4}}\right) }{1+3\Psi_1^2( p^* (\mX-\tb^2) )}d\tb\\
\non &=& \int_{-\infty}^b  \frac{1+O\left(\sft^{\frac{1}{12}}+|\tb|^{\frac 14}\sft^{\frac{3}{16}}+|\mX|^{\frac 16}\sft^{\frac{1}{4}}\right) }{1+3\Psi_1^2( p^* (\mX-\tb^2) )}d\tb+O\left(\int_{-\infty}^{-\frac{\delta}{\sft^{-\frac 34}}} \frac{d\tb}{|\tb|^{\frac 43}} \right)\\
\lab{id:mY2step1} &=& \int_{-\infty}^b  \frac{1+O\left(\sft^{\frac{1}{12}}+|\tb|^{\frac 14}\sft^{\frac{3}{16}}+|\mX|^{\frac 16}\sft^{\frac{1}{4}} \right)}{1+3\Psi_1^2( p^* (\mX-\tb^2) )}d\tb+O\left(\sft^{\frac 14} \right)
\eea
where we used the fact that $|\mX|\leq \epsilon\sft^{-3/2}$, and that for $\tb \leq -\delta \sft^{-3/4}$, $\tb^2 \gg \mX$ if $\epsilon$ is small enough, implying $\Psi_1^2(p^*(\mX-\tb^2))\approx |\tb|^{4/3}$. Hence, injecting \fref{id:mY1step1} and \fref{id:mY2step1} in \fref{id:mYstep1}:
\be \lab{id:mYleftcentral2}
\mY(a,b)= \int_{-\infty}^b  \frac{1+O\left(\sft^{\frac{1}{12}}+|\tb|^{\frac 14}\sft^{\frac{3}{16}}+|\mX|^{\frac 16}\sft^{\frac{1}{4}} \right)}{1+3\Psi_1^2( p^* (\mX-\tb^2) )}d\tb+O(\sft^{\frac 14}).
\ee
We now study the derivative of $\mY$. As $\mX(\ta,\tb)=\mX$ is inverted through $\ta=\ta(\mX,\tb)$ one has:
$$
\frac{\pa}{\pa \mX}_{|\tb}\left(\ta(\mX,\tb) \right)=\frac{1}{\pa_a \mX (\ta(\mX,\tb),\tb )}, \ \ \frac{\pa}{\pa \mX}_{|\tb} \left(\frac{1}{\pa_a \mX (\ta(\mX,\tb),\tb)} \right)=-\frac{\pa_{aa}\mX(\ta (\mX,\tb),\tb) }{(\pa_a \mX(\ta(\mX,\tb),\tb))^3}
$$
One deduces from \fref{id:intchangevariable}, \fref{bd:Z1:paamX}, \fref{bd:Z1:paaamX} that:
\bee
&& \frac{\pa}{\pa \mX}_{|b} \left(k_6\sft^{\frac 14}\int_{\Gamma_2} \frac{ds}{|\nabla x|}\right) = \frac{\pa}{\pa \mX}_{|b} \left( \int_{-\frac{\delta}{\sft^{3/4}}}^b \frac{d\tb}{\pa_a\mX(\ta(\mX,\tb),\tb)}\right) =  -\int_{-\frac{\delta}{\sft^{3/4}}}^b \frac{\pa_{aa}\mX (\ta,\tb)d\tb}{(\pa_a \mX (\ta,\tb))^3}\\
&=&  6p^*\int_{-\frac{\delta}{\sft^{3/4}}}^b \frac{\Psi_1\left(p^*\left(\mX-\tb^2 \right) \right)}{\left(1+3\Psi_1^2\left(p^*(\mX-\tb^2) \right) \right)^3} \left(1+O\left(\sft^{\frac{1}{12}}+|\mX|^{\frac 16}\sft^{\frac{1}{4}}+|\tb|^{\frac 13}\sft^{\frac{1}{4}} \right) \right)d\tb \\
&&+\int_{-\frac{\delta}{\sft^{3/4}}}^b \frac{O\left(\sft^{\frac{1}{12}} \right)}{\left(1+3\Psi_1^2\left( p^*(\mX-\tb^2) \right) \right)^3} d\tb.
\eee
Note that in the integrals above, for $\tb\leq -\delta\sft^{-3/4}$ there holds $\Psi_1(p^*(\mX-\tb^2))\gtrsim \tb^{2/3}$. Hence, integrating from infinity instead of $-\delta\sft^{-3/4}$ produces an error which is $O(\sft^{7/4})$ and:
\begin{align*}
 \frac{\pa}{\pa \mX}_{|b} \left(k_6\sft^{\frac 14}\int_{\Gamma_2} \frac{ds}{|\nabla x|}\right) &=  6p^*\int_{-\infty}^b \frac{\Psi_1\left(p^*\left(\mX-\tb^2 \right) \right)}{\left(1+3\Psi_1^2\left(p^*(\mX-\tb^2) \right) \right)^3} \left(1+O\left(\sft^{\frac{1}{12}}+|\mX|^{\frac 16}\sft^{\frac{1}{4}}+|\tb|^{\frac 14}\sft^{\frac{3}{16}} \right) \right)d\tb \\
&\qquad +\int_{-\infty}^b \frac{O\left(\sft^{\frac{1}{12}} \right)}{\left(1+3\Psi_1^2\left(p^*(\mX-\tb^2) \right) \right)^3} d\tb+O\left(\sft^{\frac{7}{4}} \right)
\end{align*}
Therefore, injecting \fref{bd:Z1:paamX}, \fref{id:mY1step1} and the above identity in \fref{id:mYstep1}:
\bea
\non && \pa_{a}\mathcal Y(a,b) \ = \ \pa_{a}\mX \left( k_6\sft^{\frac 14}\frac{\pa}{\pa \mX} \int_{\Gamma_1} \frac{ds}{|\nabla x|} +\frac{\pa}{\pa \mX}_{|b} \left(k_6\sft^{\frac 14}\int_{\Gamma_2} \frac{ds}{|\nabla x|}\right) \right) \\
\non &=& \pa_{a}\mX \Biggl( 6 p^*\int_{-\infty}^b \frac{\Psi_1\left(p^*\left(\mX-\tb^2 \right) \right)}{\left(1+3\Psi_1^2\left(p^*(\mX-\tb^2) \right) \right)^3} \left(1+O\left(\sft^{\frac{1}{12}}+|\mX|^{\frac 16}\sft^{\frac{1}{4}}+|\tb|^{\frac 14}\sft^{\frac{3}{16}} \right) \right)d\tb \\
\non &&\qquad \qquad +\int_{-\infty}^b \frac{O\left(\sft^{\frac{1}{12}} \right)}{\left(1+3\Psi_1^2\left(p^*(\mX-\tb^2) \right) \right)^3} d\tb+O\left(\sft^{\frac{7}{4}} \right)\Biggr) \\
\lab{id:paamYleftcentral2} &=& 6p^* \left(1+3\Psi_1^2\left(p^*\left(\mX-b^2 \right) \right) \right)\Biggl( \int_{-\infty}^b \frac{O\left(\sft^{\frac{1}{12}} \right)}{\left(1+3\Psi_1^2\left(p^*(\mX-\tb^2) \right) \right)^3} d\tb+O(\sft^{\frac 74})  \\
\non && + \int_{-\infty}^{b}  \frac{\Psi_1\left(p^*\left(\mX-\tb^2 \right) \right)}{\left(1+3\Psi_1^2\left(p^*(\mX-\tb^2) \right) \right)^3} \left(1+O\left(\sft^{\frac{1}{12}}+|\mX|^{\frac 16}\sft^{\frac{1}{4}}+|\tb|^{\frac 14}\sft^{\frac{3}{16}}+\sft^{\frac{1}{4}}|b|^{\frac 13} \right) \right)d\tb  \Biggr)
\eea
The estimates \fref{id:mYleftcentral} and \fref{id:paamYleftcentral}, and \fref{def:Phi}, prove \fref{id:mYleftcentral} and \fref{id:paamYleftcentral} for the first case.\\

\noindent \textbf{Step 2} \emph{The normal component for the sides of the self-similar zone}. We now prove the estimates \fref{bd:sides:mY} and \fref{bd:sides:pabmY} in the second case described by \fref{autosimaabright2}. The parametrisation of the curve $\Gamma[x]$ has to be done more carefully. In the case of the left side, i.e. $x<0$, we are in the case considered in Step 1, which has already been covered. In the case of the right side, i.e. $x>0$, we are in case (ii) of Lemma \ref{lem:lagrangian}, and then inside $Z_0^c$, the curve $\Gamma$ can be decomposed in five curves, $\Gamma_i$ for $i=1,...,5$ that enjoy the properties described there. We use different variables to parametrise the curves $\Gamma_i$, applying Lemma \ref{lem:lagrangian}. On $\Gamma_1$ we use the variable $\tb$, on $\Gamma_2$ $\ta$, on $\Gamma_3$ $\tb$, on $\Gamma_4$ $\ta$ and on $\Gamma_5$ $\tb$. Without loss of generality for the argument, we assume that $(\bar a,\bar b)$ is located on $\Gamma_2$. Indeed, treating the case of three or more different parametrisation can be done the very same way. We consider a point $(a,b)$ close to $(\bar a,\bar b)$, which still belongs to $\Gamma_2$ (up to changing slightly the constants in the definition of $Z_1$ and $Z_2$). We denote by $\Gamma_0$ the part of the curve outside $Z_0^c$. Hence from \fref{id:intchangevariable}:
\be \lab{id:mYstep4}
\mY =k_6\sft^{\frac 14} \int_{\Gamma_0} \frac{ds}{|\nabla x|}+\int_{-\frac{\delta}{\sft^{3/4}}}^{b_1} \frac{d\tb}{\pa_a \mX}+\int_{a_1}^{a} \frac{d\ta}{\pa_b \mX},
\ee
where we recall that $(a_1,b_1)$ is the endpoint of $\Gamma_1$ and the starting point of $\Gamma_2$, defined in Lemma \ref{lem:lagrangian}. The integral over $\Gamma_0$ is at distance one to $(X_0,Y_0)$ and hence in a zone where everything remains regular:
\be \lab{id:mYGamma0step4}
\int_{\Gamma_0} \frac{ds}{|\nabla x|}=O(1), \ \ \pa_{\mX}\left(\int_{\Gamma_0} \frac{ds}{|\nabla x|}\right)=O(\sft^{\frac 32}).
\ee
We now consider the second integral, corresponding to the part $\Gamma_1$ of the curve joining the points $(a_{in},-\delta \sft^{-3/4})$ and $(a_1,b_1)$. Since this part is in $Z_1$, one obtains from \fref{bd:Z1:paamX}, injecting \fref{bd:abmXautosimsides}:
\be \lab{id:mYGamma1step4}
\int_{-\frac{\delta}{\sft^{3/4}}}^{b_1} \frac{d\tb}{\pa_a \mX}= \int_{-\frac{\delta}{\sft^{3/4}}}^{b_1} \frac{d\tb}{1+3\Psi_1^2(p^*(\mX-\tb^2) )}\left(1+O\left(\sft^{\frac{1}{12}}|\mX|^{\frac{1}{18}}+\sft^{\frac{1}{12}}|\tb|^{\frac{1}{9}} \right)\right)
\ee
We turn to the third integral, corresponding to the part $\Gamma_2$ of the curve joining the points $(a_{1},b_1)$ and $(a,b)$. There, since this part is in $Z_2$, one obtains from \fref{bd:Z2:pax}:
$$
\int_{a_1}^{a} \frac{d\ta}{\pa_b \mX} = \int_{a_1}^a \frac{d\ta}{2\sqrt{\mX-\ta-p^{*2}\ta^3}} \left(1+O\left(|\mX|^{\frac 16}\sft^{\frac 14}\right)\right).
$$
In $Z_2$, one has from \fref{def:Z2} that $|\ta|\ll |\tb|^{2/3}$ and $|\tb|\gg 1$, so that from \fref{bd:xvanshen} one has 
\be \lab{bd:step4mXa}
\mX \approx |\tb|^2\gg |\ta|^3 \ \ \text{and} \ \ \sqrt{\mX-\ta-p^{*2}\ta^3}\approx \mX^{1/2}
\ee
uniformly for $\ta$ in $Z_2$, as well as $|a_1|,|a|\ll \mX^{1/3}$. Using these bounds one infers from the above identity that:
$$
\int_{a_1}^{a} \frac{d\ta}{\pa_b \mX} =\int_{a_1}^a \frac{d\ta}{2\sqrt{\mX-\ta-p^{*2}\ta^3}} +O\left(\frac{|a-a_1|}{|\mX|^{\frac 12}} |\mX|^{\frac 16}\sft^{\frac 14}\right) = \int_{a_1}^a \frac{d\ta}{2\sqrt{\mX-\ta-p^{*2}\ta^3}} +O\left( |\mX|^{-\frac 16+\frac 16}\sft^{\frac{1}{4}}\right).
$$
Also, from the identity \fref{bd:Z2:a}, using the above bound \fref{bd:step4mXa}:
\be \lab{bd:step4changeendpoints}
\left|a+\frac{1}{p^*}\Psi_1\left(p^*\left(\mX-b^2\right)\right)\right|+\left|a+\frac{1}{p^*}\Psi_1\left(p^*\left(\mX-b^2\right)\right)\right|\lesssim |\mX|^{\frac 13 +\frac{1}{18}}\sft^{\frac{1}{12}},
\ee
$$
\left| \int_{a}^{-\frac{1}{p^*}\Psi_1\left(p^*\left(\mX-b^2\right)\right)} \frac{d\ta}{\sqrt{\mX-\ta-p^{*2}\ta^3}} \right|\lesssim \frac{\left|a+\frac{1}{p^*}\Psi_1\left(p^*\left(\mX-b^2\right)\right)\right|}{|\mX|^{\frac 12}}\lesssim |\mX|^{-\frac 16+\frac{1}{18}}\sft^{\frac{1}{12}} ,
$$
$$
\left| \int_{a_1}^{-\frac{1}{p^*}\Psi_1\left(p^*\left(\mX-4b_1^2\right)\right)} \frac{d\ta}{\sqrt{\mX-\ta-p^{*2}\ta^3}} \right|\lesssim \frac{\left|a_1+\frac{1}{p^*}\Psi_1\left(p^*\left(\mX-4b_1^2\right)\right)\right|}{|\mX|^{\frac 12}}\lesssim |\mX|^{-\frac 16+\frac{1}{18}}\sft^{\frac{1}{12}} .
$$
Therefore:
$$
\int_{a_1}^{a} \frac{d\ta}{\pa_b \mX}  =\int_{-\frac{1}{p^*}\Psi_1\left(p^*\left(\mX-4b_1^2\right)\right)}^{-\frac{1}{p^*}\Psi_1\left(p^*\left(\mX-b^2\right)\right)} \frac{d\ta}{2\sqrt{\mX-\ta-p^{*2}\ta^3}} +O\left(|\mX|^{-\frac 16+\frac{1}{18}}\sft^{\frac{1}{12}} \right).
$$
We now change variables in the above integral, taking $\tb=-\sqrt{\mX-\ta-p^{*2}\ta^3}$. The left and right endpoints of the integral are precisely $b_1$ and $b$, and this produces
$$
\ta=-\frac{1}{p^*} \Psi_1\left(p^*(\mX-\tb^2)\right), \ \ d\tb=\frac12 \frac{1+3p^{*2}\ta^2}{\sqrt{\mX-\ta-p^{*2}\ta^3}}d\ta,
$$
$$
\int_{a_1}^{a} \frac{d\ta}{\pa_b \mX} = \int_{b_1}^{b} \frac{d\tb}{1+3\Psi_1^2\left(p^*\left(\mX-\tb^2\right)\right)}+O\left(|\mX|^{-\frac 16+\frac{1}{18}}\sft^{\frac{1}{12}} \right).
$$
We inject the identities \fref{id:mYGamma0step4}, \fref{id:mYGamma1step4} and the above identity in the expression \fref{id:mYstep4}, giving the following expression for $\mY$:
\be
\lab{id:mYsideselfsim} \mY = \int_{-\frac{\delta}{\sft^{\frac 34}}}^{b} \frac{d\tb \left(1+O\left(|\mX|^{\frac{1}{18}}\sft^{\frac{1}{12}}+|\tb|^{\frac{1}{9}}\sft^{\frac{1}{12}}\right)\right)}{1+3\Psi_1^2\left(p^*\left(\mX-\tb^2\right)\right)}+O\left(|\mX|^{-\frac 16+\frac{1}{18}}\sft^{\frac{1}{12}} \right) = \mY^\Theta (a,b) \left(1+O(\sft^{\frac{1}{12}}|\mX|^{\frac{1}{18}}) \right).
\ee
where we used \fref{bd:calculTheta:side:mY}. This shows the first desired bound in \fref{bd:sides:mY}. We now consider the derivatives of $\mY$. The point $(a_1,b_1)$ changes as $\mX$ changes, but the identity $\mX(a_1,b_1)=\mX(a,b)$ ensures that 
\be \lab{id:derivativea1b1}
\pa_\mX a_1\pa_a\mX(a_1,b_1)+\pa_\mX b_1 \pa_b \mX(a_1,b_1)=1.
\ee
Hence, differentiating the sum of the two leading order integrals in \fref{id:mYstep4} one obtains:
\begin{align}
\non & \sft^{\frac 14}k_6\pa_a \left(\int_{\Gamma_1} \frac{ds}{|\nabla x|}+\int_{\Gamma_2} \frac{ds}{|\nabla x|}\right)= \pa_a \left(\int_{-\frac{\delta}{\sft^{3/4}}}^{b_1(a,b)} \frac{d\tb}{\pa_a \mX(\ta(\mX,\tb),\tb)}+\int_{a_1(a,b)}^a -\frac{d\ta}{\pa_b \mX(\ta,\tb(\mX,\ta))} \right) \\
\non &= \pa_a \mX(a,b)\Biggl(\int_{-\frac{\delta}{\sft^{3/4}}}^{b_1} \pa_\mX \left(\frac{1}{\pa_a \mX (\ta(\mX,\tb),\tb )}\right)d\tb+\int_{a_1}^{a} \pa_\mX \left(-\frac{1}{\pa_b \mX(\ta,\tb(\mX,\ta) )}\right)d\ta \\
\non & \qquad \qquad \qquad \qquad +\frac{\pa_\mX b_1}{\pa_a\mX (a_1,b_1)}+\frac{\pa_\mX a_1}{\pa_b\mX(a_1,b_1)} \Biggr)-\frac{1}{\pa_b \mX(a,b)}\\
\non &= \pa_a \mX(a,b)\Biggl(\int_{-\frac{\delta}{\sft^{3/4}}}^{b_1} \frac{-\pa_{aa}\mX (\ta(\mX,\tb),\tb )d\tb }{\left(\pa_a \mX (\ta(\mX,\tb),\tb )\right)^3}+\int_{a_1}^{a} \frac{\pa_{bb}\mX (\ta,\tb(\mX,\ta) )d\ta}{\left(\pa_b \mX(\ta,\tb(\mX,\ta) )\right)^3} + \frac{1}{\pa_a \mX(a_1,b_1) \pa_b \mX(a_1,b_1)}\Biggr)\\
\lab{id:paamYstep4} &\qquad  -\frac{1}{\pa_b \mX(a,b)}.
\end{align}
The first integral is located in $Z_1$ with $|a|\gg1$, hence from \fref{bd:Z1:paamX} and \fref{bd:Z1:paaamX}:
\be \lab{id:paamYGamma1step4}
\int_{-\frac{\delta}{\sft^{3/4}}}^{b_1} \frac{-\pa_{aa}\mX\left(\ta(\mX,\tb),\tb \right)}{\left(\pa_a \mX\left(\ta(\mX,\tb),\tb \right)\right)^3}d\tb = 6p^* \int_{-\frac{\delta}{\sft^{3/4}}}^{b_1} \frac{\Psi_1\left(p^*\left(\mX-\tb^2\right)\right)\left(1+O\left(\sft^{\frac{1}{12}}|\mX|^{\frac{1}{18}}+\sft^{\frac{1}{12}}|\tb|^{\frac{1}{9}}\right)\right)}{\left(1+3\Psi_1^2\left(p^*\left(\mX-\tb^2\right) \right)\right)^3}d\tb.
\ee
The second integral is located in $Z_2$, hence from \fref{bd:Z2:pax} and \fref{bd:Z2:pabbx}:
$$
\int_{a_1}^{a} \frac{\pa_{bb}\mX \left(\ta,\tb(\mX,\ta) \right)}{\left(\pa_b \mX\left(\ta,\tb(\mX,\ta) \right)\right)^3}d\ta =  -\frac{1}{4} \int_{a_1}^{a} \frac{1+O\left(|\mX|^{\frac{1}{6}}\sft^{\frac 14}\right)}{\left(\mX-\ta-p^{*2}\ta^3 \right)^{\frac 32}}d\ta. 
$$
From \fref{bd:step4mXa}, \fref{bd:step4changeendpoints}:
\bee
\int_{a_1}^{a} \frac{\pa_{bb}\mX \left(\ta,\tb(\mX,\ta) \right)}{\left(\pa_b \mX\left(\ta,\tb(\mX,\ta) \right)\right)^3}d\ta  &=&  -\frac{1}{4} \int_{-\frac{1}{p^*} \Psi_1\left(p^*\left(\mX-4b_1^2\right) \right)}^{-\frac{1}{p^*} \Psi_1\left(p^*\left(\mX-b^2\right) \right)} \frac{d\ta}{\left(\mX-\ta-p^{*2}\ta^3 \right)^{\frac 32}} +O\left(|a-a_1|\frac{|\mX|^{\frac 16}\sft^{\frac 14}}{|\mX|^{\frac 32}} \right)\\
&&+O\left(\left|\frac{\left|a+\frac{1}{p^*} \Psi_1 \left(p^*\left(\mX-b^2 \right) \right)\right|+\left|a_1+\frac{1}{p^*} \Psi_1 \left(p^*\left(\mX-4b_1^2 \right) \right)\right|}{|\mX|^{\frac 32}} \right| \right) \\
& = & -\frac{1}{4} \int_{-\frac{1}{p^*} \Psi_1\left(p^*\left(\mX-4b_1^2\right) \right)}^{-\frac 1p \Psi_1\left(p^*\left(\mX-b^2\right) \right)} \frac{d\ta}{\left(\mX-\ta-p^{*2}\ta^3 \right)^{\frac 32}}+O\left(\left|\mX \right|^{-\frac 76+\frac{1}{18}}\sft^{\frac{1}{12}} \right) \\
\eee
We now change variables, taking $\tb=-\sqrt{\mX-\ta-p^{*2}\ta^3}$. Note that this change of variables ensures $\mX^\Theta (\ta,\tb)=\ta+p^{*2}\ta^3+\tb^2=Cte=\mX$, and
$$
-\frac 14 \frac{1}{\left(\mX-\ta-p^{*2}\ta^3 \right)^{\frac 32}}=\frac{\pa_{bb}\mX^\Theta}{(\pa_b \mX^\Theta)^3}.
$$
There holds in this case a general formula when integrating on the curve $\{\mX(a,b)=Cte\}$, obtained by performing a change of variables and an integration by parts (note the signs $\pa_a \mX^\Theta >0$ and $\pa_b \mX^\Theta<0$ in the present case):
\begin{align*}
& \int_{a_1}^{a_2} \frac{\pa_{bb}\mX^\Theta}{(\pa_{b}\mX^\Theta)^3}da = - \int_{b_1}^{b_2}  \frac{\pa_{bb}\mX^\Theta db}{(\pa_{b}\mX^\Theta)^2\pa_a \mX^\Theta} = - \int_{b_1}^{b_2} \left(\frac{d}{db} \pa_b \mX^\Theta+\frac{\pa_b  \mX^\Theta}{\pa_a  \mX^\Theta}\pa_{ba}  \mX^\Theta \right) \frac{db}{(\pa_{b} \mX^\Theta)^2\pa_a  \mX^\Theta} \\
& =  -\int_{b_1}^{b_2} \frac{d}{db} (\pa_b  \mX^\Theta )\frac{1}{(\pa_{b} \mX^\Theta)^2\pa_a  \mX^\Theta} db- \int_{b_1}^{b_2} \frac{\pa_{ab} \mX^\Theta}{\pa_b \mX^\Theta (\pa_a \mX^\Theta)^2} \\
& =  -\int_{b_1}^{b_2} \frac{d}{db} \left(\pa_b  \mX^\Theta \frac{1}{(\pa_{b} \mX^\Theta)^2\pa_a  \mX^\Theta}\right) da +\int_{b_1}^{b_2} \pa_b  \mX^\Theta \frac{d}{db}\left(\frac{1}{(\pa_{b} \mX^\Theta)^2\pa_a  \mX^\Theta}\right)   db-\int_{b_1}^{b_2} \frac{\pa_{ab} \mX^\Theta db}{\pa_b \mX^\Theta (\pa_a \mX^\Theta)^2} \\
& = - \frac{1}{\pa_a  \mX^\Theta(a_2,b_2)\pa_b \mX^\Theta(a_2,b_2)}+\frac{1}{\pa_a  \mX^\Theta(a_1,b_1)\pa_b \mX^\Theta(a_1,b_1)} \\
&\quad -2\int_{b_1}^{b_2} \left(\pa_{bb}  \mX^\Theta-\frac{\pa_b  \mX^\Theta}{\pa_a \mX^\Theta}\pa_{ab} \mX^\Theta\right)\frac{1}{(\pa_{b} \mX^\Theta)^2\pa_a  \mX^\Theta}   db\\
&\quad -\int_{b_1}^{b_2}  \left(\pa_{ab}  \mX^\Theta-\frac{\pa_b  \mX^\Theta}{\pa_{a} \mX^\Theta}\pa_{aa} \mX^\Theta \right)\frac{1}{\pa_b  \mX^\Theta (\pa_a  \mX^\Theta)^2}db-\int_{b_1}^{b_2} \frac{\pa_{ab} \mX^\Theta}{\pa_b \mX^\Theta (\pa_a \mX^\Theta)^2} \\
& =  -\frac{1}{\pa_a  \mX^\Theta(a_2,b_2)\pa_b \mX^\Theta(a_2,b_2)}+\frac{1}{\pa_a  \mX^\Theta(a_1,b_1)\pa_b \mX^\Theta(a_1,b_1)} +2\int_{a_1}^{a_2}  \frac{\pa_{bb}  \mX^\Theta}{(\pa_{b} \mX^\Theta)^3}   da+\int_{b_1}^{b_2}  \frac{\pa_{aa} \mX^\Theta}{(\pa_a  \mX^\Theta)^3}db,\\
\end{align*}
from what one deduces the change of parametrisation identity:
$$
 \int_{a_1}^{a_2} \frac{\pa_{bb} \mX^\Theta}{(\pa_{b} \mX^\Theta)^3}da= \frac{1}{\pa_a  \mX^\Theta(a_2,b_2)\pa_b \mX^\Theta(a_2,b_2)}-\frac{1}{\pa_a  \mX^\Theta(a_1,b_1)\pa_b \mX^\Theta(a_1,b_1)} -\int_{b_1}^{b_2}  \frac{\pa_{aa} \mX^\Theta}{(\pa_a  \mX^\Theta)^3}db.
$$
Applied to our case this produces, noticing that the endpoint are $(a_1,b_1)$ and $(a,b)$:
\bee
&& -\frac{1}{4} \int_{-\frac{1}{p^*} \Psi_1\left(p^*\left(\mX-4b_1^2\right) \right)}^{-\frac{1}{p^*} \Psi_1\left(p^*\left(\mX-b^2\right) \right)} \frac{1}{\left(\mX-\ta-p^{*2}\ta^3 \right)^{\frac 32}}d\ta\\
 &=& \frac{1}{\pa_a\mX^\Theta(a,b)\pa_b \mX^\Theta (a,b)}-\frac{1}{\pa_a \mX^\Theta (a_1,b_1)\pa_b \mX^\Theta(a_1,b_1)}-\int_{b_1}^b \frac{\pa_{aa}\mX^\Theta}{(\pa_a \mX^\Theta)^3}d\tb\\
  &=& \frac{1}{(1+3p^{*2}a^2)2b}-\frac{1}{(1+3p^{*2}a_1^2)2b_1}+6p^* \int_{b_1}^b \frac{\Psi_1\left(p^*\left(\mX-\tb^2\right)\right)}{\left(1+3\Psi_1^2\left(p^*\left(\mX-\tb^2 \right) \right)\right)^3}d\tb
\eee
Note that since $(a_1,b_1)$ belong to both $Z_1$ and $Z_2$, from \fref{bd:step4mXa}, \fref{bd:Z2:pax} and \fref{bd:Z1:paamX}:
$$
\frac{1}{\pa_a \mX(a_1,b_1)\pa_b\mX(a_1,b_1)} = \frac{1+O\left(|\mX|^{\frac 16}\sft^{\frac 14}\right)}{\left(1+3p^{*2}a_1^2\right)2b_1}=\frac{1}{\left(1+3p^{*2}a_1^2\right)2b_1}+O\left(|\mX|^{-\frac{7}{6}+\frac 16}\sft^{\frac 14} \right)
$$
so that:
\bee
 \int_{a_1}^{a} \frac{\pa_{bb}\mX \left(\ta,\tb(\mX,\ta) \right)}{\left(\pa_b \mX\left(\ta,\tb(\mX,\ta) \right)\right)^3}d\ta &=& \frac{1}{(1+3p^{*2}a^2)2b}-\frac{1}{\pa_a \mX(a_1,b_1)\pa_b\mX(a_1,b_1)} \\
 &&+6p^* \int_{b_1}^b \frac{\Psi_1\left(p^*\left(\mX-\tb^2\right)\right)}{\left(1+3\Psi_1^2\left(p^*\left(\mX-\tb^2 \right) \right)\right)^3}d\tb+O\left(|\mX|^{-\frac 76+\frac{1}{18}}\sft^{\frac{1}{12}} \right).
\eee
From the above identity and \fref{id:paamYGamma1step4} one concludes that:
\bee
&&\int_{-\frac{\delta}{\sft^{3/4}}}^{b_1} \frac{-\pa_{aa}\mX\left(\ta(\mX,\tb),\tb \right)}{\left(\pa_a \mX\left(\ta(\mX,\tb),\tb \right)\right)^3}d\tb+\int_{a_1}^{a} \frac{\pa_{bb}\mX \left(\ta,\tb(\mX,\ta) \right)}{\left(\pa_b \mX\left(\ta,\tb(\mX,\ta) \right)\right)^3}d\ta +\frac{1}{\pa_a \mX(a_1,b_1) \pa_b \mX(a_1,b_1)}\\
&=& 6p^*\int_{-\frac{\delta}{\sft^{3/4}}}^{b} \frac{\Psi_1\left(p^*\left(\mX-\tb^2\right)\right)\left(1+O\left(\sft^{\frac{1}{12}}|\mX|^{\frac{1}{18}}+\sft^{\frac{1}{12}}|\tb|^{\frac{1}{9}}\right)\right)}{\left(1+3\Psi_1^2\left(p^*\left(\mX-\tb^2 \right) \right)\right)^3}d\tb\\
&& + \frac{1}{(1+3p^{*2}a^2)2b}+O\left( |\mX|^{-\frac 76+\frac{1}{18}}\sft^{\frac{1}{12}}\right) 
\eee
From \fref{bd:calculTheta:side:pamY}, and using the fact that for $\tb \leq-\delta \sft^{3/4}$ one has $|\Psi_1\left(p^*\left(\mX-\tb^2\right)\right)|\approx |\tb|^{2/3}$:
\bee
&& \int_{-\frac{\delta}{\sft^{3/4}}}^{b} \frac{\Psi_1\left(p^*\left(\mX-\tb^2\right)\right)\left(1+O\left(\sft^{\frac{1}{12}}|\mX|^{\frac{1}{18}}+\sft^{\frac{1}{12}}|\tb|^{\frac{1}{9}}\right)\right)}{\left(1+3\Psi_1^2\left(p^*\left(\mX-\tb^2 \right) \right)\right)^3}d\tb \\
&=& \int_{-\infty}^{b} \frac{\Psi_1\left(p^*\left(\mX-\tb^2\right)\right)}{\left(1+3\Psi_1^2\left(p^*\left(\mX-\tb^2 \right) \right)\right)^3}d\tb+O\left(|\mX|^{-\frac 76+\frac{1}{18}}\sft^{\frac{1}{12}} \right).
\eee
Therefore:
\bea
\non &&\int_{-\frac{\delta}{\sft^{3/4}}}^{b_1} \frac{-\pa_{aa}\mX\left(\ta(\mX,\tb),\tb \right)}{\left(\pa_a \mX\left(\ta(\mX,\tb),\tb \right)\right)^3}d\tb+\int_{a_1}^{a} \frac{\pa_{bb}\mX \left(\ta,\tb(\mX,\ta) \right)}{\left(\pa_b \mX\left(\ta,\tb(\mX,\ta) \right)\right)^3}d\ta +\frac{1}{\pa_a \mX(a_1,b_1) \pa_b \mX(a_1,b_1)}\\
\lab{bd:jenaimarre}&=& \int_{-\infty}^{b} \frac{\Psi_1\left(p^*\left(\mX-\tb^2\right)\right)}{\left(1+3\Psi_1^2\left(p^*\left(\mX-\tb^2 \right) \right)\right)^3}d\tb+ \frac{1}{(1+3p^{*2}a^2)2b}+O\left( |\mX|^{-\frac 76+\frac{1}{18}}\sft^{\frac{1}{12}}\right) 
\eea
Since $(a,b)$ is in $Z_2$, from \fref{bd:step4mXa}, \fref{bd:Z2:pax} and \fref{bd:Z1:paamX}:
$$
\frac{\pa_a \mX(a,b)}{(1+3p^{*2}a^2)2b}= \frac{1}{\pa_b \mX(a,b)}+O(|\mX|^{-\frac 12+\frac 16}\sft^{\frac 14}).
$$
Injecting the two identities above, \fref{id:paamYGamma1step4} and $|\pa_a\mX|\lesssim |\mX|^{2/3}$ in the identity \fref{id:paamYstep4} gives:
$$
\sft^{\frac 14}k_6\pa_a \left(\int_{\Gamma_1} \frac{ds}{|\nabla x|}+\int_{\Gamma_2} \frac{ds}{|\nabla x|}\right)= 6 p^*\pa_a \mX(a,b) \int_{-\infty}^{b} \frac{\Psi_1\left(p^*\left(\mX-\tb^2\right)\right)d\tb}{\left(1+3\Psi_1^2\left(p^*\left(\mX-\tb^2 \right) \right)\right)^3}+O\left(|\mX|^{-\frac 12 +\frac{1}{18}}\sft^{\frac{1}{12}} \right)
$$
From this identity, \fref{id:mYstep4}, \fref{id:mYGamma0step4} and \fref{bd:sides:paamX} we have proved that:
$$
\pa_a \mY =\pa_a \mY^\Theta +O\left(|\mX|^{-\frac 12 +\frac{1}{18}}\sft^{\frac{1}{12}} \right),
$$
which is the second identity in \fref{bd:sides:mY} that we had to show. We now turn to the partial derivative with respect to $b$. From \fref{id:mYstep4} and \fref{id:derivativea1b1}, and then injecting \fref{bd:jenaimarre}, \fref{bd:sides:paamX}, \fref{bd:sides:pabmX} and \fref{bd:abmXautosimsides}:
\bee
&& \sft^{\frac 14}k_6 \pa_b \left(\int_{\Gamma_1} \frac{ds}{|\nabla x|}+\int_{\Gamma_2} \frac{ds}{|\nabla x|}\right)= \pa_b \left(\int_{-\frac{\delta}{\sft^{3/4}}}^{b_1(a,b)} \frac{d\tb}{\pa_a \mX(\ta(\mX,\tb),\tb)}+\int_{a_1(a,b)}^a -\frac{d\ta}{\pa_b \mX(\ta,\tb(\mX,\ta))} \right) \\
&=& \pa_b \mX(a,b)\left(\int_{-\frac{\delta}{\sft^{3/4}}}^{b_1} \frac{-\pa_{aa}\mX (\ta(\mX,\tb),\tb )d\tb}{\left(\pa_a \mX (\ta(\mX,\tb),\tb )\right)^3}+\int_{a_1}^{a} \frac{\pa_{bb}\mX (\ta,\tb(\mX,\ta))d\ta}{\left(\pa_b \mX (\ta,\tb(\mX,\ta))\right)^3} +\frac{1}{\pa_a \mX(a_1,b_1) \pa_b \mX(a_1,b_1)} \right)\\
&=& \pa_b \mX(a,b)\left(\int_{-\infty}^{b} \frac{\Psi_1\left(p^*\left(\mX-\tb^2\right)\right)}{\left(1+3\Psi_1^2\left(p^*\left(\mX-\tb^2 \right) \right)\right)^3}d\tb+ \frac{1}{(1+3p^{*2}a^2)2b}+O\left( |\mX|^{-\frac 76+\frac{1}{18}}\sft^{\frac{1}{12}}\right) \right)  \\
&=&\pa_b \mY^\Theta +O\left(|\mX|^{-\frac 23 +\frac{1}{18}}\sft^{\frac{1}{12}} \right)\\
\eee
which was the last estimate \fref{bd:sides:pabmY} we had to show. We claim that the computations we performed for this right side of the self-similar zone can be adapted in a straightforward way in the case where one has to consider more parts of the curve $\Gamma$ inside $Z_0^c$ to parametrise: the integral over $\Gamma_3$, $\Gamma_4$ and $\Gamma_5$ can be treated the very same way, leading to the same result.

\end{proof}

\end{appendix}

\bibliography{biblio}
\bibliographystyle{plain}

\end{document}